\documentclass[12pt,oneside,english]{amsart}
\usepackage[T1]{fontenc}
\usepackage[latin9]{inputenc}
\usepackage[a4paper]{geometry}
\usepackage{babel}
\usepackage{prettyref}
\usepackage{mathrsfs}
\usepackage{mathtools}
\usepackage{amstext}
\usepackage{amsthm}
\usepackage{amssymb}
\usepackage[unicode=true,
 bookmarks=true,bookmarksnumbered=false,bookmarksopen=false,
 breaklinks=false,pdfborder={0 0 1},backref=false,colorlinks=false]
 {hyperref}
\usepackage{breakurl}

\makeatletter
\numberwithin{equation}{section}
\numberwithin{figure}{section}
\theoremstyle{plain}
\newtheorem{thm}{\protect\theoremname}
  \theoremstyle{plain}
  \newtheorem{lem}[thm]{\protect\lemmaname}
  \theoremstyle{definition}
  \newtheorem{defn}[thm]{\protect\definitionname}
  \theoremstyle{plain}
  \newtheorem{prop}[thm]{\protect\propositionname}
  \theoremstyle{remark}
  \newtheorem{rem}[thm]{\protect\remarkname}

\makeatother

  \providecommand{\definitionname}{Definition}
  \providecommand{\lemmaname}{Lemma}
  \providecommand{\propositionname}{Proposition}
  \providecommand{\remarkname}{Remark}
\providecommand{\theoremname}{Theorem}

\begin{document}

\title[Szeg\H{o} kernel]{Bergman-Szeg\H{o} kernel asymptotics in weakly pseudoconvex finite
type cases }

\author{Chin-Yu Hsiao and Nikhil Savale}

\thanks{C.-Y. H. is partially supported by Taiwan Ministry of Science and
Technology projects 108-2115-M-001-012-MY5 and 109-2923-M-001-010-MY4 }

\thanks{N.S. is partially supported by the DFG funded project CRC/TRR 191.}

\address{Institute of Mathematics, Academia Sinica, 6F, Astronomy-Mathematics
Building, No. 1, Sec. 4, Roosevelt Road, Taipei 10617, TAIWAN}

\email{chsiao@math.sinica.edu.tw}

\address{Universität zu Köln, Mathematisches Institut, Weyertal 86-90, 50931
Köln, Germany}

\email{nsavale@math.uni-koeln.de}
\begin{abstract}
We construct a pointwise Boutet de Monvel-Sjöstrand parametrix for
the Szeg\H{o} kernel of a weakly pseudoconvex three dimensional CR
manifold of finite type assuming the range of its tangential CR operator
to be closed; thereby extending the earlier analysis of Christ \cite{Christ88,Christ89-embedding}.
This particularly extends Fefferman's boundary asymptotics of the
Bergman kernel \cite{Fefferman74} to weakly pseudoconvex domains
in $\mathbb{C}^{2}$, in agreement with D'Angelo's example \cite{Dangelo78}.
Finally our results generalize a three dimensional CR embedding theorem
of Lempert \cite{Lempert92}.
\end{abstract}

\maketitle

\section{Introduction}

Cauchy Riemann (CR) manifolds are natural analogues of complex manifolds
in odd dimensions. Their structure being modeled on that of a real-hypersurface
inside a complex manifold, the natural question of when an abstract
CR manifold can be embedded as such into complex space $\mathbb{C}^{N}$
has been long studied. In dimensions at least five a classical embedding
theorem for strongly pseudo-convex CR manifolds was proved by Boutet
de Monvel \cite{Boutet75}; thereby leaving unresolved the cases of
three dimensional manifolds and weakly pseudoconvex manifolds. In
dimension three the problem is well known to be more subtle as there
are examples of non-embeddable strongly pseudo-convex manifolds \cite{Rossi65,Andreotti-Siu70}.
However stronger conditions implying three dimensional embeddability
are known; Kohn \cite{Kohn85,Kohn86} showed that embeddability of
a strongly pseudoconvex CR manifold is equivalent its tangential Cauchy-Riemann
operator $\bar{\partial}_{b}$ having closed range. Thereafter Lempert
\cite{Lempert92} (see also Epstein \cite{Epstein1992}) showed embeddability
of a strongly pseudoconvex CR manifold assuming the existence of transversal
CR circle action. In the weakly pseudoconvex case fewer results are
known; Christ \cite{Christ88,Christ89-embedding} (see also Kohn \cite{Kohn85})
showed embeddability of a weakly pseudoconvex CR three manifold of
finite type assuming the range of its tangential Cauchy-Riemann operator
$\bar{\partial}_{b}$ to be closed. 

A closely related problem is to study the behavior of the Szeg\H{o}
kernel, the Schwartz kernel of the projector from smooth functions
onto CR functions. When the manifold is the boundary of a strictly
pseudoconvex domain the singularity of the Szeg\H{o} kernel was described
by Boutet de Monvel and Sjöstrand in \cite{Boutet-Sjostrand76}, the
Szeg\H{o} projector is in this case is a Fourier integral operator
with complex phase. Combined with the results of \cite{Boutet75,Harvey-Lawson1975,Kohn85,Kohn86}
this description extends to strongly pseudo-convex manifolds whose
tangential CR operator $\bar{\partial}_{b}$ has closed range; and
in particular those of dimension at least five. In particular this
description can be used to derive Fefferman's boundary asymptotics
of the Bergman kernel \cite{Fefferman74} of a strongly pseudoconvex
domain. The weakly pseudoconvex case analog of the problem has been
studied by several authors before, with prior results including pointwise
bounds on the kernels \cite{Christ88,Machedon88,McNeal89,Nagel-Rosay-Stein-Wainger-89}
besides special cases of the asymptotics \cite{Dangelo78,BoasStraubeYu95}.

In the present article we obtain a similar description for the pointwise
Szeg\H{o} kernel of weakly pseudoconvex CR three manifolds of finite
type whose tangential CR operator $\bar{\partial}_{b}$ has closed
range. Further we prove the boundary asymptotic expansion for the
Bergman kernel for weakly pseudoconvex finite type domains in $\mathbb{C}^{2}$.
Our results thereby extend the aforementioned analysis of Christ,
the boundary Bergman kernel asymptotics of Fefferman and the embedding
result of Lempert.

Let us now state our results more precisely. Let $\left(X,T^{1,0}X\right)$
be a compact CR manifold of dimension three. Thus $T^{1,0}X\subset T_{\mathbb{C}}X$
is a complex subbundle of dimension one satisfying $T^{1,0}X\cap T^{0,1}X=\emptyset$,
$T^{0,1}X\coloneqq\overline{T^{1,0}X}$. Denote by $HX\coloneqq\textrm{Re}\left(T^{1,0}X\oplus T^{0,1}X\right)$
the Levi-distribution and $J$ its induced integrable almost complex
structure. The Levi form is defined as
\begin{align}
\mathscr{L} & \in\left(HX^{*}\right)^{\otimes2}\otimes\left(T_{x}X/H_{x}X\right)\nonumber \\
\mathscr{L}\left(u,v\right) & \coloneqq\left[\left[u,v\right]\right]\in T_{x}X/H_{x}X\label{eq:Levi form}
\end{align}
for $u,v\in C^{\infty}\left(HX\right)$. Given a locally defined vector
field $T\in C^{\infty}\left(TX\right)$ transversal to $HX$ the Levi
form can be thought of as a skew-symmetric bi-linear form on $HX$.
We say that the point $x$ is weakly/strongly pseudoconvex iff the
corresponding bi-linear form $\mathscr{L}\left(.,J.\right)$ is positive
semi-definite/definite for some choice of orientation for $T$. The
manifold is weakly/strongly pseudoconvex if each point $x\in X$ is
weakly/strongly pseudoconvex. The CR manifold is said to be of finite
type if the Levi-distribution $HX$ is bracket generating: $C^{\infty}\left(HX\right)$
generates $C^{\infty}\left(TX\right)$ under the Lie bracket. More
precisely, the type of a point $x\in X$ is the smallest integer $r\left(x\right)$
such that $HX_{r\left(x\right)}=TX$, where $HX_{j}$, $j=1,\ldots$
are inductively defined by $HX_{1}\coloneqq HX$ and $HX_{j+1}\coloneqq HX+\left[HX_{j},HX\right],\,\forall j\geq1.$
The function $x\mapsto r\left(x\right)$ is in general only an upper
semi-continuous function. The finite type hypothesis is then equivalent
to $r\coloneqq\max_{x\in X}r\left(x\right)<\infty.$ Note that the
type of a strongly pseudoconvex point $x$ is $r\left(x\right)=2$.
For points of higher type it shall be useful to analogously define
the $r\left(x\right)-2$ jet of the Levi-form at $x$ 
\begin{align}
j^{r_{x}-2}\mathscr{L} & \in\left(HX^{*}\right)^{\otimes r_{x}}\otimes\left(T_{x}X/H_{x}X\right)\quad\textrm{ by}\nonumber \\
\left(j^{r_{x}-2}\mathscr{L}\right)\left(u_{1},\ldots,u_{r_{x}}\right) & \coloneqq\left[\textrm{ad}_{u_{1}}\textrm{ad}_{u_{2}}\ldots\textrm{ad}_{u_{r-1}}u_{r}\right]\in T_{x}X/H_{x}X\label{eq:first jet Levi form}
\end{align}
for $u_{j}\in C^{\infty}\left(HX\right)$, $j=1,\ldots,r$.

Next let $\bar{\partial}_{b}:\Omega^{0,*}\left(X\right)\rightarrow\Omega^{0,*+1}\left(X\right)$
denote the tangential CR operator and choose a smooth volume form
$\mu$ on $X$. The Szeg\H{o} kernel $\Pi\left(x,x'\right)$ is by
definition the Schwartz kernel of the $L^{2}$ projection $\Pi:L^{2}\left(X\right)\rightarrow\textrm{ker }\left(\bar{\partial}_{b}\right)$.
To describe our parametrix for $\Pi$, first recall the well known
symbol class $S_{\rho,\delta}^{m}\left(U\times\mathbb{R}_{t}\right)$,
$m\in\mathbb{R}$, $\rho,\delta\in\left(0,1\right]$, $U\subset\mathbb{R}^{2}$,
of Hörmander \cite{Hormander67-proc}: these are smooth functions
$a\left(x,t\right)$ satisfying the estimates $\partial_{t}^{k}\partial_{x}^{\alpha}a=O\left(t^{m-\rho k+\delta\left|\alpha\right|}\right)$,
$\forall\left(k,\alpha\right)\in\mathbb{N}_{0}\times\mathbb{N}_{0}^{3}$,
as $t\rightarrow\infty$, uniformly on compact subsets of $U$. Further
denote by the notation $S_{\delta}^{m}\left(U\times\mathbb{R}_{t}\right)$
the special case when $\rho=1$. We now introduce the subspace of
classical symbols $S_{\delta,\textrm{cl}}^{m}\left(U\times\mathbb{R}_{t}\right)\subset S_{\delta}^{m}\left(U\times\mathbb{R}_{t}\right)$
as those $a\left(x,t\right)$ for which there exist functions $a_{j}\in\mathcal{S}\left(\mathbb{R}^{2}\right)$,
$j=0,1,2,\ldots$, satisfying 
\begin{equation}
a\left(x,t\right)-t^{m}\left[\sum_{j=0}^{N}t^{-\delta j}a_{j}\left(t^{\delta}x\right)\right]\in S_{\delta}^{m-\delta N}\left(U\times\mathbb{R}_{t}\right),\quad\forall N\in\mathbb{N}_{0}.\label{eq:classical symbols}
\end{equation}
Our first theorem is now the following.
\begin{thm}
\label{thm:main thm parametrix} Let $X$ be a compact weakly pseudoconvex
three dimensional CR manifold of finite type for which the range of
the tangential CR operator $\bar{\partial}_{b}$ is closed. At any
point $x'\in X$ of type $r=r\left(x'\right)$, there exists a set
of coordinates $\left(x_{1},x_{2},x_{3}\right)$ centered at $x'$
and a classical symbol $a\in S_{\frac{1}{r},\textrm{cl}}^{\frac{2}{r}}\left(\mathbb{R}_{x_{1},x_{2}}^{2}\times\mathbb{R}_{t}\right)$,
with $a_{0}>0$, such that the pointwise Szeg\H{o} kernel at $x'$
satisfies
\begin{equation}
\Pi\left(x,x'\right)=\int_{0}^{\infty}dt\,e^{itx_{3}}a\left(x;t\right)+C^{\infty}\left(X\right).\label{eq:Szego parametrix}
\end{equation}
\end{thm}
We note again that the point $x'\in X$ above is fixed and thus our
'pointwise parametrix' is a distribution on the manifold $X$ rather
than the product. The direction $\partial_{x_{3}}$ is locally transverse
to the Levi distribution $HX$. More can be said about the amplitude
in \prettyref{eq:Szego parametrix}: each coefficient $a_{j}\in\mathcal{S}\left(\mathbb{R}^{2}\right)$
in its symbolic expansion \prettyref{eq:classical symbols} is a linear
combination of functions of the form $x_{1}^{\alpha_{1}}x_{2}^{\alpha_{2}}a_{j,\alpha}$,
$\alpha_{1}+\alpha_{2}\leq2jr$, with the functions $a_{j,\alpha}\in\mathcal{S}\left(\mathbb{R}^{2}\right)$
further depending only on the first jet of the Levi-form $j^{r_{x'}-2}\mathscr{L}$
at $x'$ and the indices $j,\alpha$. Furthermore at a strongly pseudoconvex
point $x'$ we may take each $a_{j,\alpha}=e^{-\left(x_{1}^{2}+x_{2}^{2}\right)}$
to be a Gaussian. Following this \prettyref{thm:main thm parametrix}
is seen to recover the pointwise version of the Boutet de Monvel-Sjöstrand
parametrix at strongly pseudoconvex points (see Remark \prettyref{rem:-(Strongly-pseudoconvex}
below). At points of higher type however the functions $a_{j,\alpha}$
are no longer Gaussians.

An important classical case arises when the CR manifold $X=\partial D$
is the boundary of a domain, i.e. a relatively compact open subset
$D\subset\mathbb{C}^{2}$. The analogous Bergman kernel $\Pi_{D}\left(z,z'\right)$
is the Schwartz kernel of the projector $\Pi_{D}:L^{2}\left(D\right)\rightarrow\textrm{ker }\left(\bar{\partial}\right)$
onto the $L^{2}$-holomorphic functions in the interior. One is then
interested in the on-diagonal behavior of the Bergman kernel as one
approaches the boundary in terms of a boundary defining function $\rho\in C^{\infty}\left(\mathbb{C}^{2}\right)$,
satisfying $D=\left\{ \rho<0\right\} $, $\left.d\rho\right|_{\partial D}\neq0$.
This is given as below.
\begin{thm}
\label{thm:Fefferman thm.} Let $D\subset\mathbb{C}^{2}$ be a domain
with boundary $X=\partial D$ being smooth, weakly pseudoconvex of
finite type. For any point $x'\in X=\partial D$ on the boundary,
of type $r=r\left(x'\right)$, the Bergman kernel satisfies the asymptotics
\[
\Pi_{D}\left(z,z\right)=\sum_{j=0}^{N}\frac{1}{\left(-\rho\right){}^{2+\frac{2}{r}-\frac{1}{r}j}}a_{j}+\sum_{j=0}^{N}b_{j}\left(-\rho\right){}^{j}\log\left(-\rho\right)+O\left(\left(-\rho\right)^{\frac{N-2-2r}{r}}\right),\quad\forall N\in\mathbb{N},
\]
as $z\rightarrow x'$ for some set of reals $a_{j},b_{j}$ with $a_{0}>0$. 
\end{thm}
Our description of Szeg\H{o} kernel \prettyref{thm:main thm parametrix}
becomes more concrete in the case when the CR manifold $X$ is circle
invariant. In this case one obtains an on diagonal expansion $\Pi_{m}\left(x,x\right)$,
$m\rightarrow\infty$, for the $m$th Fourier mode of the Szeg\H{o}
kernel, we refer to \prettyref{thm:Szego kernel expansion theorem}
in \prettyref{sec:S1 invariant CR geometry} below for the precise
statement. The asymptotics of these higher Fourier modes of the Szeg\H{o}
kernel allows one to construct a sufficient number of CR peak functions.
These can be used to prove the following embedding theorem.
\begin{thm}
\label{thm: main embedding thm}Let $X$ be a compact weakly pseudoconvex
three dimensional CR manifold of finite type admitting a transversal,
CR circle action. Then it has an equivariant CR embedding into some
$\mathbb{C}^{N}$ , $N\in\mathbb{N}$. 
\end{thm}
The Szeg\H{o} kernel parametrix of Boutet de Monvel-Sjöstrand \cite{Boutet-Sjostrand76}
has had a broad impact in complex analysis and geometry, we refer
to \cite{Hsiao2010} for a detailed account of this technique and
its applications. Particularly this recovered the prior results of
Fefferman \cite{Fefferman74} on the full boundary asymptotics for
the Bergman kernel of a strongly pseudoconvex domain, which in turn
refined its leading asymptotics by Hörmander \cite{Hormander65-L2est}.
The weakly pseudoconvex analog of the problem has also been considered
by several authors before. Prior results have included pointwise upper
\cite{Christ88,Machedon88,McNeal89,Nagel-Rosay-Stein-Wainger-89}
and lower \cite{Catlin89} bounds on the Bergman and Szeg\H{o} kernels
in low dimensions, besides particular special cases of the asymptotics
for complex ovals \cite{Dangelo78}, h-extendible/semiregular domains
\cite{BoasStraubeYu95} and certain toric domains \cite{Kamimoto2004}.
In higher dimensional weakly pseudoconvex cases the analogous bounds
as well as the asymptotics of \prettyref{thm:main thm parametrix}
and \prettyref{thm:Fefferman thm.} are wide open, some known results
in higher dimensions include weak estimates on the Bergman kernel
\cite{Ohsawa84} along with estimates on the Bergman metric \cite{Mcneal92}
and distance \cite{Diederich-Ohsawa-95}.

In the presence of a transversal circle action a weakly pseudoconvex
CR manifold is the unit circle bundle of a semi-positive holomorphic
orbifold line bundle over a complex orbifold. When the action is free
and the manifold strongly pseudoconvex, the Szeg\H{o} kernel expansion
\prettyref{thm:Szego kernel expansion theorem}corresponds to the
Bergman kernel expansion of positive line bundles and was first obtained
in \cite{Catlin97-Bergmankernel,Zelditch98-Bergmankernel}. This was
recently generalized to the Bergman kernel expansion of semipositive
line bundles over a Riemann surface in \cite{Marinescu-Savale18}.
The first author had earlier in \cite{Hsiao-Marinescu-JDG2017} given
a proof of the Szeg\H{o} kernel expansion of a circle invariant weakly
pseudoconvex CR manifold, although only on its strictly pseudoconvex
part. For general non-free actions on strongly-pseudoconvex manifolds,
the Szeg\H{o} kernel expansion corresponds to the Bergman kernel expansion
of positive orbifold line bundles and was first proved in \cite{Dai-Liu-Ma2006-Bergman-kernel},
\cite[Sec. 5.4]{Ma-Marinescu}. 

As mentioned, the embeddability of strongly pseudoconvex CR manifolds
equipped with a transversal CR circle action was shown in \cite{Lempert92}
and thus generalized by our last \prettyref{thm: main embedding thm}.
In the weakly pseudoconvex case, \cite{Christ89-embedding} showed
embeddability of a CR three manifold of finite type assuming the range
of its tangential Cauchy-Riemann operator $\bar{\partial}_{b}$ to
be closed. 

The paper is organized as follows. In \prettyref{sec:CR-Preliminaries}
we begin with some preliminaries in CR geometry including a construction
of almost analytic coordinates adapted to the CR structure in \prettyref{subsec:Construction-of-coordinates}.
In \prettyref{sec:Szego-parametrix} we construct an appropriate symbol
calculus in \prettyref{subsec:Symbol-spaces-and} and construct the
pointwise Szeg\H{o} parametrix to prove \prettyref{thm:main thm parametrix}.
In \prettyref{sec:Pseudoconvex-domains} we consider the Bergman kernel
of a weakly pseudoconvex domain in $\mathbb{C}^{2}$ and prove \prettyref{thm:Fefferman thm.}.
In \prettyref{sec:S1 invariant CR geometry} we turn to the circle
invariant case and prove \prettyref{thm:Szego kernel expansion theorem}.
In the final section \prettyref{sec:Equivariant-CR-Embedding} we
prove our embedding theorem \prettyref{thm: main embedding thm}. 

\section{\label{sec:CR-Preliminaries}CR Preliminaries}

Let $\left(X,T^{1,0}X\right)$ be a compact CR manifold of dimension
three. Thus $T^{1,0}X\subset T_{\mathbb{C}}X$ is a complex subbundle
of dimension one satisfying $T^{1,0}X\cap T^{0,1}X=\emptyset$, $T^{0,1}X\coloneqq\overline{T^{1,0}X}$.
Let $HX\coloneqq\textrm{Re}\left(T^{1,0}X\oplus T^{0,1}X\right)$
be the Levi-distribution. This carries an almost complex structure
$J:HX\rightarrow HX$
\[
J\left(v+\bar{v}\right)\coloneqq i\left(v-\bar{v}\right),\quad\forall v\in T^{1,0}X,
\]
 satisfying $J^{2}=-1$ and the integrability condition 
\begin{align}
\left[Jv,u\right]+\left[v,Ju\right] & \in C^{\infty}\left(HX\right)\nonumber \\
\left[Jv,Ju\right]-\left[v,u\right] & =J\left(\left[Jv,u\right]+\left[v,Ju\right]\right)\label{eq: NN integrability}
\end{align}
$\forall u,v\in C^{\infty}\left(HX\right)$. The antisymmetric Levi-form
defined via \prettyref{eq:Levi form} consequently satisfies $\mathscr{L}\left(Ju,v\right)=\mathscr{L}\left(u,Jv\right)$.
The last contraction $\mathscr{L}\left(.,J.\right)$ is equivalently
thought of as a Hermitian form on $T^{1,0}X$ and denoted by the same
notation via
\begin{equation}
\mathscr{L}\left(u,v\right)\coloneqq\left[\frac{2}{i}\left[u,\bar{v}\right]\right]\in\left(TX/HX\right)\otimes\mathbb{C},\quad\label{eq:Levi form-1}
\end{equation}
$u,v\in T^{1,0}X.$ The point $x\in X$ is strongly/weakly pseudoconvex
if the Levi form above is positive definite/semi-definite at $x$
for some choice of local orientation for $TX/HX$. Next one defines
the flag of subspaces$HM_{1,x}\subset HM_{2,x}\subset\ldots$, at
$x\in X$ inductively via 
\begin{align*}
HM_{1,x} & \coloneqq HM_{x}\\
HM_{j+1,x} & \coloneqq HM_{x}+\left[HM_{j,x},HM_{x}\right],\quad j\geq1.
\end{align*}
The point $x$ is said to be of finite type if $HM_{r\left(x\right),x}=TX$
for some $r\left(x\right)\in\mathbb{N}$, the minimum such integer
being the type of the point $x$. The weight vector at the point $x$
is defined to be $w\left(x\right)\coloneqq\left(1,1,r\left(x\right)\right)$.
The CR structure is of finite type if each point is of finite type.
Note that the type of a strongly pseudoconvex point $x\in X$ is $r\left(x\right)=2$
by definition.

The set of \textit{horizontal} paths of Sobolev regularity one connecting
the two points $x,x'\in X$ is denoted by
\begin{equation}
\Omega_{H}\left(x,x'\right)\coloneqq\left\{ \gamma\in H^{1}\left(\left[0,1\right];X\right)|\gamma\left(0\right)=x,\,\gamma\left(1\right)=x',\,\dot{\gamma}\left(t\right)\in HX_{\gamma\left(t\right)}\textrm{ a.e.}\right\} .\label{eq:horiz path space}
\end{equation}
Fixing a metric $g^{HX}$ on the Levi distribution $HX$, one define
the obvious the obvious length functional $l\left(\gamma\right)\coloneqq\int_{0}^{1}\left|\dot{\gamma}\right|dt$
on the above path space $\Omega_{H}$. By a classical theorem of Chow-Rashevsky
the horizontal path space \prettyref{eq:horiz path space} is non-empty
when the CR structure is of finite type; allowing the definition of
a distance function
\begin{equation}
d^{H}\left(x,x'\right)\coloneqq\inf_{\gamma\in\Omega_{H}\left(x,x'\right)}l\left(\gamma\right).\label{eq:sR distance}
\end{equation}

The weight $w\left(f\right)$ of a function $f$ at the point $x$
is defined to be the maximum integer $s\in\mathbb{N}_{0}$ for which
$a+b=s$ implies that $\left(u^{a}v^{b}f\right)\left(x\right)=0$;
where $u,v$ form a local frame for $HX$ near $x$. Similarly the
weight $w\left(P\right)$ of a differential operator $P$ at the point
$x\in X$ is the maximum integer for which $w\left(Pf\right)\geq w\left(P\right)+w\left(f\right)$
holds for each function $f\in C^{\infty}\left(X\right)$. It is known
that there exists a set of coordinates $\left(x_{1},x_{2},x_{3}\right)$
centered near a point $x\in X$ for which $\frac{\partial}{\partial x_{1}},\frac{\partial}{\partial x_{2}}$
forms a basis for $HM_{x}$ of the canonical flag and moreover each
coordinate function $\left(x_{1},x_{2},x_{3}\right)$ has weight $w\left(x\right)\coloneqq\left(1,1,r\left(x\right)\right)$
respectively; such a coordinate system is called privileged. In privileged
coordinates near $x$ the weight of a monomial $x^{\alpha}$, $\alpha\in\mathbb{N}_{0}$,
is thus $\alpha_{1}+\alpha_{2}+r\alpha_{3}$. The weight $w\left(f\right)$
of an arbitrary function $f\in C^{\infty}\left(X\right)$ is then
the minimum weight of the monomials appearing in its Taylor series
in these coordinates. The weight $w\left(V\right)$ of a vector field
$V=\sum_{j=1}^{3}f_{j}\partial_{x_{j}}$ is seen to be $w\left(V\right)\coloneqq\min\left\{ w\left(f_{1}\right)-1,w\left(f_{2}\right)-1,w\left(f_{3}\right)-r\right\} $.
The distance function and volume (with respect to an arbitrary volume
form $\mu$) of a radius $\varepsilon$ ball centered at the origin
in such coordinates are known to satisfy 
\begin{align*}
C\left(\left|x_{1}\right|+\left|x_{2}\right|+\left|x_{3}\right|^{1/r\left(x\right)}\right) & \leq d^{H}\left(x,0\right)\leq C'\left(\left|x_{1}\right|+\left|x_{2}\right|+\left|x_{3}\right|^{1/r\left(x\right)}\right)\\
C\epsilon^{2+r\left(x\right)} & \leq\mu\left(B\left(0;\varepsilon\right)\right)\leq C'\epsilon^{2+r\left(x\right)};\quad\varepsilon\in\left(0,1\right).
\end{align*}
Given a coordinate chart $U$ as above, denote by $S_{H}^{m}\left(U\right)$
the space of smooth functions $p\left(x\right)$ on $U\setminus\left\{ 0\right\} $
satisfying the estimates 
\begin{equation}
\left|\partial_{x}^{\alpha}p\right|\leq C_{\alpha}\left[d^{H}\left(x,0\right)\right]^{-m+\frac{2}{r}-\alpha.w_{x}}\mu\left(B\left(0;d^{H}\left(x,0\right)\right)\right)^{-1},\label{eq: function class}
\end{equation}
$x\neq0$,$\forall\alpha\in\mathbb{N}_{0}^{3}.$ It was shown in \cite[Sec.12]{Christ88},
cf. \cite{Nagel-Rosay-Stein-Wainger-89} that the restriction to $U$
of the Szeg\H{o} kernel lies in the class
\begin{equation}
\Pi\left(x,0\right)\in S_{H}^{\frac{2}{r}}\left(U\right)\label{eq:christ-M-M-NRSW-bounds}
\end{equation}
defined above.

\subsection{\label{subsec:Construction-of-coordinates}Construction of coordinates}

In \cite[Prop. 3.2]{Christ89-embedding} it was shown that the privileged
coordinate system near $x$ maybe further chosen so that
\begin{align}
T^{1,0}X & =\mathbb{C}\left[Z\right]\nonumber \\
Z & =\frac{1}{2}\left[\partial_{x_{1}}+\left(\partial_{x_{2}}p\right)\partial_{x_{3}}+i\left(\partial_{x_{2}}-\left(\partial_{x_{1}}p\right)\partial_{x_{3}}+R\right)\right];\label{eq:Christ normal form}
\end{align}
 where $p\left(x_{1},x_{2}\right)$ is a homogeneous real polynomial
of degree/weight $r\left(x\right)$, and $R=\sum_{j=1}^{3}r_{j}\left(x\right)\partial_{x_{j}}$
a real vector field of weight $w\left(R\right)\geq0$. Furthermore,
the pseudoconvexity of $X$ gives $\Delta p=\left(\partial_{x_{1}}^{2}+\partial_{x_{2}}^{2}\right)p\geq0$.
In this subsection we shall further show how to remove the remainder
term $R$ via almost analytic extension. We first have the following.
\begin{lem}
\label{lem: TZ is o infinity}There exists a locally defined complex
vector field $T$ such that $T_{x}=\partial_{x_{3}}$ and $\left[T,Z\right]$
vanishes to infinite order at $x$.
\end{lem}
\begin{proof}
The desired equation for the components of $T=\sum_{j=1}^{3}t_{j}\left(x\right)\partial_{x_{j}}$
is seen to be one of the form 
\[
\left(\partial_{x_{1}}+i\partial_{x_{2}}\right)t_{j}=-\left(\partial_{x_{2}}p-i\partial_{x_{1}}p\right)\partial_{x_{3}}t_{j}+\delta_{03}T\left(\partial_{x_{2}}-i\partial_{x_{1}}\right)p+Tr_{j}-Rt_{j}+O\left(\left|x\right|^{\infty}\right),
\]
$j=1,2,3.$ As the component functions $p$, $r_{j}$ have degree
at least two and one respectively; the degree $k$ homogeneous part
on the left hand side above involves the Taylor coefficients of $t_{j}$
for $x^{\alpha}$, $\left|\alpha\right|=k-1$, while those on the
right hand side involve those for $x^{\alpha},\left|\alpha\right|<k-1$.
We may hence solve the above recursively for the Taylor coefficients
of $t_{j}$, beginning with $\left(t_{1},t_{2},t_{3}\right)=\left(0,0,1\right)+O\left(\left|x\right|\right)$,
and apply Borel's construction.
\end{proof}
Next we complexify the open neighborhood of $x\in U\subset\mathbb{R}^{3}$
on which the above coordinates are defined to an open set $U^{\mathbb{C}}\subset\mathbb{C}^{3}$
such that $U^{\mathbb{C}}\cap\mathbb{R}^{3}=U$. Denote by $z_{j}=x_{j}+iy_{j}$,
$j=1,2,3$, the corresponding complex coordinates. For a function
$f\in C_{c}^{\infty}\left(\mathbb{C}^{3}\right)$, we write $f\sim0$
if it vanishes to infinite order along the real plane: $\left|f\left(z\right)\right|=O\left(\left|\textrm{Im}z\right|^{\infty}\right)$.
A function $f\in C_{c}^{\infty}\left(\mathbb{C}^{3}\right)$, is said
to be almost analytic iff $\partial_{\bar{z}_{j}}f\sim0$, $j=1,2,3.$
A complex vector field $L=\sum_{j=1}^{3}\left[a_{j}\partial_{z_{j}}+b_{j}\partial_{\bar{z}_{j}}\right]$
is said to be almost analytic iff $Lf$ is almost analytic and $L\bar{f}\sim0$
for all almost analytic $f\in C_{c}^{\infty}\left(\mathbb{C}^{3}\right)$.
This is seen to be equivalent to $a_{j}$ being almost analytic and
$b_{j}\sim0$ for $j=1,2,3$. For two complex vector fields we write
$L_{1}\sim L_{2}$ iff $L_{1}-L_{2}$ is almost analytic. We choose
almost analytic extensions $\tilde{T},\tilde{Z}$ of $T,Z$ respectively.
We now have the next lemma.
\begin{lem}
\label{lem: linearize T tilde} There exist almost analytic complex
coordinates $w_{j}=z_{j}+z_{3}O\left(\left|z\right|\right)$, $j=1,2,3$,
on $U^{\mathbb{C}}$ such that $\tilde{T}-\partial_{w_{3}}$ vanishes
to infinite order at $x$.
\end{lem}
\begin{proof}
Firstly we have $\tilde{T}\sim\sum_{j=1}^{3}t_{j}\left(z\right)\partial_{z_{j}}$
for some almost analytic functions $t_{j}$, $j=1,2,3$, satisfying
$\left(t_{1},t_{2},t_{3}\right)=\left(0,0,1\right)+O\left(\left|z\right|\right)$.
Next we find an almost analytic function $w_{1}\left(z\right)$ satisfying
$\sum_{j=1}^{3}t_{j}\partial_{z_{j}}w_{1}=O\left(\left|z\right|^{\infty}\right)$
or equivalently
\begin{equation}
\partial_{z_{3}}w_{1}=-\left[t_{1}\partial_{z_{1}}w_{1}+t_{2}\partial_{z_{2}}w_{1}+\left(t_{3}-1\right)\partial_{z_{3}}w_{1}\right]+O\left(\left|z\right|^{\infty}\right).\label{eq: const coords}
\end{equation}
The degree $k$ homogeneous part on the left hand side above involves
the Taylor coefficients of $w_{1}$ for $z^{\alpha}$, $\left|\alpha\right|=k$,
while those on the right hand side involve those for $z^{\alpha},\left|\alpha\right|<k$.
We may hence again solve the above recursively for the Taylor coefficients
of $w_{1}$, beginning with $w_{1}\left(z\right)=z_{1}+O\left(\left|z\right|^{2}\right)$,
and apply Borel's construction. Since solving the equation \prettyref{eq: const coords}
involves integration in $z_{3}$, the higher order terms in the Taylor
expansion $w_{1}\left(z\right)=z_{1}+z_{3}O\left(\left|z\right|\right)$
can be further taken to be multiples of $z_{3}$. In similar vein,
we find almost an analytic functions $w_{2}\left(z\right)=z_{2}+z_{3}O\left(\left|z\right|\right)$,
$w_{3}\left(z\right)=z_{3}\left(1+O\left(\left|z\right|\right)\right)$
satisfying $\sum_{j=1}^{3}t_{j}\partial_{z_{j}}w_{2}=O\left(\left|z\right|^{\infty}\right)$
and $\sum_{j=1}^{3}t_{j}\partial_{z_{j}}w_{3}=1+O\left(\left|z\right|^{\infty}\right)$
respectively. Thus $\left(w_{1},w_{2},w_{3}\right)$ is the required
coordinate system.
\end{proof}
We may now prove our main result of this subsection.
\begin{thm}
\label{thm:almost analytic coordinates} There exist almost analytic
complex coordinates $\tilde{z}_{j}=p_{j}+iq_{j}$, $j=1,2,3$ and
an almost analytic function $\varphi\left(\tilde{z}_{1},\tilde{z}_{2}\right)$
(of the first two new coordinates) on $U^{\mathbb{C}}$ such that 
\begin{enumerate}
\item $\tilde{z}_{j}=z_{j}+O\left(\left|z\right|^{2}\right)$, $j=1,2$,
$\tilde{z}_{3}=z_{3}+z_{3}O\left(\left|z\right|\right)+O\left(\left|z\right|^{\infty}\right)$,
\item $\tilde{Z}=\frac{1}{2}\left(\partial_{\tilde{z}_{1}}+i\partial_{\tilde{z}_{2}}\right)-\frac{i}{2}\left(\partial_{\tilde{z}_{1}}\varphi+i\partial_{\tilde{z}_{2}}\varphi\right)\partial_{\tilde{z}_{3}}$
and $\tilde{T}=\partial_{\tilde{z}_{3}}+O\left(\left|z\right|^{\infty}\right)$
for
\item $\varphi\left(\tilde{z}_{1},\tilde{z}_{2}\right)=\varphi_{0}\left(\tilde{z}_{1},\tilde{z}_{2}\right)+O\left(\left|z\right|^{r+1}\right)$
with $\varphi_{0}$ a homogeneous polynomial with real coefficients
satisfying $\left(\partial_{p_{1}}^{2}+\partial_{p_{2}}^{2}\right)\left(\left.\varphi_{0}\right|_{q=0}\right)\geq0$.
\end{enumerate}
\end{thm}
\begin{proof}
Firstly we have by definition $\tilde{Z}\sim\sum_{j=1}^{3}a_{j}\partial_{w_{j}}$,
for some almost analytic functions $a_{j}$, $j=1,2,3$ satisfying
\begin{align*}
a_{1}\left(0\right)=\frac{1}{2}, & \quad a_{2}\left(0\right)=\frac{i}{2}\quad,\\
a_{3}= & \frac{1}{2}\left(\partial_{w_{2}}\tilde{p}-i\partial_{w_{1}}\tilde{p}\right)+O\left(\left|w\right|^{r}\right)
\end{align*}
 and $\tilde{p}$ being an almost analytic extension of $p$. Furthermore,
from the preceding \prettyref{lem: TZ is o infinity}, \prettyref{lem: linearize T tilde}
we have $\left[\tilde{T},\tilde{Z}\right]=\left[\partial_{w_{3}},\tilde{Z}\right]+O\left(\left|w\right|^{\infty}\right)$
and may assume that $a_{j}$, $j=1,2,3$, are independent of $w_{3}$.
Next as in \prettyref{eq: const coords} we find almost analytic functions
$\tilde{w}_{j}\left(w_{1},w_{2}\right)=w_{j}+O\left(\left|w\right|^{\infty}\right)$,
$j=1,2$, such that 
\begin{align*}
a_{1}\partial_{w_{1}}\tilde{w}_{1}+a_{2}\partial_{w_{2}}\tilde{w}_{1}-\frac{1}{2} & =O\left(\left|w\right|^{\infty}\right)\\
a_{1}\partial_{w_{1}}\tilde{w}_{2}+a_{2}\partial_{w_{2}}\tilde{w}_{2}-\frac{i}{2} & =O\left(\left|w\right|^{\infty}\right).
\end{align*}
Setting $\tilde{z}_{3}=w_{3}$, we have then thus far achieved $\tilde{Z}=\frac{1}{2}\left(\partial_{\tilde{w}_{1}}+i\partial_{\tilde{w}_{2}}\right)+a_{3}\left(\tilde{w}_{1},\tilde{w}_{2}\right)\partial_{\tilde{w}_{3}}+O\left(\left|\tilde{w}\right|^{\infty}\right)$.
It is then easy to find $\varphi\left(\tilde{w}_{1},\tilde{w}_{2}\right)$
satisfying $a_{3}=-\frac{i}{2}\left(\partial_{\tilde{w}_{1}}\varphi+i\partial_{\tilde{w}_{2}}\varphi\right)+O\left(\left|\tilde{w}\right|^{\infty}\right)$
by a further application of the Borel construction giving $\tilde{Z}=\frac{1}{2}\left(\partial_{\tilde{w}_{1}}+i\partial_{\tilde{w}_{2}}\right)-\frac{i}{2}\left(\partial_{\tilde{w}_{1}}\varphi+i\partial_{\tilde{w}_{2}}\varphi\right)\partial_{\tilde{w}_{3}}+\tilde{Z}_{\infty}$
for some almost analytic vector field $\tilde{Z}_{\infty}=O\left(\left|\tilde{w}\right|^{\infty}\right)$. 

Finally to remove this infinite order error term one applies the scattering
trick of Nelson \cite[Ch. 3]{Nelson-book69}. Choose an almost analytic
function $\chi\in C_{c}^{\infty}\left(U^{\mathbb{C}}\right)$, equal
to one near zero, and set 
\begin{equation}
\tilde{Z}_{1}=\frac{1}{2}\left(\partial_{\tilde{w}_{1}}+i\partial_{\tilde{w}_{2}}\right)-\frac{i}{2}\left(\partial_{\tilde{w}_{1}}\varphi+i\partial_{\tilde{w}_{2}}\varphi\right)\partial_{\tilde{w}_{3}}+\left(1-\chi\right)\tilde{Z}_{\infty}.\label{eq:modified vector field}
\end{equation}
It is clear that the almost analytic flows of $\tilde{Z}$, $\tilde{Z}_{1}$
starting at $U^{\mathbb{C}}$ exit $U^{\mathbb{C}}$ in uniformly
finite time, outside which they are equal. Thus the limiting almost
analytic map 
\begin{equation}
W\coloneqq\lim_{t\rightarrow\infty}e^{t\tilde{Z}}\circ e^{-t\tilde{Z}_{1}}\label{eq:scattering operator}
\end{equation}
exists with the limit achieved in finite time. One then calculates
\begin{align}
\frac{d}{dt}\left(e^{-t\tilde{Z}_{1}}\circ e^{t\tilde{Z}}\right)^{*}\tilde{w}_{j} & =\left(e^{t\tilde{Z}}\right)^{*}\left(\tilde{Z}-\tilde{Z}_{1}\right)\left(e^{-t\tilde{Z}_{1}}\right)^{*}\tilde{w}_{j}=O\left(\left|\tilde{w}\right|^{\infty}\right)\nonumber \\
\textrm{and thus }\quad\tilde{z}_{j} & \coloneqq W^{*}\tilde{w}_{j}=\tilde{w}_{j}+O\left(\left|\tilde{w}\right|^{\infty}\right).\label{eq:final coordinates}
\end{align}
This finally gives 
\begin{align*}
\tilde{Z} & =W_{*}\tilde{Z}_{1}\\
 & =\frac{1}{2}\left(\partial_{\tilde{z}_{1}}+i\partial_{\tilde{z}_{2}}\right)-\frac{i}{2}\left(\partial_{\tilde{z}_{1}}\varphi+i\partial_{\tilde{z}_{2}}\varphi\right)\partial_{\tilde{z}_{3}}
\end{align*}
near zero from \prettyref{eq:modified vector field}, \prettyref{eq:scattering operator},
\prettyref{eq:final coordinates} proving the second part of the theorem.
The last part follows by a Taylor expansion and the corresponding
subharmonicity of $p\left(x_{1},x_{2}\right)$ \prettyref{eq:Christ normal form}.
\end{proof}
We remark that although Nelson's method may also be used to linearize
the almost analytic vector field $\tilde{Z}$ in some almost analytic
coordinates, the resulting coordinates thereby will not satisfy the
first property of the above \prettyref{thm:almost analytic coordinates}
which shall be used later.

Before putting the above coordinates to use in the next section, we
shall also need the construction of almost analytic continuations
of functions in the class $S_{H}^{m}$ defined in \prettyref{eq: function class}.
These shall be defined on the region
\begin{equation}
R_{\delta,U}\coloneqq\left\{ \left(x,y\right)\in U\times\mathbb{R}^{3}|\left|y\right|\leq\delta\left|x\right|^{2},\left|y_{3}\right|\leq\delta\left|x_{3}\right|\left|x\right|\right\} \label{eq:defined region}
\end{equation}
 for any $\delta>0$.
\begin{lem}
\label{lem:Almost-analytic-continuations}(Almost analytic continuations
in $S_{H}^{m}$) For each $\delta>0$ there exists almost analytic
extension map $\mathcal{E}:S_{H}^{m}\left(U\right)\rightarrow C^{\infty}\left(R_{\delta,U}\setminus\left\{ 0\right\} \right)$
satisfying 
\begin{align}
\left.\mathcal{E}f\right|_{\mathbb{R}} & =f,\nonumber \\
\partial_{\bar{z}}\mathcal{E}f & =O\left(\left|\textrm{Im}z\right|^{\infty}\right)\textrm{ uniformly on }R_{\delta,U}\textrm{ and }\nonumber \\
Lf\in C^{\infty}\left(U\right) & \implies\tilde{L}\mathcal{E}f\in C^{\infty}\left(R_{\delta,U}\right)\label{eq:properties of almost analytic extension}
\end{align}
for each $f\in S_{H}^{m}$ and vector field $L$ with almost analytic
extension $\tilde{L}$.
\end{lem}
\begin{proof}
The map $\mathcal{E}$ is defined by the usual Borel-Hörmander construction.
Namely with $\chi\in C_{c}^{\infty}\left(\mathbb{R}\right)$ and equal
to one near zero, set 
\begin{equation}
\left(\mathcal{E}f\right)\left(x,y\right)\coloneqq\sum_{\alpha}\frac{\left(iy\right)^{\alpha}}{\alpha!}f^{\left(\alpha\right)}\left(x\right)\chi\left(\lambda_{\left|\alpha\right|}\left|y\right|\right).\label{eq:almost analytic continuattion}
\end{equation}
Note that on the given region $R_{\delta,U}$ \prettyref{eq:defined region}
each successive term above satisfies the estimate $\frac{\left(iy\right)^{\alpha}}{\alpha!}f^{\left(\alpha\right)}\left(x\right)=O\left(\left|y\right|^{\left(\left|\alpha\right|-rm\right)/2}\right)$.
For a suitable sequence constants $\lambda_{k}\rightarrow\infty$
sufficiently fast, the series above is then seen to be $C^{\infty}$
convergent, and hence defining a smooth function, on compact subsets
of $R_{\delta,U}\setminus\left\{ 0\right\} $. The first property
in \prettyref{eq:properties of almost analytic extension} then follows
immediately from the above definition. The second property, follows
easily on differentiating the definition \prettyref{eq:almost analytic continuattion}
and applying the estimates \prettyref{eq: function class} on the
region $R_{\delta,U}$. Finally for the last property, note that $g\coloneqq\tilde{L}\mathcal{E}f-\mathcal{E}Lf$
is an almost analytic continuation of zero in the sense $\left.g\right|_{\mathbb{R}}=0$
and $\partial_{\bar{z}}g=O\left(\left|\textrm{Im}z\right|^{\infty}\right)\textrm{ uniformly on }R_{\delta,U}$.
It furthermore satisfies estimates similar to \prettyref{eq: function class}
by definition. The Taylor expansion of $g$ in the $y$-variable is
seen to be 
\[
g\left(x,y\right)\coloneqq\sum_{\left|\alpha\right|\leq N}\frac{\left(iy\right)^{\alpha}}{\alpha!}\underbrace{g^{\left(\alpha\right)}\left(x\right)}_{=0}+O\left(\left|y\right|^{\left(N+1-rm\right)/2}\right),
\]
giving $g=O\left(\left|\textrm{Im}z\right|^{\infty}\right)$ and $g\in C^{\infty}\left(R_{\delta,U}\right)$.
Since $Lf\in C^{\infty}\left(U\right)$ is smooth and $\mathcal{E}Lf\in C^{\infty}\left(R_{\delta,U}\right)$
by construction, the result follows. 
\end{proof}

\section{\label{sec:Szego-parametrix}Szeg\H{o} parametrix}

In this section we shall prove our main \prettyref{thm:main thm parametrix}.
It shall first be useful to define a requisite symbol calculus below.

\subsection{\label{subsec:Symbol-spaces-and} Symbol spaces and calculus}

Below we denote by $x=\left(\hat{x},x_{3}\right)$ local coordinates
on $\mathbb{R}^{3}$, with $\hat{x}=\left(x_{1},x_{2}\right)$ denoting
the local coordinates on $\mathbb{R}^{2}$. 

A smooth function $f\left(\hat{x},\hat{y}\right)\in C^{\infty}\left(\mathbb{R}^{2}\times\mathbb{R}^{2}\right)$
is said to lie in the class $f\in\hat{S}\left(\mathbb{R}^{2}\times\mathbb{R}^{2}\right)$
if for each $\left(\hat{\alpha},\hat{\beta}\right)\in\mathbb{N}_{0}^{4}$
there exists $N\left(\hat{\alpha},\hat{\beta}\right)\in\mathbb{N}$
such that
\begin{equation}
\left|\partial_{\hat{x}}^{\hat{\alpha}}\partial_{\hat{y}}^{\hat{\beta}}f\left(\hat{x},\hat{y}\right)\right|\leq C_{N,\hat{\alpha}\hat{\beta}}\frac{\left(1+\left|\hat{x}\right|+\left|\hat{y}\right|\right)^{N\left(\hat{\alpha},\hat{\beta}\right)}}{\left(1+\left|\hat{x}-\hat{y}\right|\right)^{-N}},\label{eq:diagonal schwartz class}
\end{equation}
 $\forall\left(\hat{x},\hat{y},N\right)\in\mathbb{R}^{2}\times\mathbb{R}^{2}\times\mathbb{N}.$
We note that for $f\in\hat{S}\left(\mathbb{R}^{2}\times\mathbb{R}^{2}\right)$
the functions
\begin{equation}
f\left(.,\hat{y}\right),\,f\left(\hat{x},.\right)\in\mathcal{S}\left(\mathbb{R}^{2}\right)\label{eq:fixed argument gives Schwartz}
\end{equation}
 are Schwartz for fixed $\hat{y}$ and $\hat{x}$ respectively. 

We now introduce some symbol spaces.
\begin{defn}
\label{def:symbol class} Let $r\in\mathbb{N}$, $r\geq2$. A function
$a\left(x,y,t\right)\in C^{\infty}\left(\mathbb{R}_{x,y}^{6}\times\mathbb{R}_{t}\right)$
is said to lie in the symbol class $\hat{S}_{\frac{1}{r}}^{m}$, $m\in\mathbb{R}$,
if for each $\left(\alpha,\beta,\gamma\right)\in\mathbb{N}_{0}^{7}$
there exists $N\left(\alpha,\beta,\gamma\right)\in\mathbb{N}$ such
that
\begin{align}
\left|\partial_{x}^{\alpha}\partial_{y}^{\beta}\partial_{t}^{\gamma}a(x,y,t)\right| & \leq C_{N,\alpha\beta\gamma}\left\langle t\right\rangle ^{m-\gamma+\frac{1}{r}\left(\left|\hat{\alpha}\right|+\left|\hat{\beta}\right|\right)+\alpha_{3}+\beta_{3}}\frac{\left(1+\left|t^{\frac{1}{r}}\hat{x}\right|+\left|t^{\frac{1}{r}}\hat{y}\right|\right)^{N\left(\alpha,\beta,\gamma\right)}}{\left(1+\left|t^{\frac{1}{r}}\hat{x}-t^{\frac{1}{r}}\hat{y}\right|\right)^{-N}},\label{symbolic estimates}\\
 & \quad\quad\quad\quad\qquad\qquad\forall\left(x,y,t,N\right)\in\mathbb{R}_{x,y}^{6}\times\mathbb{R}_{t}\times\mathbb{N}.\nonumber 
\end{align}
 We further set 
\begin{equation}
\hat{S}_{\frac{1}{r}}^{m,k}\coloneqq\bigoplus_{p+p'\leq k}\left(tx_{3}\right)^{p}\left(ty_{3}\right)^{p'}\hat{S}_{\frac{1}{r}}^{m},\quad\forall\left(m,k\right)\in\mathbb{R}\times\mathbb{N}_{0}.\label{eq:S^mk def.}
\end{equation}

The subset $\hat{S}_{\frac{1}{r},{\rm cl\,}}^{m}\subset\hat{S}_{\frac{1}{r}}^{m}$
of classical symbols is those $a\left(x,y,t\right)$ for which there
exist $a_{jpp'}\left(\hat{x},\hat{y}\right)\in\hat{S}\left(\mathbb{R}^{2}\times\mathbb{R}^{2}\right)$,
$j,p,p'\in\mathbb{N}_{0}$, such that
\begin{equation}
a\left(x,y,t\right)-\sum_{j=0}^{N}\sum_{p+p'\leq j}t^{m-\frac{1}{r}j}\left(tx_{3}\right)^{p}\left(ty_{3}\right)^{p'}a_{jpp'}\left(t^{\frac{1}{r}}\hat{x},t^{\frac{1}{r}}\hat{y}\right)\in\hat{S}_{\frac{1}{r}}^{m-\left(N+1\right)\frac{1}{r},N+1}\label{eq:symbolic expansion}
\end{equation}
$\forall N\in\mathbb{N}_{0}$. We also set 
\[
\hat{S}_{\frac{1}{r},{\rm cl\,}}^{m,k}\coloneqq\bigoplus_{p+p'\leq k}\left(tx_{3}\right)^{p}\left(ty_{3}\right)^{p'}\hat{S}_{\frac{1}{r},{\rm cl\,}}^{m}.
\]
\end{defn}
The following inclusions are clear
\begin{align}
\partial_{t}\hat{S}_{\frac{1}{r}}^{m,k} & \subset\hat{S}_{\frac{1}{r}}^{m-1,k}\nonumber \\
\partial_{\hat{x}}\hat{S}_{\frac{1}{r}}^{m,k},\,\partial_{\hat{y}}\hat{S}_{\frac{1}{r}}^{m,k} & \subset\hat{S}_{\frac{1}{r}}^{m+\frac{1}{r},k}\nonumber \\
\partial_{x_{3}}\hat{S}_{\frac{1}{r}}^{m,k},\,\partial_{y_{3}}\hat{S}_{\frac{1}{r}}^{m,k} & \subset\hat{S}_{\frac{1}{r}}^{m+1,k}\nonumber \\
\hat{S}_{\frac{1}{r}}^{m,k} & \subset\hat{S}_{\frac{1}{r}}^{m+1,k-1},\quad k\geq1,\nonumber \\
\hat{S}_{\frac{1}{r}}^{m,k} & \subset\hat{S}_{\frac{1}{r}}^{m,k+1},\nonumber \\
\hat{S}_{\frac{1}{r}}^{m,k} & \subset\hat{S}_{\frac{1}{r}}^{m',k},\quad m<m',\label{eq:inclusions symbol spaces}
\end{align}
with similar inclusions applying for $\hat{S}_{\frac{1}{r},\textrm{cl}}^{m,k}$.

Next we set 
\begin{align*}
\hat{S}_{\frac{1}{r}}^{m,k,-\infty} & \coloneqq\cap_{j\in\mathbb{N}_{0}}\hat{S}_{\frac{1}{r}}^{m-\frac{j}{r},k+j}\\
\hat{S}_{\frac{1}{r}}^{-\infty} & \coloneqq\cup_{m,k}\hat{S}_{\frac{1}{r}}^{m,k,-\infty}.
\end{align*}

Following a standard Borel construction, one has asymptotic summation:
for any $a_{j}\in\hat{S}_{\frac{1}{r}}^{m-\frac{1}{r}j,k+j}$, $j=0,1,\ldots$,
there exists $a\in\hat{S}_{\frac{1}{r}}^{m,k}$ such that 
\begin{equation}
a-\left(\sum_{j=1}^{N}a_{j}\right)\in\hat{S}_{\frac{1}{r}}^{m-\frac{1}{r}\left(N+1\right),k+N+1},\quad\forall N\in\mathbb{N},\label{eq:asymptotic sum}
\end{equation}
with a similar property being true for the classical symbols $\hat{S}_{\frac{1}{r},{\rm cl\,}}^{m}$.
Moreover the symbol $a$ \prettyref{eq:asymptotic sum} above is unique
modulo $\hat{S}_{\frac{1}{r}}^{m,k,-\infty}$.

We now define the quantizations of the symbols in \prettyref{def:symbol class}.
\begin{defn}
\label{exotic pseudos def.} An operator $G:C_{c}^{\infty}\left(\mathbb{R}^{3}\right)\rightarrow C^{-\infty}\left(\mathbb{R}^{3}\right)$
is said to be in the class $G\in\hat{L}_{\frac{1}{r},{\rm cl\,}}^{m,k}$
if its distribution kernel satisfies
\begin{equation}
G\left(x,y\right)\equiv g^{L}\coloneqq\int_{0}^{\infty}dt\,e^{it\left(x_{3}-y_{3}\right)}g\left(x,y,t\right)\quad\label{eq:quantization definition}
\end{equation}
for some $g\in\hat{S}_{\frac{1}{r},\textrm{cl}}^{m,k}+\hat{S}_{\frac{1}{r}}^{-\infty}$. 
\end{defn}
It is an easy exercise that for $G\in\hat{L}_{\frac{1}{r}}^{m}$,
$m<-1-k$, the kernel $G\left(.,y\right)\in C^{k}$ for fixed $y$.
We next have a reduction lemma showing that the show that the amplitude
$g$ in the quantization above \prettyref{eq:quantization definition}
maybe chosen independent of $x_{3}$ or $y_{3}$. 
\begin{lem}
For any $g\in\hat{S}_{\frac{1}{r},\textrm{cl}}^{m,k}$ there exist
$g_{1},g_{2}\in\hat{S}_{\frac{1}{r},\textrm{cl}}^{m,k}$ independent
of $x_{3},y_{3}$ respectively such that $g^{L}=g_{1}^{L}=g_{2}^{L}$
.
\end{lem}
\begin{proof}
 By a Fourier transform, it is easy to see that
\begin{align*}
G & =g_{1}^{L}\quad\textrm{for }\\
g_{1}\left(\hat{x},x_{3},\hat{y},t\right) & =\left[e^{i\partial_{t}\partial_{y_{3}}}g\right]_{y_{3}=x_{3}}.
\end{align*}
Here the above notation follows \cite[Sec. 7.6]{HormanderI} wherein
the partial $y_{3},t$ Fourier transform $\mathcal{F}_{y_{3},t}$
of $e^{i\partial_{t}\partial_{y_{3}}}g$ is given
\[
e^{i\partial_{t}\partial_{y_{3}}}g=\mathcal{F}_{y_{3},t}^{-1}e^{i\tau\eta_{3}}\mathcal{F}_{y_{3},t}g
\]
 by multiplication by the exponential of the dual variables $\eta_{3},\tau$
respectively. From \cite[Thm. 7.6.5]{HormanderI} and \prettyref{eq:inclusions symbol spaces}
it is then easy to see that $g_{1}\in\hat{S}_{\frac{1}{r}}^{m+6,k}\left(U\right)$.
In particular we have $g_{1}\in\hat{S}_{\frac{1}{r}}^{-\infty}$ for
$g\in\hat{S}_{\frac{1}{r}}^{-\infty}$.

Next for $g\in\hat{S}_{\frac{1}{r},\textrm{cl}}^{m,k}$ we plug in
its classical expansion \prettyref{eq:symbolic expansion} into \prettyref{eq:quantization definition}.
By writing $ty_{3}=tx_{3}+t\left(y_{3}-x_{3}\right)$ and repeated
integration by parts using $\partial_{t}e^{it\left(x_{3}-y_{3}\right)}=i\left(x_{3}-y_{3}\right)e^{it\left(x_{3}-y_{3}\right)}$
we obtain $g_{1,N}\in\hat{S}_{\frac{1}{r},\textrm{cl}}^{m,k}$, $N\in\mathbb{N}_{0}$,
independent of $y_{3}$ such that $g_{1}-g_{1,N}\in\hat{S}_{\frac{1}{r}}^{m-\left(N+1\right)\frac{1}{r},k+N+1}$
, $\forall N\in\mathbb{N}_{0}$. By asymptotic summation we find $g_{1}\sim g_{1,1}+\sum_{N=1}^{\infty}\left(g_{1,N+1}-g_{1,N}\right)\in\hat{S}_{\frac{1}{r},\textrm{cl}}^{m,k}$
, independent of $y_{3}$, which satisfies $g-g_{1}\in\hat{S}_{\frac{1}{r}}^{m,k,-\infty}\subset\hat{S}_{\frac{1}{r}}^{-\infty}$.
From this and the first part of the proof the Lemma follows. The construction
of $g_{2}$ is similar.
\end{proof}
Following the above we shall define the principal symbol in $\hat{L}_{\frac{1}{r},{\rm cl\,}}^{m,k}$
via
\[
\sigma_{L}\left(G\right)=g_{000}\left(\hat{x},\hat{y}\right)\in\hat{S}(\mathbb{R}^{2}\times\mathbb{R}^{2}),
\]
$G\in\hat{L}_{\frac{1}{r},{\rm cl\,}}^{m,k}$, as the leading term
in the symbolic expansion \prettyref{eq:symbolic expansion}. The
following symbol exact sequence is then clear 
\[
0\rightarrow\hat{L}_{\frac{1}{r},{\rm cl\,}}^{m-\frac{1}{r},k+1}\rightarrow\hat{L}_{\frac{1}{r},{\rm cl\,}}^{m,k}\xrightarrow{\sigma_{L}}\hat{S}(\mathbb{R}^{2}\times\mathbb{R}^{2})\rightarrow0.
\]
The class $\hat{L}_{\frac{1}{r},{\rm cl\,}}^{m,k}(U)$ is clearly
closed under adjoints. The symbol of the adjoint is furthermore easily
computed 
\begin{equation}
\sigma_{L}\left(G^{*}\right)\left(\hat{x},\hat{y}\right)=\overline{\sigma_{L}\left(G\right)\left(\hat{y},\hat{x}\right)}.\label{eq:symbol adjoint}
\end{equation}

We next have the composition of operators in $\hat{L}_{\frac{1}{r},{\rm cl\,}}^{m,k}$.
\begin{prop}
\label{closure under composition} For any $G\in\hat{L}_{\frac{1}{r},{\rm cl\,}}^{m,k}$,
$H\in\hat{L}_{\frac{1}{r},{\rm cl\,}}^{m',k'}$ one has the composition
$G\circ H\in\hat{L}_{\frac{1}{r},{\rm cl\,}}^{m+m'-\frac{2}{r},k+k'}$.
Furthermore the leading symbol of the composition is given by
\begin{equation}
\sigma_{L}\left(G\circ H\right)\left(\hat{x},\hat{y}\right)=\int d\hat{u}\,\sigma_{L}\left(G\right)\left(\hat{x},\hat{u}\right)\sigma_{L}\left(H\right)\left(\hat{u},\hat{y}\right).\label{eq:leading symbol composition}
\end{equation}
 
\end{prop}
\begin{proof}
Write $G=g_{1}^{L}$ , $H=h_{2}^{L}$ in terms of their $x_{3},y_{3}$
independent quantizations respectively. From Fourier inversion it
is easy to check that 
\begin{align*}
\left(G\circ H\right)\left(x,y\right) & =\int dt\,e^{it\left(x_{3}-y_{3}\right)}\left(g\circ h\right)\left(x,y,t\right)\quad\textrm{ for }\\
\left(g\circ h\right)\left(x,y,t\right) & \coloneqq\int d\hat{u}\,g\left(\hat{x},x_{3},\hat{u},t\right)h\left(\hat{u},\hat{y},y_{3},t\right)\\
 & =t^{-2\frac{1}{r}}\int dt\,e^{it(x_{3}-y_{3})}d\hat{v}g\left(\hat{x},x_{3},t^{-\frac{1}{r}}\hat{v},t\right)h\left(t^{-\frac{1}{r}}\hat{v},\hat{y},y_{3},t\right)
\end{align*}
upon a change of variables $t^{\frac{1}{r}}\hat{u}=\hat{v}$. The
$\hat{v}$ integral is seen to be convergent on account of \prettyref{eq:fixed argument gives Schwartz},
which also gives the necessary symbolic estimates for $g\circ h$.
To obtain the symbolic expansion, we plug $x_{3},y_{3}$ independent
symbolic expansions for $g$, $h$ respectively into the above to
obtain a symbolic expansion for the composed symbol $g\circ h$ along
with the formula \prettyref{eq:leading symbol composition} for the
leading part.
\end{proof}
Finally we show that our algebra of operators is a module over the
usual algebra of pseudodifferential operators.
\begin{prop}
\label{class is module over pso} Let $G\in\hat{L}_{\frac{1}{r},{\rm cl\,}}^{m,k}$
and let $P\in\Psi_{{\rm cl\,}}^{m'}$ be a classical pseudodifferential
operator on $U$ of order $k$. Then, $PG\in\hat{L}_{\frac{1}{r},{\rm cl\,}}^{m+m',k}$
with leading symbol
\begin{equation}
\sigma_{L}\left(PG\right)\left(\hat{x},\hat{y}\right)=\sigma\left(P\right)\left(0,0;0,1\right)\sigma_{L}\left(G\right)\left(\hat{x},\hat{y}\right).\label{eq:symbol of product with pdo}
\end{equation}
\end{prop}
\begin{proof}
First write the kernels
\begin{align*}
P\left(x,u\right) & =\frac{1}{\left(2\pi\right)^{3}}\int d\xi e^{i\left(x-u\right)\xi}p\left(x,\xi\right)\\
G\left(u,y\right) & =\int dte^{i\left(u_{3}-y_{3}\right)t}g\left(\hat{u};\hat{y},y_{3},t\right)
\end{align*}
 $p\in S_{{\rm cl\,}}^{m'}$, $g\in\hat{S}_{\frac{1}{r},\textrm{cl}}^{m,k}$
using an $u_{3}$ independent quantization for $G$. Then Fourier
inversion gives the composition to be 
\begin{align}
\left(P\circ G\right)\left(x,y\right) & =\int dt\,e^{i\left(x_{3}-y_{3}\right)t}q\left(x,y,t\right)\nonumber \\
q\left(x,y,t\right) & \coloneqq\frac{1}{\left(2\pi\right)^{2}}\int d\hat{u}d\hat{\xi}e^{i\left(\hat{x}-\hat{u}\right)\hat{\xi}}p\left(x,\hat{\xi},t\right)g\left(\hat{u};\hat{y},y_{3},t\right).\label{eq:formula composition}
\end{align}
Again the above amplitude satisfies necessary symbolic estimates on
account of \prettyref{eq:fixed argument gives Schwartz}. 

To obtain the symbolic expansion, a change of variables $\hat{v}=t^{\frac{1}{r}}\hat{u}$,$\hat{\eta}=t^{-\frac{1}{r}}\hat{\xi}$
first gives
\begin{align}
q\left(t^{-\frac{1}{r}}\hat{x},x_{3},t^{-\frac{1}{r}}\hat{y},y_{3},t\right) & =\int d\hat{u}d\hat{\xi}e^{i\left(t^{-\frac{1}{r}}\hat{x}-\hat{u}\right)\hat{\xi}}p\left(t^{-\frac{1}{r}}\hat{x},x_{3},\hat{\xi},t\right)g\left(\hat{u};t^{-\frac{1}{r}}\hat{y},y_{3},t\right)\nonumber \\
 & =\int d\hat{v}d\hat{\eta}e^{i\left(\hat{x}-\hat{v}\right)\hat{\eta}}p\left(t^{-\frac{1}{r}}\hat{x},x_{3},t^{\frac{1}{r}}\hat{\eta},t\right)g\left(t^{-\frac{1}{r}}\hat{v};t^{-\frac{1}{r}}\hat{y},y_{3},t\right).\label{eq:symbol of composition}
\end{align}
Next we plug in the symbolic expansion for $g$ as well as
\[
p\left(t^{-\frac{1}{r}}\hat{x},x_{3},t^{\frac{1}{r}}\hat{\eta},t\right)\sim t^{k}\left[p_{0}\left(0,x_{3},0,1\right)+\sum_{j=1}^{\infty}t^{-j/r}p_{j}\left(\hat{x},x_{3},\hat{\eta}\right)\right],
\]
obtained from the classical symbolic expansion for $p$, into \prettyref{eq:formula composition},
\prettyref{eq:symbol of composition}. A further Taylor expansion
in $x_{3}$ for $p_{0}\left(0,x_{3},0,1\right)\sim\sum_{j=0}^{\infty}t^{-j}\left(tx_{3}\right)^{j}\left(\partial_{x_{3}}^{j}p_{0}\right)\left(0,0,0,1\right)$
and each $p_{j}$ plugged into the above completes the proof.
\end{proof}
Finally we need some mapping properties of operators in $\hat{L}_{\frac{1}{r},{\rm cl\,}}^{m,k}$.
To introduce the functional spaces first define 
\begin{align}
\hat{S}_{\frac{1}{r}}^{m}\left(\mathbb{R}^{2}\right) & \subset\hat{S}_{\frac{1}{r}}^{m}\left(\mathbb{R}^{3}\times\mathbb{R}^{3}\right)\nonumber \\
\hat{S}_{\frac{1}{r},{\rm cl\,}}^{m}\left(\mathbb{R}^{2}\right) & \subset\hat{S}_{\frac{1}{r},{\rm cl\,}}^{m}\left(\mathbb{R}^{3}\times\mathbb{R}^{3}\right)\label{eq:x3 y independent classes}
\end{align}
 as the subspace of $x_{3},y$-independent elements in \prettyref{def:symbol class}.
Note that the above are included 
\begin{align}
\hat{S}_{\frac{1}{r}}^{m}\left(\mathbb{R}^{2}\right) & \subset S_{\frac{1}{r}}^{m}\left(\mathbb{R}^{2}\times\mathbb{R}_{t}\right)\nonumber \\
\hat{S}_{\frac{1}{r},{\rm cl\,}}^{m}\left(\mathbb{R}^{2}\right) & \subset S_{\frac{1}{r},\textrm{cl}}^{m}\left(\mathbb{R}^{2}\times\mathbb{R}_{t}\right)\label{eq:inclusion into Hormander class}
\end{align}
in the Hörmander symbol classes from the introduction.

We next define the space of partial $t$-Fourier transforms of the
classes \prettyref{eq:x3 y independent classes} below
\begin{align}
S_{H}^{m}\left(\mathbb{R}^{3}\right) & \coloneqq\left\{ p\in\mathcal{S}'\left(\mathbb{R}^{3}\right)|p=\int dte^{itx_{3}}a\left(t,\hat{x}\right),\;a\in\hat{S}_{\frac{1}{r}}^{m}\left(\mathbb{R}^{2}\right)\right\} \nonumber \\
S_{H,{\rm cl\,}}^{m}\left(\mathbb{R}^{3}\right) & \coloneqq\left\{ p\in\mathcal{S}'\left(\mathbb{R}^{3}\right)|p=\int dte^{itx_{3}}a\left(t,\hat{x}\right),\;a\in\hat{S}_{\frac{1}{r},{\rm cl\,}}^{m}\left(\mathbb{R}^{2}\right)\right\} .\label{eq:functional space}
\end{align}
It is an easy exercise using Fourier transforms to see that elements
of $S_{H}^{m}\left(\mathbb{R}^{3}\right)$ \prettyref{eq:functional space}
above are smooth outside the origin. While the space $S_{H}^{m}\left(U\right)$
consists of restrictions to $U$ of elements in the space $S_{H}^{m}\left(\mathbb{R}^{3}\right)$
defined above. It is further easy to see the inclusion 
\begin{equation}
S_{H}^{m}\left(\mathbb{R}^{3}\right)\subset C^{\alpha\ }\left(\mathbb{R}^{3}\right),\quad m<-1-\alpha.\label{eq:regularity of spaces}
\end{equation}
 We now have the following.
\begin{prop}
\label{prop:functional boundedness} For $G\in\hat{L}_{\frac{1}{r},{\rm cl\,}}^{m,k}\left(\mathbb{R}^{3}\right)$
and $p\in S_{H}^{m'}\left(\mathbb{R}^{3}\right)$ we have $Gp\in S_{H}^{m+m'-\frac{2}{r}}\left(\mathbb{R}^{3}\right)$.
A similar property holds for $S_{H,{\rm cl\,}}^{m}\left(\mathbb{R}^{3}\right)$.
\end{prop}
\begin{proof}
Again using a $y_{3}$ independent quantization for $G$, the Fourier
transform expression $p=\int dte^{itx_{3}}a\left(t,\hat{x}\right)$
and Fourier inversion gives 
\begin{align*}
Gp\left(x\right) & =\int dte^{ix_{3}t}\left(g\circ a\right)\left(x,t\right)\\
\left(g\circ a\right)\left(x,t\right) & \coloneqq\int g\left(x,\hat{y},t\right)a\left(t,\hat{y}\right)d\hat{y}.
\end{align*}
Next we plugin the symbolic expansion for $g$ into the above and
use repeated integration by parts using $\partial_{t}e^{ix_{3}t}=ix_{3}e^{ix_{3}t}$
to obtain $x_{3}$-independence of the amplitude modulo $C^{\infty}$.
Plugging in a classical expansion for $p\in S_{H,{\rm cl\,}}^{m}\left(\mathbb{R}^{3}\right)$
gives a similar expansion for $g\circ a$.
\end{proof}

\subsection{Local Bergman kernels}

In this section we shall define certain local Bergman kernels using
the coordinates introduced in Sec. \prettyref{subsec:Construction-of-coordinates}.
Furthermore these shall be shown to lie in the symbol classes introduced
in the previous section.

First with the notation as in \prettyref{thm:almost analytic coordinates}
one sets $V=U^{\mathbb{C}}\cap\left\{ q=0\right\} \subset\mathbb{R}_{p}^{3}$.
With $\chi\left(p_{1},p_{2}\right)\in C_{c}^{\infty}\left(\mathbb{R}^{2}\right)$
of sufficiently small support and equal to one near zero, the function
\begin{equation}
\varphi\left(p_{1},p_{2}\right)\coloneqq\left.\varphi_{0}\right|_{q=0}+\underbrace{\chi\left.\left(\varphi-\varphi_{0}\right)\right|_{q=0}}_{\eqqcolon\varphi_{1}}\label{eq:modified potential}
\end{equation}
is well defined on $\mathbb{R}^{2}$. This equals the restriction
of $\varphi$ to $V$ near the origin, and hence we use the same notation.
Next set 
\begin{equation}
\hat{Z}\coloneqq\frac{1}{2}\left(\partial_{p_{1}}+i\partial_{p_{2}}\right)-\frac{i}{2}\left(\partial_{p_{1}}\varphi+i\partial_{p_{2}}\varphi\right)\partial_{p_{3}}\label{eq: aa vector field restriction}
\end{equation}
 and define
\begin{align}
\bar{\partial}_{t}: & \Omega^{0,0}\left(\mathbb{R}^{2}\right)\rightarrow\Omega^{0,1}\left(\mathbb{R}^{2}\right)\nonumber \\
\bar{\partial}_{t}u & \coloneqq\left[\frac{1}{2}\left(\partial_{p_{1}}+i\partial_{p_{2}}\right)u+\frac{1}{2}t\left(\partial_{p_{1}}\varphi+i\partial_{p_{2}}\varphi\right)u\right]d\bar{z}\label{eq:Dolbeault}
\end{align}
with $d\bar{z}=dp_{1}-idp_{2}$. Define the Kodaira Laplacian via
\begin{align}
\Box_{t} & \coloneqq\bar{\partial}_{t}^{*}\bar{\partial}_{t}\label{eq:Kodaira Dirac =000026 laplacian}
\end{align}
acting on $\Omega^{0,0}$ . We denote by 
\begin{equation}
B_{t}:L^{2}\left(\mathbb{R}_{p}^{2}\right)\rightarrow\textrm{ker}\left(\Box_{t}\right)\label{eq:local Bergman}
\end{equation}
the local Bergman projector onto the kernel of $\Box_{t}$ and $B_{t}\left(p,p'\right)$
it Schwartz kernel. 

Replacing $\varphi$ with its leading polynomial $\varphi_{0}$ in
\prettyref{eq: aa vector field restriction}, \prettyref{eq:Dolbeault},
\prettyref{eq:Kodaira Dirac =000026 laplacian} one analogously defines
\begin{align}
\bar{\partial}_{t}^{0}u & \coloneqq\left[\frac{1}{2}\left(\partial_{p_{1}}+i\partial_{p_{2}}\right)u+\frac{1}{2}t\left(\partial_{p_{1}}\varphi_{0}+i\partial_{p_{2}}\varphi_{0}\right)u\right]d\bar{z}\label{eq:leading dolbeault}\\
\Box_{t}^{0} & \coloneqq\left(\bar{\partial}_{t}^{0}\right)^{*}\bar{\partial}_{t}^{0}\label{eq: leading Kodaira laplacian}
\end{align}
 as well as a corresponding Bergman projection $B_{t}^{0}$ with kernel
$B_{t}^{0}\left(\hat{p},\hat{p}'\right)$. 
\begin{thm}
\label{thm:local Bergman is a symbol} One has $B_{t}^{0}\left(\hat{p},\hat{p}'\right),\,B_{t}\left(\hat{p},\hat{p}'\right)\in\hat{S}_{\frac{1}{r},{\rm cl\,}}^{\frac{2}{r},0}$,
with furthermore 
\begin{align}
B_{t}^{0}\left(\hat{p},\hat{p}'\right) & =t^{\frac{2}{r}}b_{0}\left(t^{\frac{1}{r}}\hat{p},t^{\frac{1}{r}}\hat{p}'\right)\label{eq:B0 is symbol}\\
B_{t}\left(\hat{p},\hat{p}'\right) & =t^{\frac{2}{r}}b_{0}\left(t^{\frac{1}{r}}\hat{p},t^{\frac{1}{r}}\hat{p}'\right)+\hat{S}_{\frac{1}{r},{\rm cl\,}}^{\frac{1}{r},0}\label{eq:B is symbol}
\end{align}
 for some $b_{0}(\hat{p},\hat{p}')\in\hat{S}\left(\mathbb{R}^{2}\times\mathbb{R}^{2}\right)$. 
\end{thm}
\begin{proof}
Being symmetric and bounded below, the Kodaira Laplacians \prettyref{eq:Kodaira Dirac =000026 laplacian},
\prettyref{eq: leading Kodaira laplacian} are essentially self-adjoint.
Furthermore under the rescaling/dilation $\delta_{t^{-1/r}}\left(\hat{p},\hat{p}'\right)=\left(t^{-\frac{1}{r}}\hat{p},t^{-\frac{1}{r}}\hat{p}'\right)$
these are seen to satisfy
\begin{align}
\boxdot_{t}^{0}\coloneqq t^{-2/r}\left(\delta_{t^{-1/r}}\right)_{*}\Box_{t}^{0} & =\Box_{1}^{0}\nonumber \\
\boxdot_{t}\coloneqq t^{-2/r}\left(\delta_{t^{-1/r}}\right)_{*}\Box_{t} & =\Box_{1}^{0}+t^{-1/r}E,\label{eq:rescaled operators}
\end{align}
where 
\begin{align*}
E & =a\left(\hat{p},t\right)\bar{\partial}_{1}^{0}+b\left(\hat{p},t\right)\bar{\partial}_{1}^{0}+c\left(\hat{p},t\right)
\end{align*}
is a self-adjoint operator with the coefficients $a\left(\hat{p},t\right),$
$b\left(\hat{p},t\right)$ and $c\left(\hat{p},t\right)$ being uniformly
(in $t$) $C^{\infty}$ bounded. 

Next, as in \cite[Sec. 4.1]{Marinescu-Savale18}, see also Prop. \prettyref{prop:spectral gap}
below, we have 
\begin{align}
\textrm{Spec}\left(\Box_{t}^{0}\right) & \subset\left\{ 0\right\} \cup\left[c_{1}t^{2/r},\infty\right)\label{eq: spectral gap funs.-1}\\
\textrm{Spec}\left(\Box_{t}\right) & \subset\left\{ 0\right\} \cup\left[c_{1}t^{2/r}-c_{2},\infty\right).\label{eq:spectral gap fns 2}
\end{align}
We remark the fact that $\varphi$ here is complex valued makes no
difference to the above formulas \prettyref{eq: spectral gap funs.-1},
\prettyref{eq:spectral gap fns 2} so far as the leading part $\varphi_{0}$
is real and sub-harmonic. This is because the higher order Taylor
coefficients of $\varphi$ appear at lower order $O\left(t^{-1/r}\right)$
after rescaling in \prettyref{eq:rescaled operators}. Hence for any
$\chi\in C_{c}^{\infty}\left(-c_{1},c_{1}\right)$ with $\chi=1$
near $0$, the Bergman kernels equal
\begin{align*}
B_{t}^{0}\left(\hat{p},\hat{p}'\right) & =\chi\left(t^{-2/r}\Box_{t}^{0}\right)\left(\hat{p},\hat{p}'\right)=t^{2/r}\chi\left(\Box_{1}^{0}\right)\left(t^{1/r}p,t^{1/r}p'\right)\\
B_{t}\left(\hat{p},\hat{p}'\right) & =\chi\left(t^{-2/r}\Box_{t}\right)\left(\hat{p},\hat{p}'\right)=t^{2/r}\chi\left(\boxdot_{t}\right)\left(t^{1/r}p,t^{1/r}p'\right)
\end{align*}
for $t\gg0$. By standard elliptic arguments, the Schwartz kernels
of $\partial_{p}^{\alpha}\partial_{p'}^{\alpha'}\chi\left(\Box_{1}^{0}\right)$,
$\partial_{p}^{\alpha}\partial_{p'}^{\alpha'}\chi\left(\boxdot_{t}\right)$,
$\alpha,\alpha'\in\mathbb{N}_{0}^{2}$, are rapidly decaying off-diagonal.
Regarding their on-diagonal behavior, the growth of $\chi\left(\Box_{1}^{0}\right)\left(\hat{p},\hat{p}\right),\,\chi\left(\boxdot_{t}\right)\left(\hat{p},\hat{p}\right)$
as $\hat{p}\rightarrow\infty$ is controlled by the growth of the
coefficient functions of the operators \prettyref{eq:rescaled operators},
which in turn have polynomial growth. Hence $B_{t}^{0}$, $B_{t}$
satisfy estimates \ref{symbolic estimates} with $N\left(\hat{\alpha},\hat{\beta}\right)=r\left(\left|\hat{\alpha}\right|+\left|\hat{\beta}\right|\right)$,
cf. also \prettyref{eq:uniform derivative estimate} below. This gives
\prettyref{eq:B0 is symbol} with 
\begin{align}
b_{0}\left(\hat{p},\hat{p}'\right) & =\chi\left(\Box_{1}^{0}\right)\left(\hat{x},\hat{y}\right)\in\hat{S}\left(\mathbb{R}^{2}\times\mathbb{R}^{2}\right)\nonumber \\
B_{t}\left(\hat{p},\hat{p}'\right) & \in\hat{S}_{\frac{1}{r}}^{\frac{2}{r},0}.\label{eq:bergman kernels are non-classical symbols}
\end{align}

To show the classical expansion for the above one may use a full expansion
of the operator $\boxdot_{t}$ \prettyref{eq:rescaled operators}
as in \prettyref{sec:S1 invariant CR geometry} below. We shall however
give a different proof consistent with the rest of this section. To
this end, first begin with $\varphi=\varphi_{0}+\varphi_{1}$ from
\prettyref{eq:modified potential}, where
\begin{equation}
\varphi_{1}\left(\hat{p}\right)=O\left(\left|\hat{p}\right|^{r+1}\right),\label{eq:higher order part potential}
\end{equation}
$\varphi_{1}(\hat{p})\in C_{0}^{\infty}(\mathbb{R}^{2},\mathbb{C})$.
Next define the operator with distributional kernel 
\begin{align}
\tilde{B}_{t}: & L^{2}(\mathbb{R}^{2})\rightarrow L^{2}(\mathbb{R}^{2})\nonumber \\
\tilde{B}_{t}\left(\hat{p},\hat{p}'\right) & =e^{-t\varphi_{1}\left(\hat{p}\right)}B_{t}^{0}\left(\hat{p},\hat{p}'\right)e^{t\varphi_{1}\left(\hat{p}'\right)}.\label{eq:approximate local bergman kernel}
\end{align}
It is clear that the above $\bar{\partial}_{t}\tilde{B}_{t}=0$, $\Box_{t}\tilde{B}_{t}=0$
lies in the kernels of \prettyref{eq:Dolbeault}, \prettyref{eq:Kodaira Dirac =000026 laplacian}.
This gives $B_{t}\tilde{B}_{t}=\tilde{B}_{t}$ and 
\begin{equation}
\tilde{B}_{t}^{*}B_{t}=\tilde{B}_{t}^{*},\label{e-gue200512yydI}
\end{equation}
where $\tilde{B}_{t}^{*}$ is the adjoint of $\tilde{B}_{t}$. Let
$R_{t}:=\tilde{B}_{t}-\tilde{B}_{t}^{*}$ whose Schwartz kernel is
computed to be
\begin{equation}
R_{t}\left(\hat{p},\hat{p}'\right)=e^{-t\varphi_{1}\left(\hat{p}\right)}B_{t}^{0}\left(\hat{p},\hat{p}'\right)e^{t\varphi_{1}\left(\hat{p}'\right)}-e^{t\bar{\varphi_{1}}\left(\hat{p}\right)}B_{t}^{0}\left(\hat{p},\hat{p}'\right)e^{-t\bar{\varphi_{1}}\left(\hat{p}'\right)}.\label{e-gue200512yydII}
\end{equation}
Since $\left(\bar{\partial}+t\left(\bar{\partial}\varphi_{0}\right){}^{\wedge}\right)\left(e^{t\varphi_{1}}B_{t}\right)=0$,
we have 
\begin{equation}
\tilde{B}_{t}B_{t}=B_{t}.\label{e-gue200512yydIII}
\end{equation}
From \eqref{e-gue200512yydI} and \eqref{e-gue200512yydIII}, we get
$(I-R_{t})B_{t}=\tilde{B}_{t}^{*}$ and hence 
\begin{equation}
(I-R_{t}^{N})B_{t}=(I+R_{t}+R_{t}^{2}+\cdots+R_{t}^{N-1})\tilde{B}_{t}^{*},\ \ \mbox{ \ensuremath{\forall}\ensuremath{N\ensuremath{\in\mathbb{N}}}}.\label{e-gue200512yyda}
\end{equation}
From the first part \prettyref{eq:bergman kernels are non-classical symbols},
\prettyref{eq:higher order part potential}, \eqref{e-gue200512yydII}
and a Taylor expansion, it is easy to see that $R_{t}\in\hat{S}_{\frac{1}{r},{\rm cl\,}}^{\frac{1}{r},0}$
and 
\begin{equation}
\mbox{\ensuremath{R_{t}^{j}\in\hat{S}_{\frac{1}{r},{\rm cl\,}}^{\left(2-j\right)\frac{1}{r},0},\;} \ensuremath{\forall}\ensuremath{j\ensuremath{\in\mathbb{N}}}},\label{e-gue200512yydb}
\end{equation}
by an argument similar to Prop. \ref{closure under composition}.
From the above, \prettyref{eq:bergman kernels are non-classical symbols},
\eqref{e-gue200512yyda}, \eqref{e-gue200512yydb} and $\tilde{B}_{t}^{*}\in\hat{S}_{\frac{1}{r},{\rm cl\,}}^{\frac{2}{r}}$,
the theorem follows. 
\end{proof}
Following the above we now prove one of our main theorems \prettyref{thm:main thm parametrix}.
\begin{proof}[Proof of \prettyref{thm:main thm parametrix}]
 Choose $B$ as in \prettyref{eq:local Bergman} and $\chi\in C_{c}^{\infty}\left(\mathbb{R}^{3}\right)$
a cutoff equal to one near zero. Define the operator 
\begin{align}
\hat{B}:C_{c}^{\infty}\left(\mathbb{R}^{3}\right) & \rightarrow C^{-\infty}\left(\mathbb{R}^{3}\right)\nonumber \\
\hat{B} & \coloneqq\frac{1}{2\pi}\int_{0}^{\infty}dt\,e^{it\left(p_{3}-p_{3}'\right)}B_{t}\left(p,p'\right)\chi\left(t^{\frac{1}{r}-\epsilon}\hat{p}',t^{1-\frac{1}{2r}}p_{3}'\right).\label{eq:B hat definition}
\end{align}
By definition using \prettyref{thm:local Bergman is a symbol} and
a Taylor expansion of the cutoff we see $\hat{B}\in\hat{L}_{\frac{1}{r},{\rm cl\,}}^{\frac{2}{r},0}+\hat{L}_{\frac{1}{r},2}^{\frac{2}{r},0,-\infty}$,
for $\hat{L}_{\frac{1}{r},2}^{\frac{2}{r},0,-\infty}\coloneqq\bigcap_{j\in\mathbb{N}_{0}}\hat{L}_{\frac{1}{r}}^{\frac{2-j}{r},2j}$
. Furthermore for $\epsilon$ sufficiently small, that Schwartz kernel
of the above $\hat{B}\left(p,p'\right)$ can be shown to be smooth
away from $p=0$ and satisfies estimates similar to \prettyref{eq: function class}
in $p$. We let $\tilde{B}\left(z,p'\right)$ denote the almost analytic
continuation of the Schwartz kernel $\hat{B}\left(p,p'\right)$ in
the $p$-variable given by \prettyref{lem:Almost-analytic-continuations}.
Now consider the coordinates $\left(x_{1},x_{2},x_{3}\right)$ on
a neighborhood $U$ centered at the point $x'\in X$, along with $\left(p_{1},p_{2},p_{3}\right)$
being (the real parts of) the corresponding almost analytic coordinates
given by \ref{thm:almost analytic coordinates}. Letting $\chi_{1},\chi_{2}\in C_{c}^{\infty}\left(U\right)$
be such that $\chi_{1}=1$ on $\textrm{spt }\left(\chi_{2}\right)$
and $\chi_{2}=1$ near zero we set 
\begin{equation}
B\left(x,x'\right)\coloneqq\chi_{1}\left(x\right)\left(\left.\tilde{B}\right|_{y,y'=0}\right)\chi_{2}\left(x'\right)\in C^{-\infty}\left(X\times X\right).\label{eq:the operator B}
\end{equation}
Since $\hat{Z}\hat{B}=0$ we have $Z\left(\left.\tilde{B}\right|_{y,y'=0}\right)\in C^{\infty}$
by \prettyref{lem:Almost-analytic-continuations} from which is it
easy to check that $\bar{\partial}_{b}B$ is smooth . 

Let $\Pi:L^{2}\left(X\right)\rightarrow H_{b}^{0}\left(X\right)\coloneqq\left\{ u\in L^{2}\left(X\right)|\bar{\partial}_{b}u=0\right\} $
denote the Szeg\H{o} projection. Assuming $\bar{\partial}_{b}$ has
closed range it was shown in \cite[Prop. 4.1]{Christ89-embedding},
\cite{Christ88} that there exists a bounded linear operator $G:\textrm{Range}\left(\bar{\partial}_{b}\right)\rightarrow L^{2}\left(X\right)$
such that $\Pi=I-G\bar{\partial}_{b}$. Furthermore $G$ is microlocal
and it maps $G:\textrm{Range}\left(\bar{\partial}_{b}\right)\cap H^{s}\left(X\right)\rightarrow H^{s+\frac{1}{r}}\left(X\right)$,
$\forall s\in\mathbb{R}$. It now follows that $\Pi$ is microlocal
or that the Szeg\H{o} kernel is smooth away from the diagonal. Furthermore
$\Pi B=B-G\bar{\partial}_{b}B=B+C^{\infty}$ and 
\begin{equation}
B^{*}\Pi=B^{*}+C^{\infty}.\label{eq:BP=00003DB}
\end{equation}

Next, replace $\Pi\left(x,0\right)$ by $\Pi_{1}\left(x\right)=\chi_{1}\left(x\right)\Pi\left(x,0\right)$,
which has the same singularities near $x=0$, is compactly supported
and satisfies similar bounds to \prettyref{eq:christ-M-M-NRSW-bounds}.
We almost analytically continue $\Pi_{1}\left(x,0\right)$ in the
$x$ variable to define $\tilde{\Pi}_{1}\left(z,0\right)$. We may
further suppose that $\tilde{\Pi}_{1}$ is compactly supported by
construction. The restriction $\tilde{\Pi}_{1}\left(p,0\right)=\left.\tilde{\Pi}_{1}\right|_{q=0}$
is well-defined and we set $\tilde{\Pi}_{1,t}\left(p_{1},p_{2}\right)\coloneqq\int e^{-itp_{3}}\tilde{\Pi}_{1}\left(p,0\right)dp_{3}$.
Since $Z\Pi_{1}\in C^{\infty}$ it follows that for the almost analytic
extension $\tilde{Z}\tilde{\Pi}_{1}$ is smooth from \prettyref{lem:Almost-analytic-continuations}.
From \prettyref{thm:almost analytic coordinates} it follows that
$\left[\frac{1}{2}\left(\partial_{\tilde{z}_{1}}+i\partial_{\tilde{z}_{2}}\right)-\frac{i}{2}\left(\partial_{\tilde{z}_{1}}\varphi+i\partial_{\tilde{z}_{2}}\varphi\right)\partial_{\tilde{z}_{3}}\right]\tilde{\Pi}_{1}$
is smooth. Hence
\[
\bar{\partial}_{t}\tilde{\Pi}_{1,t}\coloneqq\left[\frac{1}{2}\left(\partial_{p_{1}}+i\partial_{p_{2}}\right)-\frac{i}{2}t\left(\partial_{p_{1}}\varphi+i\partial_{p_{2}}\varphi\right)\right]\tilde{\Pi}_{1,t}=O\left(t^{-\infty}\right)
\]
 in the Schwartz norm. From here it is clear that $B_{t}\tilde{\Pi}_{1,t}=\tilde{\Pi}_{1,t}+O\left(t^{-\infty}\right)$
in the Schwartz norm. Thus one has $\left(\hat{B}\tilde{\Pi}_{1}-\tilde{\Pi}_{1}\right)\left(p,0\right)\in C^{\infty}$
and hence by almost analytic continuation
\begin{equation}
\left(\tilde{B}\tilde{\Pi}_{1}-\tilde{\Pi}_{1}\right)\left(x,0\right)\in C^{\infty}.\label{eq:BP=00003DP mod smooth}
\end{equation}
Next for each $N\in\mathbb{N}$ define the operator with kernel
\begin{align}
B_{N}\left(x,x'\right) & =\sum_{\left|\alpha\right|,\left|\beta\right|\leq N}\frac{1}{\alpha!\beta!}\left(-\frac{\partial}{\partial x'}\right)^{\alpha}\left[\left(iy\left(x'\right)\right)^{\alpha}\left(iq\left(x\right)\right)^{\beta}\left(\frac{\partial}{\partial p}\right)^{\beta}B\left(p\left(x\right),p'\left(x'\right)\right)\left|\frac{dp}{dx}\left(x'\right)\right|\right]\chi_{1}\left(x'\right)\nonumber \\
\textrm{where}\quad p\left(x\right) & =p\left(x,0\right)=x+O\left(x^{2}\right)\nonumber \\
q\left(x\right) & =q\left(x,0\right)=O\left(x^{2}\right)\nonumber \\
y\left(x\right) & =y\left(p\left(x,0\right),0\right)=O\left(x^{2}\right)\label{eq:BN def.}
\end{align}
denote the coordinates coming from the change of variables \prettyref{thm:almost analytic coordinates},
while the multiplication factor $\left|\frac{dp}{dx}\right|$ is the
Jacobian for the change of variables with respect to the first. Following
an integration by parts argument using \prettyref{eq:almost analytic continuattion},
\prettyref{eq:BP=00003DP mod smooth}, Prop. \ref{class is module over pso},
and Prop. \prettyref{prop:functional boundedness} which motivates
the construction of \prettyref{eq:BN def.}, it is easy to see that
\begin{equation}
\left(B_{N}\Pi-\Pi_{1}\right)\left(x,0\right)\in S_{H}^{\frac{2-N}{r}}\label{eq:BNP - P}
\end{equation}
$\forall N\in\mathbb{N}$. Furthermore writing \prettyref{eq:BN def.}
in the $p,p'$ coordinates gives
\begin{align}
B_{N+1}-B_{N} & \in\hat{L}_{\frac{1}{r},{\rm cl\,}}^{\frac{2-N}{r},N}+\hat{L}_{\frac{1}{r},2}^{\frac{2}{r},0,-\infty}\nonumber \\
\textrm{hence }\quad B_{\infty}\coloneqq & B_{0}+\sum_{N=0}^{\infty}\left(B_{N+1}-B_{N}\right)\in\hat{L}_{\frac{1}{r},{\rm cl\,}}^{\frac{2}{r},0}+\hat{L}_{\frac{1}{r},2}^{\frac{2}{r},0,-\infty}\label{eq: B infty def}
\end{align}
is well-defined by asymptotic summation. The last two equations \prettyref{eq:BNP - P},
\prettyref{eq: B infty def} then give
\begin{equation}
\left(B_{\infty}\Pi-\Pi\right)\left(x,0\right)\in C^{\infty}.\label{eq:BinftyP=00003DP}
\end{equation}

Next the Schwartz kernel of the $\mu=e^{g\left(x\right)}dx$ adjoint
$B^{*}$ of $B$ is calculated to be 
\begin{equation}
B^{*}\left(x,x'\right)=e^{g\left(x'\right)-g\left(x\right)}\chi_{2}\left(x'\right)\overline{\tilde{B}\left(p\left(x'\right),p\left(x\right)+iq\left(x\right)\right)}\chi_{1}\left(x\right).\label{eq:B* computation}
\end{equation}
Following \prettyref{thm:almost analytic coordinates}, a Taylor expansion,
and writing in $p,p'$ coordinates, it is easy to see $B^{*}\in\hat{L}_{\frac{1}{r},{\rm cl\,}}^{\frac{2}{r},0}+\hat{L}_{\frac{1}{r},2}^{\frac{2}{r},0,-\infty}$.
Furthermore from the above and \prettyref{eq:BN def.} it has the
same principal symbol as $B_{\infty}$ 
\begin{align}
\sigma_{L}\left(B^{*}\right) & =\sigma_{L}\left(B_{\infty}\right)\quad\textrm{hence }\nonumber \\
R\coloneqq B_{\infty}-B^{*} & \in\hat{L}_{\frac{1}{r},{\rm cl\,}}^{\frac{1}{r},1}+\hat{L}_{\frac{1}{r},2}^{\frac{2}{r},0,-\infty}.\label{eq:R lies in class}
\end{align}

Finally combining \prettyref{eq:BP=00003DB} and \prettyref{eq:BinftyP=00003DP}
we have
\begin{align}
\Pi\left(x,0\right) & =\left[B_{\infty}\Pi\right]\left(x,0\right)+C^{\infty}\nonumber \\
 & =\left(B^{*}+R\right)\Pi\left(x,0\right)\nonumber \\
 & =\left[B^{*}+R\Pi\right]\left(x,0\right)+C^{\infty}\quad\textrm{ hence }\label{eq:relation Pi, B, R}\\
\left[\left(I-R^{N}\right)\Pi\right]\left(x,0\right) & =\left[1+R+R^{2}+\ldots+R^{N-1}\right]B^{*}\left(x,0\right)\label{eq:Pi in terms B , R}
\end{align}
$\forall N\in\mathbb{N}$. Following \ref{closure under composition},
\prettyref{eq:R lies in class} we have $R^{j}\in\hat{L}_{\frac{1}{r},{\rm cl\,}}^{\left(2-j\right)\frac{1}{r},j}+\hat{L}_{\frac{1}{r},2}^{\frac{2}{r},0,-\infty}$,
$j=1,2,\ldots$. Hence by asymptotic summation $\exists P\in\hat{L}_{\frac{1}{r},{\rm cl\,}}^{\frac{2}{r},0}$
such that 
\begin{equation}
P_{N}\coloneqq P-\left[1+R+R^{2}+\ldots+R^{N-1}\right]\in\hat{L}_{\frac{1}{r},{\rm cl\,}}^{\left(2-N\right)\frac{1}{r},N}+\hat{L}_{\frac{1}{r},2}^{\frac{2}{r},0,-\infty}\label{eq:asymptotic summation}
\end{equation}
$\forall N\in\mathbb{N}$. Since $B^{*}\left(x,0\right)\in S_{H,{\rm cl\,}}^{\frac{2}{r}}$
by definition, we have 
\begin{equation}
\Pi\left(x,0\right)=\left[PB^{*}\right]\left(x,0\right)+S_{H}^{\frac{2-N}{r}}\label{eq:last step of proof}
\end{equation}
$\forall N\in\mathbb{N}$, from \prettyref{eq:christ-M-M-NRSW-bounds},
\prettyref{eq:Pi in terms B , R}, \prettyref{eq:asymptotic summation}
and Prop. \prettyref{prop:functional boundedness}. Choosing $N$
large gives $\Pi\left(x,0\right)\in S_{H,{\rm cl\,}}^{\frac{2}{r}}$
using \prettyref{eq:regularity of spaces} and Prop. \prettyref{prop:functional boundedness}
which completes the proof on account of \prettyref{eq:x3 y independent classes}
and \prettyref{eq:inclusion into Hormander class}. 
\end{proof}
The next remark shows that our parametrix \prettyref{thm:main thm parametrix}
recovers the Boutet de Monvel-Sjöstrand parametrix at strongly pseudoconvex
points. 
\begin{rem}
\label{rem:-(Strongly-pseudoconvex} (Strongly pseudoconvex points)
Here we show that our main \prettyref{thm:main thm parametrix} recovers
the Boutet de Monvel-Sjöstrand description of the Szeg\H{o} kernel
at strongly pseudoconvex points $x'\in X$. As noted before, the type
of a strongly pseudoconvex point is $r_{x'}=2$. The two degree $2$
homogeneous polynomials in \prettyref{eq:Christ normal form} and
\prettyref{thm:almost analytic coordinates} can be further taken
to be $p\left(x_{1},x_{2}\right)=x_{1}^{2}+x_{2}^{2}$, $\varphi_{0}\left(\tilde{z}_{1},\tilde{z}_{2}\right)=\tilde{z}_{1}^{2}+\tilde{z}_{2}^{2}$
respectively. Following these, the model Bergman kernel is computed
to be an exponential $B_{t}^{0}\left(\hat{p},\hat{p}'\right)=tb_{0}\left(t^{\frac{1}{2}}p,t^{\frac{1}{2}}p'\right)$
\begin{align}
B_{t}^{0}\left(\hat{p},\hat{p}'\right) & =tb_{0}\left(t^{\frac{1}{2}}p,t^{\frac{1}{2}}p'\right)=te^{-t\Phi_{0}\left(\hat{p},\hat{p}'\right)}\nonumber \\
b_{0}\left(\hat{p},\hat{p}'\right) & =e^{-\Phi_{0}\left(\hat{p},\hat{p}'\right)}\nonumber \\
\Phi_{0}\left(\hat{p},\hat{p}'\right) & \coloneqq\frac{1}{4}\left(p_{1}^{2}+p_{2}^{2}+\left(p'_{1}\right)^{2}+\left(p'_{2}\right)^{2}+2p_{1}p'_{1}+2p_{2}p'_{2}+2ip'_{1}p_{2}-2ip_{1}p'_{2}\right)\label{eq:model bergman spc}
\end{align}
\cite[Sec. 4.1.6]{Ma-Marinescu}. And hence 
\[
\tilde{B}_{t}\left(p,p'\right)=te^{-t\left[\Phi_{0}\left(\hat{p},\hat{p}'\right)+\varphi_{1}\left(\hat{p}\right)-\varphi_{1}\left(\hat{p}'\right)\right]}
\]

Next the local Bergman kernel $B_{t}$ \prettyref{eq:local Bergman}
is by \ref{e-gue200512yyda} modulo $C^{N}$ a finite a sum of terms
of the form 
\[
\tilde{B}_{t}^{*}\left(\tilde{B}_{t}\tilde{B}_{t}^{*}\right)^{k}\quad\textrm{ or }\quad\left(\tilde{B}_{t}\tilde{B}_{t}^{*}\right)^{k}.
\]
Applying the complex stationary phase formula of Melin-Sjöstrand \cite[Sec. 2]{Melin-Sjostrand75},
the kernels of the above take the form $a_{k}\left(\hat{p}',\hat{p},t\right)e^{-t\Phi_{k}\left(\hat{p}',\hat{p}\right)}$,
where $a_{k}\in S_{0,\textrm{cl}}^{1}\left(\mathbb{R}_{p,p'}^{4}\times\mathbb{R}_{t}\right)$
is a classical symbol and $\Phi_{k}=\Phi_{0}+O\left(\left|\left(p,p'\right)\right|^{2}\right)$
a phase function agreeing with \prettyref{eq:model bergman spc} at
leading order. To get the Boutet de Monvel-Sjöstrand description however
one needs to ensure that all phase functions $\Phi_{k}$ agree. To
this end, one may replace $\tilde{B}_{t}$ \prettyref{eq:approximate local bergman kernel}
with the alternate approximation for the local Bergman kernel $B_{t}$
given by
\begin{align*}
\tilde{B}_{t}^{1}\left(\hat{p},\hat{p}'\right) & \coloneqq e^{t\Phi_{1}\left(\hat{p},\hat{p}'\right)}B_{t}^{0}\left(\hat{p},\hat{p}'\right)\\
\Phi_{1}\left(\hat{p},\hat{p}'\right) & \coloneqq\varphi_{1}\left(\hat{p}\right)+\varphi_{1}\left(\hat{p}'\right)-2\sum_{\alpha,\beta}\left(\partial_{\zeta}^{\alpha}\partial_{\bar{\zeta}}^{\beta}\varphi_{1}\right)\left(0\right)\frac{\zeta^{\alpha}\bar{\zeta'}^{\beta}}{\alpha!\beta!},
\end{align*}
$\zeta\coloneqq p_{1}+ip_{2}$. The proof of Thm. \prettyref{thm:local Bergman is a symbol},
\prettyref{thm:main thm parametrix} all carry through with $\tilde{B}_{t}$
replaced by $\tilde{B}_{t}^{1}$. The above further has the advantage
of being self-adjoint 
\begin{align*}
\tilde{B}_{t}^{1}\left(\hat{p},\hat{p}'\right) & =\overline{\tilde{B}_{t}^{1}\left(\hat{p}',\hat{p}\right)}\quad\textrm{ and equals }\\
\tilde{B}_{t}^{1}\left(\hat{p},\hat{p}'\right) & =e^{t\left[\Phi_{0}+\Phi_{1}\right]}
\end{align*}
in the strongly pseudoconvex case again using \prettyref{eq:model bergman spc}.
Furthermore, the composition of complex Fourier integral operators
and the complex stationary phase formula of Melin-Sjöstrand \cite[Sec. 2]{Melin-Sjostrand75}
in this case gives 
\begin{equation}
\left(\tilde{B}_{t}^{1}\right)^{2}=a\left(\hat{p},\hat{p}',t\right)e^{-t\Phi\left(\hat{p},\hat{p}'\right)}\label{eq:Boutet Sjostrand amplitude}
\end{equation}
with the same phase function $\Phi$ for $a\in S_{0,\textrm{cl}}^{1}\left(\mathbb{R}_{p,p'}^{4}\times\mathbb{R}_{t}\right)$
a classical symbol. Following this and repeating the argument for
\prettyref{thm:local Bergman is a symbol} with $\tilde{B}_{t}$ replaced
by $\tilde{B}_{t}^{1}$, the equations \prettyref{eq:approximate local bergman kernel},
\ref{e-gue200512yydII} and \ref{e-gue200512yyda} are seen to give
a similar form as \prettyref{eq:Boutet Sjostrand amplitude} for the
local Bergman kernel $B_{t}=a\left(\hat{p},\hat{p}',t\right)e^{-t\Phi\left(\hat{p},\hat{p}'\right)}$,
$a\in S_{0,\textrm{cl}}^{1}\left(\mathbb{R}_{p,p'}^{4}\times\mathbb{R}_{t}\right)$.
Plugging this form for the local Bergman kernel into the equations
\prettyref{eq:the operator B}, \prettyref{eq:BN def.}, \prettyref{eq:B* computation},
\prettyref{eq:Pi in terms B , R} and \prettyref{eq:last step of proof}
within the proof of \prettyref{thm:main thm parametrix}, and another
use of the Melin-Sjöstrand formula gives 
\begin{equation}
\Pi\left(x,0\right)=\int_{0}^{\infty}dt\,a\left(\hat{p},t\right)e^{itp_{3}-t\Phi\left(\hat{p},0\right)}\label{eq:Boutet de Monvel sjostrnd parametrix}
\end{equation}
for some $a\in S_{0,\textrm{cl}}^{1}\left(\mathbb{R}_{\hat{p}}^{2}\times\mathbb{R}_{t}\right)$
which is the pointwise version of the Boutet de Monvel-Sjöstrand form
for the parametrix at strongly pseudoconvex points.

We finally note that the reduction to the form \prettyref{eq:Boutet de Monvel sjostrnd parametrix}
above is possible on account of the explicit knowledge of the model
Bergman kernel $B_{t}^{0}$ \prettyref{eq:model bergman spc}, related
to Mehler's formula for the harmonic oscillator $\boxdot_{t}^{0}$,
at a strongly pseudoconvex point. At points of higher type the model
kernel to contend with is less explicit, modeled on anharmonic oscillators,
and one has to live with the description \prettyref{eq:Szego parametrix}.
\end{rem}

\section{\label{sec:Pseudoconvex-domains} Pseudoconvex domains}

We now consider the special case when the CR manifold is the boundary
of a domain $D$ in $\mathbb{C}^{2}$. Thus $D\subset\mathbb{C}^{2}$
is a relatively compact open subset with smooth boundary $X=\partial D$.
The CR structure on the boundary is simply obtained by restriction
$T^{1,0}X=T^{1,0}\mathbb{C}^{2}\cap T_{\mathbb{C}}X$ of the complex
tangent space of $\mathbb{C}^{2}$. 

We fix a Hermitian metric $\langle\,\cdot\,|\,\cdot\,\rangle$ on
$\mathbb{C}T\mathbb{C}^{2}$ so that $T^{1,0}\mathbb{C}^{2}\perp T^{0,1}\mathbb{C}^{2}$.
The Hermitian metric $\langle\,\cdot\,|\,\cdot\,\rangle$ on $\mathbb{C}T\mathbb{C}^{2}$
induces by duality, Hermitian metrics $\langle\,\cdot\,|\,\cdot\,\rangle$
on $\oplus_{0\leq p,q\leq2}T^{*p,q}\mathbb{C}^{2}$, where $T^{*p,q}\mathbb{C}^{2}$
denote the bundles of $(p,q)$ forms. With $dv$ being the induced
volume form on $\mathbb{C}^{2}$ let $(\,\cdot\,|\,\cdot\,)_{D}$
and $(\,\cdot\,|\,\cdot\,)_{\mathbb{C}^{2}}$ be the inner products
on $\Omega^{0,q}(\overline{D})$ and $\Omega_{0}^{0,q}(\mathbb{C}^{2})$
defined by 
\begin{equation}
\begin{split} & (\,f\,|\,h\,)_{D}=\int_{D}\langle\,f\,|\,h\,\rangle dv,\ \ f,h\in\Omega^{0,q}(\overline{D}),\\
 & (\,f\,|\,h\,)_{\mathbb{C}^{2}}=\int_{\mathbb{C}^{2}}\langle\,f\,|\,h\,\rangle dv,\ \ f,h\in\Omega_{0}^{0,q}(\mathbb{C}^{2}).
\end{split}
\label{e-gue200418yyd}
\end{equation}
Also denote by $\|\cdot\|_{D}$ and $\|\cdot\|_{\mathbb{C}^{2}}$
be the corresponding norms and by $L^{2}\left(D\right),\,L_{(0,q)}^{2}\left(D\right)$
the corresponding spaces of square integrable functions. Let $\rho\in C^{\infty}(\mathbb{C}^{2},\mathbb{R})$
be a defining function of $X$ satisfying $\rho=0$ on $X$, $\rho<0$
on $D$ and $\left.d\rho\right|_{X}\neq0$ . This maybe further chosen
to satisfy $\|d\rho\|=1$ on $X$. 

Let $\bar{\partial}:\Omega^{0,q}(\mathbb{C}^{2})\rightarrow\Omega^{0,q+1}(\mathbb{C}^{2})$
be the exterior differential operator and consider its formal adjoint
\begin{align*}
\overline{\partial}_{f}^{*}: & \Omega^{0,1}(\mathbb{C}^{2})\rightarrow C^{\infty}(\mathbb{C}^{2})\quad\textrm{satisfying}\\
(\,\bar{\partial}f\,|\,h\,)_{\mathbb{C}^{2}} & =(f\,|\,\overline{\partial}_{f}^{*}h\,)_{\mathbb{C}^{2}},\quad f\in C_{c}^{\infty}(\mathbb{C}^{2}),\;h\in\Omega^{0,1}(\mathbb{C}^{2}).
\end{align*}
 Also denote by $\bar{\partial}^{*}:L_{(0,1)}^{2}(D)\rightarrow L^{2}(D)$
the $L^{2}$ adjoint of $\bar{\partial}$, as an unbounded operator,
with respect to $(\,\cdot\,|\,\cdot\,)_{D}$. The Bergman kernel of
the domain is the distributional kernel $\Pi_{D}\left(z,z'\right)\in C^{-\infty}(D\times D)$
of the orthogonal projection
\[
\Pi_{D}:L^{2}(D)\rightarrow{\rm Ker\,}\bar{\partial}\subset L^{2}(D)
\]
with respect to $(\,\cdot\,|\,\cdot\,)_{D}$. The goal of this section
is to establish an asymptotic expansion for $\Pi_{D}\left(z,z\right)$
as $z\rightarrow x'$ approaches a point on the boundary $x'\in X$. 

This shall use the relation of the Bergman kernel with the Szeg\H{o}
kernel of the boundary via the Poisson operator \cite[Sec. 3b]{Boutet-Sjostrand76}.
To state this, let 
\[
\Box_{f}=\bar{\partial}\,\bar{\partial}_{f}^{*}+\bar{\partial}_{f}^{*}\,\bar{\partial}:C^{\infty}(\mathbb{C}^{2})\rightarrow C^{\infty}(\mathbb{C}^{2})
\]
denote the complex Laplace-Beltrami operator on functions and $\gamma$
the operator of restriction to the boundary $X$. The Poisson operator
is the solution operator to the Dirichlet problem on $D$ defined
by
\begin{align}
P: & C^{\infty}(X)\rightarrow C^{\infty}(\overline{D})\label{e-gue200420yydIq}\\
\Box_{f}Pu=0, & \ \ \gamma Pu=u,\ \ \text{ }\forall u\in C^{\infty}(X).\label{eq:e-gue200420yydI}
\end{align}
Its adjoint is defined to satisfy
\begin{align}
P^{*}:C^{\infty}(\overline{D}) & \rightarrow C^{-\infty}(X)\label{e-gue200420yydb}\\
(\,P^{*}u\,|\,v\,)_{X} & =(\,u\,|\,Pv\,)_{D},\ \ u\in C^{\infty}(\overline{D}),\ v\in C^{\infty}(X)
\end{align}

The microlocal structure of $P$ was described by Boutet de Monvel
\cite{Boutet1971}. Firstly, from \cite[pg. 29]{Boutet1971} the operators
$P,P^{*}$ extend continuously 
\begin{align*}
P:H^{s}(X) & \rightarrow H^{s+\frac{1}{2}}(\overline{D}),\\
P^{*}:H^{s}(\overline{D}) & \rightarrow H^{s+\frac{1}{2}}(X),\ \forall s\in\mathbb{R},
\end{align*}
and in particular map smooth functions onto smooth ones. Furthermore,
$P^{*}P:C^{\infty}(X)\rightarrow C^{\infty}(X)$ is an injective continuous
operator and its inverse $\left(P^{*}P\right)^{-1}$ is a classical
elliptic pseudodifferential operator of order one on $X$. Its principal
symbol is given by 
\begin{equation}
\sigma_{\left(P^{*}P\right){}^{-1}}=\sigma_{\left(2\sqrt{-\triangle_{X}}\right)},\label{e-gue200519yyd}
\end{equation}
\cite{Hsiao2010} where $\sigma_{\left(2\sqrt{-\triangle_{X}}\right)}$
denotes the principal symbol of the square root of the Laplace-Beltrami
operator. Next, there is a continuous operator 
\begin{align}
G: & H^{s}(\overline{D})\rightarrow H^{s+2}(\overline{D}),\ \forall s\in\mathbb{R},\quad\textrm{ satisfying }\label{eq:Green operator for Poisson}\\
G\Box_{f}+P\gamma & =I\ \ \mbox{on \ensuremath{C^{\infty}(\overline{D})}}.\label{eq: defining property green operator}
\end{align}
It furthermore follows from the methods of \cite{Boutet1971} that
the Schwartz kernel of $G$ satisfies the estimates 
\begin{equation}
\left|\partial_{z}^{\alpha}\partial_{w}^{\beta}G\left(z,w\right)\right|\leq C\left|z-w\right|^{-5-\left|\alpha\right|-\left|\beta\right|}\label{eq:Green function Schwartz estimates}
\end{equation}
$\forall z,w\in\overline{D}$, $\alpha,\beta\in\mathbb{N}_{0}^{4}$,
along the diagonal.

With 
\begin{align*}
\Gamma^{\wedge}: & T^{*0,q}\mathbb{C}^{2}\rightarrow T^{*0,q+1}\mathbb{C}^{2}\\
\Gamma^{\wedge,*}: & T^{*0,q+1}\mathbb{C}^{2}\rightarrow T^{*0,q}\mathbb{C}^{2},\quad\forall\Gamma\in T^{*0,1}\mathbb{C}^{2},
\end{align*}
denoting the wedge and contraction, adjoint with respect to $\langle\,\cdot\,|\,\cdot\,\rangle$,
operators, one has
\begin{equation}
\begin{split}I & =2(\bar{\partial}\rho)^{\wedge}(\bar{\partial}\rho)^{\wedge,*}+2(\bar{\partial}\rho)^{\wedge,*}(\bar{\partial}\rho)^{\wedge},\ \ \mbox{on \ensuremath{\Omega^{0,q}(\mathbb{C}^{2})}},\\
\bar{\partial}_{b} & =2\gamma(\bar{\partial}\rho)^{\wedge,*}(\bar{\partial}\rho)^{\wedge}\bar{\partial}P\ \ \mbox{on \ensuremath{C^{\infty}(X)}}.
\end{split}
\label{expression db}
\end{equation}

In using the above to describe the behavior of $\Pi_{D}$ near a boundary
point $x'\in X$ one uses the parametrix construction for the Szeg\H{o}
kernel $\Pi$ from \prettyref{thm:main thm parametrix}. Let $(x_{1},x_{2},x_{3})$
be the local coordinates on an open set $x'\in U\subset X$ on the
boundary as in the proof of Theorem~\ref{thm:main thm parametrix}
with $B$ the operator \prettyref{eq:the operator B} therein. Since
$B$ is smoothing away the diagonal we may again assume that $B$
is properly supported on $U$. It is well-known that $\Box_{b}$ is
elliptic outside its characteristic variety $\Sigma\coloneqq HX^{\perp}\subset T^{*}X$
given by the annihilator of the Levi distribution $HX$. The characteristic
variety $\Sigma$ carries an orientation given by $J^{t}d\rho$, with
$J^{t}$ denoting the dual complex structure on $T^{*}\mathbb{C}^{2}$,
and we denote by $\Sigma^{-}$ its negatively oriented part. By construction
\prettyref{eq:B hat definition} we have 
\begin{align*}
WF\left(Bu\right) & \subset\Sigma^{-}\cap T^{*}U,\quad\forall u\in C^{-\infty}\left(U\right).
\end{align*}
Let $\tilde{U}$ be an open set of $\mathbb{C}^{2}$ with $\tilde{U}\cap\overline{D}=U$.
We then have the next lemma.
\begin{lem}
\label{l-gue200419yydI} The Schwartz kernel $\bar{\partial}PB\left(z,x\right)\in C^{\infty}\left(\left(\tilde{U}\times U\right)\cap\left(\overline{D}\times X\right)\right)$ 
\end{lem}
\begin{proof}
From \eqref{expression db}, we have 
\[
2\gamma(\bar{\partial}\rho)^{\wedge,*}(\bar{\partial}\rho)^{\wedge}\bar{\partial}PB=\bar{\partial}_{b}B\in C^{\infty}.
\]
Combining this with $\left(\bar{\partial}_{f}^{*}\bar{\partial}\right)P=0$,
$P\gamma\bar{\partial}P=\bar{\partial}P$, we have 
\begin{align}
0 & =\bar{\partial}_{f}^{*}\bar{\partial}PB\nonumber \\
 & =\bar{\partial}_{f}^{*}P\gamma\bar{\partial}PB\nonumber \\
 & =\bar{\partial}_{f}^{*}P\gamma\left(I-2(\bar{\partial}\rho)^{\wedge,*}(\bar{\partial}\rho)^{\wedge}\right)\bar{\partial}PB+C^{\infty}\nonumber \\
 & =\bar{\partial}_{f}^{*}P\gamma\left(2(\bar{\partial}\rho)^{\wedge}(\bar{\partial}\rho)^{\wedge,*}\right)\bar{\partial}PB+C^{\infty}.\label{eq:e-gue200419ycdII}
\end{align}
From \prettyref{eq:e-gue200419ycdII}, we deduce that 
\[
\gamma(\bar{\partial}\rho)^{\wedge}\bar{\partial}_{f}^{*}P\gamma(2(\bar{\partial}\rho)^{\wedge}(\bar{\partial}\rho)^{\wedge,*})\bar{\partial}PB\in C^{\infty}.
\]
It is known that \cite[Propsition 4.2]{Hsiao2010} that 
\[
\gamma(\bar{\partial}\rho)^{\wedge}\bar{\partial}_{f}^{*}P:C^{\infty}(X,I^{0,2}T^{*}\mathbb{C}^{2})\rightarrow C^{\infty}(X,I^{0,2}T^{*}\mathbb{C}^{2})
\]
is elliptic near\, $\Sigma^{-}$, where $I^{0,2}T^{*}\mathbb{C}^{2}$
is the vector bundle over $\mathbb{C}^{2}$ with fiber 
\[
I^{0,2}T^{*}\mathbb{C}^{2}=\left\{ (\bar{\partial}\rho)(z)\wedge g;\,g\in T_{z}^{*0,1}\mathbb{C}^{2}\right\} .
\]
Since ${\rm WF\,}\left(Bu\right)\subset\Sigma^{-}\cap T^{*}U$, $u\in C^{-\infty}\left(U\right)$,
we get $\gamma(2(\bar{\partial}\rho)^{\wedge}(\bar{\partial}\rho)^{\wedge,*})\bar{\partial}PBu$
is smooth and hence 
\[
\gamma(2(\bar{\partial}\rho)^{\wedge}(\bar{\partial}\rho)^{\wedge,*})\bar{\partial}PB\in C^{\infty}
\]
as required. 
\end{proof}
In view of \cite[Lemmas 4.1, 4.2]{Hsiao2010}, we see that $P$ and
$(P^{*}P)^{-1}P^{*}$ are smoothing away from the diagonal. Hence,
they maybe replaced by continuous properly supported operators 
\[
\begin{split} & L:C_{c}^{\infty}(\tilde{U}\cap\overline{D})\rightarrow C_{c}^{\infty}(U),\\
 & \hat{P}:C_{c}^{\infty}(U)\rightarrow C_{c}^{\infty}(\tilde{U}\cap\overline{D})
\end{split}
\]
such that 
\begin{equation}
\begin{split} & L-(P^{*}P)^{-1}P^{*}\equiv0\mod C^{\infty}((U\times\tilde{U})\cap(X\times\overline{D})),\\
 & \hat{P}-P\equiv0\mod C^{\infty}((\tilde{U}\times U)\cap(\overline{D}\times X)).
\end{split}
\label{e-gue190524syds}
\end{equation}
We now set
\begin{equation}
A:=\hat{P}BL:C^{\infty}(\tilde{U}\cap\overline{D})\rightarrow C^{\infty}(\tilde{U}\cap\overline{D}).\label{e-gue200419ycdhq}
\end{equation}
From Lemma~\ref{l-gue200419yydI}, we see that 
\begin{equation}
\bar{\partial}A\equiv0\mod C^{\infty}((\tilde{U}\times\tilde{U})\cap(\overline{D}\times\overline{D})).\label{e-gue200419yydi}
\end{equation}
Since the boundary $X$ is of finite type, it was proved by Kohn \cite{Kohn-dbarN-72}
that there is a pseudolocal continuous operator 
\[
\begin{split} & N:L^{2}(M)\rightarrow L_{(0,1)}^{2}(M)\cap{\rm Dom\,}\bar{\partial}^{*},\\
 & N:C^{\infty}(\overline{D})\rightarrow\Omega^{0,1}(\overline{D}),
\end{split}
\]
such that 
\begin{equation}
\Pi_{D}=I-\bar{\partial}^{*}N\bar{\partial}.\label{e-gue200419ycdg}
\end{equation}
From \eqref{e-gue200419yydi} and \eqref{e-gue200419ycdg}, we deduce
that 
\begin{equation}
\Pi_{D}A\equiv A\quad\mod C^{\infty}((\tilde{U}\times\tilde{U})\cap(\overline{D}\times\overline{D})).\label{PiA =00003DA}
\end{equation}
Next we take $\tilde{U}$ small enough enough so that $z=\left(x,\rho\right)$
form local coordinates of $\tilde{U}$. 

The Bergman kernel $\Pi_{D}$ is then known to satisfy the bounds
\begin{equation}
\left|\partial_{\rho}^{\gamma}\partial_{\rho'}^{\gamma'}\partial_{x}^{\alpha}\Pi_{D}\left(\left(x,\rho\right),\left(0,\rho'\right)\right)\right|\leq C_{\alpha\gamma\gamma'}\left(\left|\rho\right|+\left|\rho'\right|+d^{H}\left(x\right)\right)^{-w.\alpha-\gamma-\gamma'-r\left(x'\right)-2},\label{eq:Bergman kernel estimates}
\end{equation}
 $\forall\left(\alpha,\gamma,\gamma'\right)\in\mathbb{N}_{0}^{5},$
similar to \prettyref{eq: function class} (see \cite{McNeal89,Nagel-Rosay-Stein-Wainger-89}).
This gives corresponding estimates for the kernel $\gamma\Pi_{D}\left(\left(x,\rho\right),\left(0,\rho'\right)\right)$
which satisfies $\bar{\partial}_{b}\gamma\Pi_{D}=0$. Following these,
we can repeat the procedure in the proof of Theorem~\ref{thm:main thm parametrix}
to conclude 
\begin{equation}
(B_{\infty}\gamma\Pi_{D})((x,0),(0,\rho'))\equiv\gamma\Pi_{D}((x,0),(0,\rho'))\mod C^{\infty}((\tilde{U}\times\tilde{U})\cap(X\times\overline{D})),\label{e-gue200424yydI}
\end{equation}
as with eqn. \prettyref{eq:BinftyP=00003DP}.Next $\Pi_{D}=P\gamma\Pi_{D}$
gives $P^{*}\Pi_{D}=P^{*}P\gamma\Pi_{D}$ and hence $\left(P^{*}P\right)^{-1}P^{*}\Pi_{D}=\gamma\Pi_{D}$.
This combines with \eqref{e-gue200424yydI} to give
\begin{align}
(A_{\infty}\Pi_{D})(z,(0,\rho')) & \equiv\Pi_{D}(z,(0,\rho'))\mod C^{\infty}((\tilde{U}\times\tilde{U})\cap(\overline{D}\times\overline{D}))\label{e-gue200424yydII}\\
\textrm{ for }\quad A_{\infty} & \coloneqq\hat{P}B_{\infty}L.\nonumber 
\end{align}

We now have the following proposition.
\begin{lem}
\label{Bergman kernel in terms of Poisson} One has 
\[
\Pi_{D}(z,(0,\rho))=(PQP^{*})(z,(0,\rho))+C^{\infty}\left(\ensuremath{\tilde{U}\times\mathbb{R}_{\rho}}\right)
\]
for some properly supported $Q\in\hat{L}_{\frac{1}{r},{\rm cl\,}}^{1+\frac{2}{r}}$.
\end{lem}
\begin{proof}
Denote by 
\begin{equation}
A^{*}:C_{c}^{\infty}(\tilde{U}\cap\overline{D})\rightarrow C^{\infty}(\tilde{U}\cap\overline{D}),\ \label{e-gue200513yyd}
\end{equation}
be the adjoint of $A$ with respect to $(\,\cdot\,|\,\cdot\,)$. From
\eqref{PiA =00003DA}, we have 
\begin{equation}
A^{*}\Pi_{D}\equiv A^{*}\mod C^{\infty}((\tilde{U}\times\tilde{U})\cap(\overline{D}\times\overline{D})).\label{e-gue200519ycd}
\end{equation}
Thus 
\begin{equation}
\begin{split} & (A_{\infty}\Pi_{D})(z,(0,\rho'))\\
 & =(A^{*}\Pi_{D})(z,(0,\rho'))+((A_{\infty}-A^{*})\Pi_{D})(z,(0,\rho'))\\
 & \equiv A^{*}(z,(0,\rho'))+(R\Pi_{D})(z,(0,\rho'))\mod C^{\infty}((\tilde{U}\times\tilde{U})\cap(\overline{D}\times\overline{D})),
\end{split}
\label{e-gue200424yydIII}
\end{equation}
for 
\begin{align}
R & \coloneqq A_{\infty}-A^{*}\nonumber \\
 & =PB_{\infty}(P^{*}P)^{-1}P^{*}-P(P^{*}P)^{-1}B^{*}P^{*}\mod C^{\infty}((\tilde{U}\times\tilde{U})\cap(\overline{D}\times\overline{D}))\nonumber \\
 & =P(P^{*}P)^{-1}\underbrace{\Bigr((P^{*}P)B_{\infty}(P^{*}P)^{-1}-B^{*}\Bigr)}_{E\coloneqq}P^{*}\mod C^{\infty}((\tilde{U}\times\tilde{U})\cap(\overline{D}\times\overline{D})).\label{eq:calculation remainder}
\end{align}
From \eqref{e-gue200424yydII} and \eqref{e-gue200424yydIII}, we
get 
\[
(I-R)\Pi_{D}(z,(0,\rho'))\equiv A^{*}(z,(0,\rho))\mod C^{\infty}((\tilde{U}\times\mathbb{R}_{\rho})\cap(\overline{D}\times\mathbb{R}_{\rho}))
\]
\begin{equation}
\Bigr((I-R^{N})\Pi_{D}\Bigr)(z,(0,\rho'))\equiv\Bigr((I+R+R^{2}+\cdots+R^{N-1})A^{*}\Bigr)(z,(0,\rho))\mod C^{\infty}((\tilde{U}\times\mathbb{R}_{\rho})\cap(\overline{D}\times\mathbb{R}_{\rho})).\label{important calculation}
\end{equation}
$\forall N\in\mathbb{N}$. 

From \prettyref{eq:symbol adjoint} and \prettyref{eq:symbol of product with pdo}
one has $E\in\hat{L}_{\frac{1}{r},{\rm cl\,}}^{\frac{1}{r},1}$ and
thus $E^{N}\in\hat{L}_{\frac{1}{r},{\rm cl\,}}^{(2-N)\frac{1}{r},N}$,
$\forall N\in\mathbb{N}$. By Prop. \ref{class is module over pso},
Prop. \ref{closure under composition} and asymptotic summation $\exists Q\in\hat{L}_{\frac{1}{r},{\rm cl\,}}^{1+\frac{2}{r}}$
such that 
\begin{align*}
Q & -(P^{*}P)^{-1}\Bigr((I+E+E^{2}+\cdots E^{N})B_{1}^{*}\Bigr)\in\hat{L}_{\frac{1}{r},{\rm cl\,}}^{1+\left(2-N\right)\frac{1}{r},N}.
\end{align*}
Thus for each $l$, $\exists N_{l}\in\mathbb{N}$ such that
\begin{align}
 & PQP^{*}-\Bigr((I+R+R^{2}+\cdots+R^{N-1})A^{*}\Bigr)(z,(0,\rho))\nonumber \\
= & P\left[Q-(P^{*}P)^{-1}\Bigr((I+E+E^{2}+\cdots+E^{N-1})B_{1}^{*}\Bigr)\right]P^{*}\in C^{l}((\tilde{U}\times\mathbb{R}_{\rho})\cap(\overline{D}\times\mathbb{R}_{\rho}))\label{eq:penultimate step proof}
\end{align}
$\forall N\geq N_{l}$. Finally from the kernel estimates \prettyref{eq:Bergman kernel estimates},
for each $l\geq0$, $\exists N_{l}'\in\mathbb{N}$ such that 
\begin{equation}
(R^{N}\Pi_{D})((x,0),(0,\rho))\in C^{\ell}((\tilde{U}\times\mathbb{R}_{\rho})\cap(\overline{D}\times\mathbb{R}_{\rho})).\label{last step proof}
\end{equation}
 $\forall N\geq N_{l}'$. From \ref{important calculation}, \prettyref{eq:penultimate step proof}
and \ref{last step proof} the Lemma follows.
\end{proof}
Next let 
\[
\Box_{f}=\bar{\partial}\,\bar{\partial}_{f}^{*}+\bar{\partial}_{f}^{*}\,\bar{\partial}:C^{\infty}(\mathbb{C}^{2})\rightarrow C^{\infty}(\mathbb{C}^{2})
\]
denote the complex Laplace-Beltrami operator on functions and we denote
$q_{0}$ the principle symbol of $\Box_{f}$. Repeating the proof
of \cite[Prop. 7.6]{Hsiao2010} one has the following.
\begin{lem}
\label{l-gue200522ycdh} There exists a smooth function $\phi(z,y)\in C^{\infty}((\tilde{U}\times U)\cap(\overline{D}\times X))$
such that 
\begin{equation}
\begin{split} & \phi(x,y)=x_{3}-y_{3},\\
 & \phi(z,y)=x_{3}-y_{3}-i\rho\sqrt{-\sigma_{\triangle_{X}}(x,(0,0,1))}+O(\left|\rho\right|^{2}),\\
 & \mbox{\ensuremath{q_{0}(z,d_{z}\phi)} vanishes to infinite order on \ensuremath{\rho=0}},
\end{split}
\label{eq: def of phase}
\end{equation}
where ${\rm Re\,}\sqrt{-\sigma_{\triangle_{X}}(x,(0,0,1))}>0$, $z=(x,\rho)$. 
\end{lem}
For our next result we shall need an extension of the symbol spaces
Definition \prettyref{def:symbol class}. Namely one may similarly
define the classes 
\begin{align}
 & \hat{S}_{\frac{1}{r}}^{m}\left(\mathbb{C}^{2}\times\mathbb{R}^{3}\times\mathbb{R}_{t}\right),\;\hat{S}_{\frac{1}{r}}^{m,k}\left(\mathbb{C}^{2}\times\mathbb{R}^{3}\times\mathbb{R}_{t}\right),\;\hat{S}_{\frac{1}{r},{\rm cl\,}}^{m,k}\left(\mathbb{C}^{2}\times\mathbb{R}^{3}\times\mathbb{R}_{t}\right)\label{eq:classes with rhp}\\
 & \hat{S}_{\frac{1}{r}}^{m}\left(\mathbb{C}^{2}\times\mathbb{C}^{2}\times\mathbb{R}_{t}\right),\;\hat{S}_{\frac{1}{r}}^{m,k}\left(\mathbb{C}^{2}\times\mathbb{C}^{2}\times\mathbb{R}_{t}\right),\;\hat{S}_{\frac{1}{r},{\rm cl\,}}^{m,k}\left(\mathbb{C}^{2}\times\mathbb{C}^{2}\times\mathbb{R}_{t}\right)\label{eq:classes with rho =000026 rho'}
\end{align}
as functions depending on additional $\rho$ or $\rho,\rho'$ variables.
The additional variables appear in a fashion similar to the $x_{3},y_{3}$
variables in the symbolic estimates and expansions. That is the equations
\ref{symbolic estimates}, \prettyref{eq:S^mk def.} and \prettyref{eq:symbolic expansion}
are replaced by 
\begin{align}
\left|\partial_{x}^{\alpha}\partial_{\rho}^{\alpha_{4}}\partial_{y}^{\beta}\partial_{t}^{\gamma}a(x,y,t)\right| & \leq C_{N,\alpha\beta\gamma}\left\langle t\right\rangle ^{m-\gamma+\frac{1}{r}\left(\left|\hat{\alpha}\right|+\left|\hat{\beta}\right|\right)+\alpha_{3}+\alpha_{4}+\beta_{3}}\frac{\left(1+\left|t^{\frac{1}{r}}\hat{x}\right|+\left|t^{\frac{1}{r}}\hat{y}\right|\right)^{N\left(\alpha,\beta,\gamma\right)}}{\left(1+\left|t^{\frac{1}{r}}\hat{x}-t^{\frac{1}{r}}\hat{y}\right|\right)^{-N}},\label{symbolic estimates-1}\\
 & \quad\quad\quad\quad\qquad\qquad\forall\left(x,y,t,N\right)\in\mathbb{R}_{x,y}^{6}\times\mathbb{R}_{t}\times\mathbb{N}.\nonumber 
\end{align}
\begin{equation}
\hat{S}_{\frac{1}{r}}^{m,k}\coloneqq\bigoplus_{p+q+p'\leq k}\left(tx_{3}\right)^{p}\left(t\rho\right)^{q}\left(ty_{3}\right)^{p'}\hat{S}_{\frac{1}{r}}^{m},\quad\forall\left(m,k\right)\in\mathbb{R}\times\mathbb{N}_{0}.\label{eq:S^mk def.-1}
\end{equation}
\begin{equation}
a\left(x,y,t\right)-\sum_{j=0}^{N}\sum_{p+q+p'\leq j}t^{m-\frac{1}{r}j}\left(tx_{3}\right)^{p}\left(t\rho\right)^{q}\left(ty_{3}\right)^{p'}a_{jpp'}\left(t^{\frac{1}{r}}\hat{x},t^{\frac{1}{r}}\hat{y}\right)\in\hat{S}_{\frac{1}{r}}^{m-\left(N+1\right)\frac{1}{r},N+1}\label{eq:symbolic expansion-1}
\end{equation}
in defining $\hat{S}_{\frac{1}{r}}^{m}\left(\mathbb{C}^{2}\times\mathbb{R}^{3}\times\mathbb{R}_{t}\right),\;\hat{S}_{\frac{1}{r}}^{m,k}\left(\mathbb{C}^{2}\times\mathbb{R}^{3}\times\mathbb{R}_{t}\right),\;\hat{S}_{\frac{1}{r},{\rm cl\,}}^{m,k}\left(\mathbb{C}^{2}\times\mathbb{R}^{3}\times\mathbb{R}_{t}\right)$
respectively. And similarly for the classes \prettyref{eq:classes with rho =000026 rho'}.

We now have the following
\begin{lem}
\label{l-gue200523yyd} Let $H=h^{L}\in\hat{L}_{\frac{1}{r},{\rm cl\,}}^{m,k}$
be an operator in the class \ref{exotic pseudos def.} with distribution
kernel 
\[
H(x,y)=h^{L}\left(x,y\right)=\int_{0}^{\infty}e^{i(x_{3}-y_{3})t}h\left(x,y,t\right)dt.
\]

Then there exists $\alpha(z;y,t)\in\hat{S}_{\frac{1}{r},{\rm cl\,}}^{m,k}\left(\mathbb{C}^{2}\times\mathbb{R}^{3}\times\mathbb{R}_{t}\right)$,
with $\alpha\left(x,0;y,t\right)=h\left(x,y,t\right)$, such that
\begin{align*}
\Lambda(z,y) & =\int_{0}^{\infty}e^{i\phi(z,y)t}\alpha\left(z,y,t\right)dt\quad\textrm{ with}\\
\left(PH-\Lambda\right)((0,\rho),y) & \in C^{\infty}\left(\mathbb{R}_{\rho}\times U\right).
\end{align*}
\end{lem}
\begin{proof}
Denote the Riemannian metric on $T\mathbb{C}^{2}$, induced from the
Hermitian metric $\langle\,\cdot\,|\,\cdot\,\rangle$, by 
\[
g=\sum_{j,k=1}^{4}g_{j,k}(z)dx_{j}\otimes dx_{k},\ \ dx_{4}=d\rho
\]
and let $\left(g_{j,k}(z)\right)_{1\leq j,k\leq4}^{-1}=\left(g^{j,k}(z)\right)_{1\leq j,k\leq4}$
be the inverse metric on $T^{*}\mathbb{C}^{2}$. In the local coordinates
$z=(x,\rho)$ chosen one has
\begin{align}
\Box_{f} & =-\frac{1}{2}\Bigl(a^{4,4}(z)\frac{\partial^{2}}{\partial\rho^{2}}+2\sum_{j=1}^{3}a^{4,j}(z)\frac{\partial^{2}}{\partial\rho\partial x_{j}}+T(\rho)\Bigl)+\text{\,first order },\quad\textrm{where }\label{eq: Laplacian locally}\\
T(\rho) & =\sum_{j,k=1}^{3}a^{j,k}(z)\frac{\partial^{2}}{\partial x_{j}\partial x_{k}}.\label{eq: term in Laplacian}
\end{align}
Furthermore, $T(0)-\triangle_{X}$ is a first order operator on the
boundary with
\begin{align}
a^{4,4}(x) & =1,\ \ \nonumber \\
a^{4,j}(x) & =0,\ \ j=1,\ldots,3.\label{eq:coefficients in Laplacian}
\end{align}

From the above \prettyref{eq: def of phase}, \prettyref{eq: Laplacian locally},
\prettyref{eq: term in Laplacian} and \prettyref{eq:coefficients in Laplacian}
we now compute 
\begin{equation}
\begin{split}\Box_{f}\left[\int_{0}^{\infty}e^{i\phi(z,y)t}\beta(x,y,t)dt\right] & \equiv\int_{0}^{\infty}e^{i\phi(z,y)t}\left[\frac{t}{2}\sqrt{-\sigma_{\triangle_{X}}(x,(0,0,1))}\partial_{\rho}\beta+L\beta\right]dt\\
 & \qquad\,\mod\,C^{\infty}((\tilde{U}\times\tilde{U})\cap(\overline{D}\times X)),\quad\textrm{where }\\
L & =\left(t\rho b_{1}\left(z,y\right)+b_{2}\left(z,y\right)\right)\partial_{\rho}+L_{2,x}+tL_{1,x}
\end{split}
\label{eq:computation Box f}
\end{equation}
for some smooth $t-$independent functions $b_{1},b_{2}$ and second/first
order differential operators $L_{2,x}$, $L_{1,x}$ respectively in
the $\left(x_{1},x_{2},x_{3}\right)$ variables. It is easy to check
that the above maps 
\begin{align*}
\frac{\rho}{t}L:\hat{S}_{\frac{1}{r},{\rm cl\,}}^{m,k} & \rightarrow\hat{S}_{\frac{1}{r},{\rm cl\,}}^{m,k}\\
\left(\frac{\rho}{t}L\right)^{N}:\hat{S}_{\frac{1}{r},{\rm cl\,}}^{m,k} & \rightarrow\rho^{N-k}\hat{S}_{\frac{1}{r},{\rm cl\,}}^{m,k}+\hat{S}_{\frac{1}{r},{\rm cl\,}}^{m-\frac{N}{r},k+N},\quad N\geq k.
\end{align*}

Next setting 
\begin{align*}
h_{1}(z,y,t) & \coloneqq h(x,y,t)-\frac{2\rho}{t\sqrt{-\sigma_{\triangle_{X}}(x,(0,0,1))}}Lh\\
 & =h(x,y,t)-\frac{2\rho}{t\sqrt{-\sigma_{\triangle_{X}}(x,(0,0,1))}}\left(L_{2,x}+tL_{1,x}\right)h\in\hat{S}_{\frac{1}{r},{\rm cl\,}}^{m,k}
\end{align*}
and following \prettyref{eq:computation Box f} one computes
\begin{align*}
\Box_{f}\left[\int_{0}^{\infty}e^{i\phi(z,y)t}h_{1}(z,y,t)dt\right] & \equiv\int_{0}^{\infty}e^{i\phi(z,y)t}r_{1}(z,y,t)dt\,\mod\,C^{\infty}((\tilde{U}\times\tilde{U})\cap(\overline{D}\times X)),\\
r_{1} & \in\rho^{1-k}\hat{S}_{\frac{1}{r},{\rm cl\,}}^{m+2,k}+\hat{S}_{\frac{1}{r},{\rm cl\,}}^{m+2-\frac{1}{r},k+1},\quad1\geq k.
\end{align*}
Continuing in this way, we can find $h_{N}(z,y,t)\in\hat{S}_{\frac{1}{r},{\rm cl\,}}^{m,k}$
such that 
\begin{align*}
h_{N+1}-h_{N} & \in\rho^{N-k}\hat{S}_{\frac{1}{r},{\rm cl\,}}^{m,k}+\hat{S}_{\frac{1}{r},{\rm cl\,}}^{m-\frac{N}{r},k+N}\quad\textrm{ and }\\
\Box_{f}\left[\int_{0}^{\infty}e^{i\phi(z,y)t}h_{N}(z,y,t)dt\right] & \equiv\int_{0}^{\infty}e^{i\phi(z,y)t}r_{N}(z,y,t)dt\,\mod\,C^{\infty}((\tilde{U}\times\tilde{U})\cap(\overline{D}\times X))\\
r_{N} & \in\rho^{N-k}\hat{S}_{\frac{1}{r},{\rm cl\,}}^{m+2,k}+\hat{S}_{\frac{1}{r},{\rm cl\,}}^{m+2-\frac{N}{r},k+N},\quad N\geq k.
\end{align*}
By asymptotic summation we can find $\alpha\coloneqq h_{1}+\sum_{N=1}^{\infty}\left(h_{N+1}-h_{N}\right)\in\hat{S}_{\frac{1}{r},{\rm cl\,}}^{m,k}$
satisfying 
\begin{align*}
\Box_{f}\left[\underbrace{\int_{0}^{\infty}e^{i\phi(z,y)t}\alpha(z,y,t)dt}_{=\Lambda}\right] & \equiv\int_{0}^{\infty}e^{i\phi(z,y)t}r_{\infty}(z,y,t)dt\,\mod\,C^{\infty}((\tilde{U}\times\tilde{U})\cap(\overline{D}\times X))\\
r_{\infty} & \in\rho^{\infty}\hat{S}_{\frac{1}{r},{\rm cl\,}}^{m+2,k}+\hat{S}_{\frac{1}{r},{\rm cl\,}}^{m+2,k,-\infty}\\
\gamma\Lambda & =H.
\end{align*}

Finally we apply the Green's operator $G$ \prettyref{eq:Green operator for Poisson}
to both sides of the above equation to get 
\[
G\Box_{f}\left(\left(0,\rho\right),y\right)=\int dw\int_{0}^{\infty}dtG\left(\left(0,\rho\right),w\right)e^{i\phi(w,y)t}r_{\infty}(w,y,t)dt\,.
\]
Writing $w=\left(u_{1},u_{2},u_{3},\rho'\right)$ and repeated integration
by parts using $e^{iu_{3}t}=\frac{1}{it}\partial_{u_{3}}e^{iu_{3}t}$,
\prettyref{eq:Green function Schwartz estimates} and \prettyref{eq: def of phase}
gives the above 
\begin{align*}
G\Box_{f}\left(\left(0,\rho\right),y\right) & \in C^{\infty}\left(\mathbb{R}_{\rho}\times U\right)\quad\textrm{ and hence }\\
\left[PH-\Lambda\right]\left(\left(0,\rho\right),y\right) & \in C^{\infty}\left(\mathbb{R}_{\rho}\times U\right)
\end{align*}
from \prettyref{eq: defining property green operator} as required.
\end{proof}
Similarly, as Lemma~\ref{l-gue200522ycdh}, we can find $\Phi(z,w)\in C^{\infty}((\tilde{U}\times\tilde{U})\cap(\overline{D}\times\overline{D}))$
such that $\Phi(z,y)=\phi(z,y)$ and 
\begin{equation}
\begin{split} & \Phi(z,w)=\phi(z,y)-i\rho'\sqrt{-\sigma_{\triangle_{X}}(y,(0,0,1))}+O(\left|\rho'\right|^{2}),\\
 & \mbox{\ensuremath{q_{0}(w,-\overline{d}_{w}\Phi)} vanishes to infinite order on \ensuremath{\rho'=0}}.
\end{split}
\label{two variable phase}
\end{equation}
A similar argument to Lemma~\ref{l-gue200523yyd} then gives the
following.
\begin{lem}
\label{l-gue200523yydI} Let $H=h^{L}\in\hat{L}_{\frac{1}{r},{\rm cl\,}}^{m,k}$
be an operator in the class \ref{exotic pseudos def.} with distribution
kernel 
\[
H(x,y)=h^{L}\left(x,y\right)=\int_{0}^{\infty}e^{i(x_{3}-y_{3})t}h\left(x,y,t\right)dt.
\]

Then there exists $\alpha(z;w,t)\in\hat{S}_{\frac{1}{r},{\rm cl\,}}^{m,k}\left(\mathbb{C}^{2}\times\mathbb{R}^{3}\times\mathbb{R}_{t}\right)$,
with $\alpha\left(x,0;y,0,t\right)=h\left(x,y,t\right)$, such that
\begin{align*}
\Lambda(z,w) & =\int_{0}^{\infty}e^{i\Phi(z,w)t}\alpha\left(z,w,t\right)dt\quad\textrm{ with}\\
\left(PHP^{*}-\Lambda\right)((0,\rho),y) & \in C^{\infty}\left(\mathbb{R}_{\rho}\times\mathbb{R}_{\rho'}\right).
\end{align*}
\end{lem}
Finally from the above we can prove one of our the main theorems \prettyref{thm:Fefferman thm.}.
Setting $z=\left(0,\rho\right)$, $w=\left(0,\rho\right)$ in \ref{Bergman kernel in terms of Poisson}
and the above \ref{l-gue200523yydI} gives 
\[
\Pi_{D}((0,\rho),(0,\rho))=\int_{0}^{\infty}e^{i\Phi((0,\rho),(0,\rho))t}\alpha((0,0,\rho)),(0,0,\rho);t)dt+C^{\infty}\left(\mathbb{R}_{\rho}\right).
\]
for $\alpha\in\hat{S}_{\frac{1}{r},{\rm cl\,}}^{1+\frac{2}{r}}(\tilde{U}\times\tilde{U}\times\mathbb{R}_{+})$.
Plugging the classical symbolic expansion for $\alpha$ into the above
and using 
\[
\Phi((0,\rho),(0,\rho))=-2i\rho\sqrt{-\sigma_{\triangle_{X}}(0,(0,0,1))}+O(\left|\rho\right|^{2})
\]
 gives
\[
\Pi_{D}((0,\rho),(0,\rho))=\sum_{j=0}^{N}\frac{1}{\left(-\rho\right){}^{2+\frac{2}{r}-\frac{1}{r}j}}a_{j}+\sum_{j=0}^{N}b_{j}\left(-\rho\right){}^{j}\log\left(-\rho\right)+O\left(\left(-\rho\right)^{\frac{N-2-2r}{r}}\right),
\]
$\forall N\in\mathbb{N},$ proving \prettyref{thm:Fefferman thm.}.

\section{\label{sec:S1 invariant CR geometry}$S^{1}$ invariant case}

In this section we investigate the Szeg\H{o} kernel parametrix in
the circle invariant case obtaining a more concrete version of our
main theorem \prettyref{thm:main thm parametrix}.

Thus we now assume that $X$ is equipped with a CR $S^{1}$-action
which is transversal; that is the generator $T$ of the $S^{1}$action
satisfies
\begin{align}
\left[T,C^{\infty}\left(T^{1,0}X\right)\right] & \subset C^{\infty}\left(T^{1,0}X\right)\label{eq: action is CR}\\
\mathbb{C}\left[T\right]\oplus T^{1,0}X\oplus T^{0,1}X & =TX\otimes\mathbb{C}.\label{eq:transversal decomposition}
\end{align}
Denote by $s_{x}\coloneqq\left|S_{x}\right|$ cardinality of the stabilizer
$S_{x}\coloneqq\left\{ e^{i\theta}\in S^{1}|e^{i\theta}x=x\right\} $
of the point $x\in X$ with respect to the circle action. For a locally
free circle action as above the function $x\mapsto s_{x}$ is also
upper semi-continuous with the compactness of $X$ implying $s\coloneqq\max_{x\in X}s_{x}<\infty$.
We further set $s_{0}\coloneqq\min_{x\in X}s_{x}$, $X_{i,j}\coloneqq\left\{ x\in X|s_{x}=j,\,r_{x}=i\right\} $
to obtain a decomposition of the manifold $X=\bigcup_{i=1,j=2}^{s,r}Y_{i,j}$
where each $X_{\leq i,\leq j}\coloneqq\bigcup_{i'=1,j'=2}^{i,j}Y_{i,j}$
is open and $X_{s_{0},\leq r}\subset X$ is dense. The $m$-th Fourier
mode of the Szeg\H{o} kernel $\Pi_{m}\left(x,x'\right)$, $m\in\mathbb{Z}$,
taken with respect to an $S^{1}$-invariant volume form $\mu$, is
now a smooth function on the product. What corresponds to the singularity
in \prettyref{eq:Szego parametrix} is its asymptotic behavior as
$m\rightarrow\infty$. This is described below and is the main theorem
of this section. 
\begin{thm}
\label{thm:Szego kernel expansion theorem}Let $X$ be a compact pseudoconvex
three dimensional CR manifold of finite type admitting a transversal,
CR circle action. The $m$-th Fourier mode of the Szeg\H{o} kernel
has the pointwise expansion on diagonal 
\begin{equation}
\Pi_{m}\left(x,x\right)=\phi_{m}\left(s_{x}\right)\left[m^{2/r_{x}}\sum_{j=0}^{N}c_{j}\left(x\right)m^{-2j/r_{x}}+O\left(m^{-2N/r_{x}}\right)\right],\quad\forall N\in\mathbb{N},\label{eq:Szego expansion}
\end{equation}
as $m\rightarrow\infty$. Here each $c_{j}$ is a smooth function
on $X$, with the leading term $c_{0}=\Pi^{g_{x}^{HX},j^{r_{x}-2}\mathscr{L},J^{HX}}\left(0,0\right)>0$
given in terms of certain model Bergman kernels on the Levi-distribution
$HX$ at $x$ and the phase factor $\phi_{m}\left(s_{x}\right)\coloneqq\sum_{l=0}^{s_{x}-1}e^{i\frac{2\pi lm}{s_{x}}}=\begin{cases}
s_{x}; & s_{x}\mid m,\\
0; & s_{x}\nmid m.
\end{cases}$
\end{thm}
We first begin with some requisite CR geometry in the circle invariant
case.

\subsection{$S^{1}$ invariant CR geometry}

Let $HX\coloneqq\textrm{Re}\left(T^{1,0}X\oplus T^{0,1}X\right)\subset TX$
be the real Levi distribution. The volume form $\mu$ is further assumed
to be $S^{1}$ invariant. We let $h^{T^{1,0}X}$ be an $S^{1}$ invariant
Hermitian metric on $T^{1,0}X$ and denote by $h^{T^{0,1}X}$ the
invariant Hermitian metric on $T^{0,1}X$. This gives one $h^{TX}$
on $TX\otimes\mathbb{C}$ for which \prettyref{eq: action is CR}
is an orthogonal decomposition with $\left|T\right|=1$. We also denote
by $g^{TX}$ the induced Riemannian metric on the real tangent space
and by $\left\langle ,\right\rangle $ the corresponding $\mathbb{C}$-bilinear
form on $TX\otimes\mathbb{C}$. It is easy to see that $h^{T^{0,1}X}$
may be chosen so that the volume form $\mu$ arises as the Riemannian
volume of such an invariant metric.

Such a $S^{1}$-invariant CR manifold is locally the unit circle bundle
of a Hermitian, holomorphic line bundle $\left(L,h^{L}\right)$
\begin{align}
X=S^{1}L\xrightarrow{\pi}Y;\quad\label{eq:locally circle bundle}
\end{align}
 on a complex manifold $Y$, $\textrm{dim}_{\mathbb{C}}Y=1$. To describe
the equivalence, choose a local hypersurface $Y\subset X$ through
a given point $x\in X$ transversal to generator $T\pitchfork TY$
satisfying $T_{x}Y=\left(HX\right)_{x}$. The map $TY\hookrightarrow TX\rightarrow TX/\mathbb{R}\left[T\right]\cong HX$
is then an isomorphism inducing a integrable complex structure on
$TY$ and a corresponding Hermitian metric $h^{T^{1,0}Y}$ on its
complex tangent space from $h^{T^{1,0}X}$. We choose $S^{1}\times Y\eqqcolon X_{0}\rightarrow Y$
to be the trivial circle bundle with the map $\iota:X_{0}\rightarrow X$;
$\iota\left(y,e^{i\theta}\right)=ye^{i\theta}$ being a local diffeomorphism
between collar neighborhoods
\begin{align}
\iota: & \left(-\frac{\pi}{s_{x}},\frac{\pi}{s_{x}}\right)\times Y\xrightarrow{\sim}NY\subset X\nonumber \\
s_{x} & \coloneqq\left|S_{x}\right|\label{eq:collar neighbourhood}
\end{align}
of the zero section of $X_{0}$ and $Y\subset X$. The Levi-distribution
then defines a horizontal distribution on $X_{0}$ giving corresponding
connections $\nabla^{L}$ on $X_{0}$ and the associated Hermitian
line bundle $L\coloneqq\mathbb{C}\times Y\rightarrow Y$ corresponding
to the trivial representation of $S^{1}$. By the integrability condition
the curvature of the corresponding connection is a $\left(1,1\right)$
form on $Y$ and hence the $\left(0,1\right)$ part of the connection
prescribes a holomorphic structure on $L$. It is now clear that $X_{0}$
it the unit circle bundle of $L$ with $\iota$ being the required
CR isomorphism by definition. We also note that pseudoconvexity of
the CR structure corresponds to semi-positivity of the curvature $R^{L}$
of $\nabla^{L}$.

We may also obtain a local coordinate expression for the CR structure.
To this end, start with a local orthonormal basis $\left\{ e_{1},e_{2}=Je_{1}\right\} $
of $T_{x}Y=\left(HX\right)_{x}$. Using the exponential map obtain
a geodesic coordinate system on a geodesic ball $B_{2\varrho}\left(x\right)$
centered at $x\in Y$ . The point $x$ corresponding to point $\left(y,\mathtt{l}_{y}\right)\in\left(S^{1}L\right)_{y}$
in the fiber above $y$, one may parallel transport $\mathtt{l}_{y}$
along geodesic rays in $Y$ to obtain a local orthonormal frame $\mathtt{l}$
for $L$. In such a parallel frame the connection on the tensor product
is of the following form 
\begin{align}
\nabla^{\Lambda^{0,*}\otimes L^{m}} & =d+a^{\Lambda^{0,*}}+ma^{L}\nonumber \\
a_{j}^{\Lambda^{0,*}} & =\int_{0}^{1}d\rho\left(\rho y^{k}R_{jk}^{\Lambda^{0,*}}\left(\rho x\right)\right)\nonumber \\
a_{j}^{L} & =\int_{0}^{1}d\rho\left(\rho y^{k}R_{jk}^{L}\left(\rho x\right)\right)\label{eq:connection formulas}
\end{align}
in terms of the respective curvatures of $\nabla^{T^{1,0}X}$, $\nabla^{L}$
see \cite[Sec. 4]{Marinescu-Savale18}. The connection one form may
further be written 
\begin{align}
a_{1}^{L}+ia_{2}^{L} & =\partial_{z}\varphi\nonumber \\
 & =\frac{1}{r_{x}}\bar{z}\left(\sum_{\left|\alpha\right|=r_{x}-2}R_{\alpha}^{L}y^{\alpha}\right)+O\left(y^{r_{x}}\right)\label{eq:connection expansion}
\end{align}
in terms of a potential function and a Taylor expansion with the tensor
$R_{\alpha}^{L}$ denoting the first non-vanishing jet of the curvature
at $x$. In some local coordinate system $\left(\theta,y\right)\in\left(-\frac{\pi}{s_{x}},\frac{\pi}{s_{x}}\right)\times B_{2\varrho}\left(x\right),$
$s_{x}\coloneqq\left|S_{x}\right|,$ on an open set $U\subset X$
the CR structure on $X$ is then locally given by 
\begin{align}
T^{1,0}X & =\mathbb{C}\left[\partial_{z}+i\left(\partial_{z}\varphi\right)\partial_{\theta}\right],\nonumber \\
T & =\partial_{\theta},\label{eq: BRT condition}
\end{align}
$z=y_{1}+iy_{2}$, by construction. A trivialization/coordinate system
in which the CR looks as above \prettyref{eq: BRT condition} is referred
to as a \textit{BRT trivialization} \cite[Thm II.1]{Baouendi-Rothschild-Treves-85}. 

Following its local description as a unit circle bundle, several notions/formulas
from complex geometry carry over to the $S^{1}$-invariant CR geometry
of $X$. Firstly, the tangential CR operator $\bar{\partial}_{b}$
locally corresponds to Dolbeault differential $\bar{\partial}$ on
$Y$ \prettyref{eq:locally circle bundle} under pullback $\bar{\partial}_{b}\left(\pi^{*}\omega\right)=\pi^{*}\bar{\partial}\omega$,
$\omega\in\Omega^{0,*}\left(Y\right)$ and similarly for their adjoints.
An analog of the Chern connection then exists on $T^{0,1}X$. This
is the unique ($S^{1}$-invariant) connection $\nabla^{T^{0,1}X}$
compatible with $\left\langle ,\right\rangle $ satisfying $\nabla_{T}^{T^{0,1}X}=\mathcal{L}_{T}$
and whose $\left(0,1\right)$-component agrees with the tangential
CR operator 
\begin{align*}
\nabla_{U}^{T^{0,1}X} & =i_{U}\bar{\partial}_{b},\quad U\in T^{0,1}X.
\end{align*}
Locally, $\nabla_{T}^{T^{0,1}X}$ is just the pullback of the Chern
connection $\nabla^{T^{0,1}Y}$ from $Y$. Complex conjugation then
defines a connection $\nabla^{T^{1,0}X}$ on $\left(T^{1,0}X\right)^{*}$.
We denote by the same notation dual connections on $T^{1,0}X/T^{0,1}X$
and set $\tilde{\nabla}^{TX}\coloneqq d\oplus\nabla^{T^{1,0}X}\oplus\nabla^{T^{0,1}X}$
to be a connection on $TX\otimes\mathbb{C}$ with $d$ denoting the
trivial connection on $\mathbb{C}\left[T\right]$ under the decomposition
\prettyref{eq: action is CR}. The connection $\tilde{\nabla}^{TX}$
preserves the real tangent space $TX\subset TX\otimes\mathbb{C}$
and we denote by $\mathcal{T}$ its torsion. The torsion $\mathcal{T}$
maps $T^{1,0}X\otimes T^{1,0}X$ into $T^{1,0}X$ (respectively for
$T^{0,1}X$), $T^{1,0}X\otimes T^{0,1}X$ into $\mathbb{C}\left[T\right]$
and vanishes on $HX\otimes\mathbb{C}\left[T\right]$. Indeed its components
involving the generator $T$ are 
\begin{align*}
\mathcal{T}\left(.,T\right) & =\mathcal{T}\left(T,.\right)=0\\
\left\langle \mathcal{T}\left(U,\bar{V}\right),T\right\rangle  & =i\mathscr{L}\left(U,\bar{V}\right),\quad U,V\in T^{1,0}X.
\end{align*}
Next with $\nabla^{TX}$ being the Levi-Civita connection and $P^{HX}$
being the horizontal projection onto $HX$, define a new connection
on $TX$ via $\nabla^{HX}\coloneqq d\oplus P^{HX}\nabla^{TX}$ with
respect to the decomposition $TX=\mathbb{R}\oplus HX$. Locally; $\nabla^{HX}$
is the pullback of the Levi-Civita connection $\nabla^{TY}$ from
$Y$ \prettyref{eq:locally circle bundle}. The torsion $\mathcal{T}^{HX}$
of $\nabla^{HX}$ thus has 
\[
\mathcal{T}^{HX}\left(U,V\right)=i\mathscr{L}\left(U,\bar{V}\right)T,\quad U,V\in HX,
\]
as its only non-vanishing component. The difference $S\coloneqq\tilde{\nabla}^{TX}-\nabla^{HX}$
maybe computed 
\begin{align}
\left\langle S\left(U\right)V,W\right\rangle =0, & \left\langle S\left(U\right)\bar{V},W\right\rangle =-\frac{1}{2}\left\langle \mathcal{T}\left(U,\bar{V}\right),W\right\rangle ;\label{eq:diff tensor 1}\\
\left\langle S\left(T\right)U,\bar{V}\right\rangle =\frac{i}{2}\mathscr{L}\left(U,\bar{V}\right), & \left\langle S\left(U\right)\bar{V},T\right\rangle =-\left\langle S\left(U\right)T,\bar{V}\right\rangle =0,\label{eq: diff tensor 2}
\end{align}
$U,V,W\in T^{1,0}X,$ in terms of the torsion and Levi forms. 

One next defines the Bismut connection $\nabla^{B}$ on $TX$ via
$\nabla^{B}\coloneqq\nabla^{HX}+S^{B}$; 
\begin{align*}
\left\langle S^{B}\left(T\right)U,\bar{V}\right\rangle = & \frac{i}{2}\mathscr{L}\left(U,\bar{V}\right)\\
\left\langle S^{B}\left(U\right)V,W\right\rangle \coloneqq & 0,
\end{align*}
$U,V,W\in T^{1,0}X.$ Its horizontal projection $P^{HX}\nabla^{B}$
locally agrees with the pullback of the Bismut connection of $Y$
\cite[Def. 1.2.9]{Ma-Marinescu}. The connection $\nabla^{B}$ preserves
the decomposition \prettyref{eq: action is CR} and hence induces
a connection on $T^{1,0}X$ , $\left(T^{0,1}X\right)^{*}$ and their
exterior powers which we again denote by $\nabla^{B}$. Finally set
\begin{equation}
\nabla^{B,\Lambda^{0,*}}\coloneqq\nabla^{B}+\left\langle S\left(.\right)w,\bar{w}\right\rangle ,\label{eq:Bismut connection}
\end{equation}
$w\in T^{1,0}X$, $\left|w\right|=1$. We now define the Clifford
multiplication endomorphism 
\begin{align*}
c:\left(TX\right)^{*} & \rightarrow\textrm{End}\left(\Lambda^{*}T^{0,1}X\right)\\
c\left(v\right) & \coloneqq\sqrt{2}\left(v^{1,0}\wedge-i_{v^{0,1}}\right),\quad\forall v\in\left(HX\right)^{*},\\
c\left(\theta\right)\text{\ensuremath{\omega}} & \coloneqq\pm\text{\ensuremath{\omega}},\quad\forall\omega\in\Lambda^{\textrm{even/odd}}T^{0,1}X.
\end{align*}
Next, the Kohn-Dirac operator 
\begin{align}
D_{b} & \coloneqq\sqrt{2}\left(\bar{\partial}_{b}+\bar{\partial}_{b}^{*}\right)\label{eq:Kohn Dirac}\\
 & =c\circ\nabla^{B,\Lambda^{0,*},H}\nonumber 
\end{align}
maybe written as the composition of Clifford multiplication with the
horizontal component of the Bismut connection
\[
\nabla^{B,\Lambda^{0,*},H}\coloneqq\pi^{H}\circ\nabla^{B,\Lambda^{0,*}}:C^{\infty}\left(X;\Lambda^{0,*}\right)\rightarrow C^{\infty}\left(X;H^{*}X\otimes\Lambda^{0,*}\right)
\]
(cf. \cite[Thm 1.4.5]{Ma-Marinescu}). We then have the following
Lichnerowicz formula.
\begin{thm}
\label{thm:Lichnerowicz formula} The Kohn Laplacian \prettyref{eq:Kohn Dirac}
satisfies 
\begin{equation}
2\Box_{b}=D_{b}^{2}=\underbrace{\left(\nabla^{B,\Lambda^{0,*},H}\right)^{*}\nabla^{B,\Lambda^{0,*},H}}_{\Delta^{B,\Lambda^{0,*}}\coloneqq}+\frac{1}{2}r^{X}\bar{w}i_{\bar{w}}+\mathscr{L}\left(w,\bar{w}\right)\left[2\bar{w}i_{\bar{w}}-1\right]i\mathcal{L}_{T}\label{eq:Lichnerowicz formula 3D}
\end{equation}
where $\frac{1}{2}r^{X}\coloneqq R^{T^{1,0}X}\left(w,\bar{w}\right)$
for $w\in T^{1,0}X$, $\left|w\right|=1$.
\end{thm}
\begin{proof}
On account of $S^{1}$ invariance, both sides of the formula commute
with the generator $\mathcal{L}_{T}$ . It then suffices to check
their equality on sections that are locally of the form $s\left(y,e^{i\theta}\right)=s_{0}\left(y\right)e^{im\theta}$,
$s_{0}\in\Omega^{0,*}\left(Y\right)$, eigenspaces of $\mathcal{L}_{T}$,
on the unit circle bundle \prettyref{eq:locally circle bundle}. Further
one locally has the correspondence
\begin{equation}
C_{m}^{\infty}\left(X\right)\cong C^{\infty}\left(Y;L^{m}\right)\label{eq:local correspondence with sections}
\end{equation}
between sections on $X$ that are $m$-eigenspaces of $\mathcal{L}_{T}$
and sections of $L^{m}$ on $Y$ for example. Under this correspondence
$D_{b}$, $\nabla^{B,\Lambda^{0,*},H}$ act by the Dolbeault-Dirac
operators and Bismut connection on tensor powers $L^{m}$. The Lie
derivative $\mathcal{L}_{T}$ acts by multiplication by $m$ while
the curvature of $L$ is identified with the Levi form by definition.
The curvatures $r^{X},\,R^{T^{1,0}X}$ are pulled back from the scalar
and Chern curvatures on $Y$ while the horizontal components of $\Theta$
agree with the components of the corresponding tensor on $Y$. With
these identifications, the Lichnerowicz formula \prettyref{eq:Lichnerowicz formula 3D}
is locally the same as \cite[Thm 1.4.7]{Ma-Marinescu}.
\end{proof}
The (horizontal) Bochner Laplacian appearing in \prettyref{eq:Lichnerowicz formula 3D}
can be written 
\[
\Delta^{B,\Lambda^{0,*}}\coloneqq\sum_{j=1}^{2m}\left[-\left(\nabla_{U_{j}}^{B,\Lambda^{0,*}}\right)^{2}s+\left(\textrm{div}U_{j}\right)\nabla_{U_{j}}^{B,\Lambda^{0,*}}s\right],
\]
in terms of a real orthonormal basis $\left\{ U_{j}\right\} _{j=1}^{2m}$
for $HX$. It is a sub-Riemannian Laplacian associated to the metric
(bracket generating) distribution $HX$ and the natural Riemannian
volume see \cite[Sec. 2]{Marinescu-Savale18}. From the above expression
it is clearly a hypoelliptic operator of Hörmander type \cite{Hormander1967}.
It satisfies a hypoelliptic estimate: $\exists C>0$ such that 
\begin{equation}
\left\langle \Delta^{B,\Lambda^{0,*}}u,u\right\rangle +\left\Vert u\right\Vert ^{2}\geq C\left\Vert u\right\Vert _{H^{1/r}}^{2}\label{eq:subelliptic estimate}
\end{equation}
$\forall u\in\Omega^{0,*}\left(X\right)$ (see \cite{Rothschild-Stein76}).
Here $r$ is the maximal type of a point on $X$. This is also referred
to as the step or degree of non-holonomy of the Levi distribution
$HX$ in sub-Riemannian geometry.

\subsection{\label{subsec:Spectral-gap-and}Spectral gap and closed range}

In this subsection we show that the spectral gap property for the
Kohn Laplacian as well as the closedness of the range for $\bar{\partial}_{b}$
in three dimensions automatically follow in the circle invariant case.

First, each $\Omega^{0,q}\left(X\right)$ has an orthogonal decomposition
(with respect to the chosen invariant metric) into the Fourier modes
for the $S^{1}$ action
\begin{align}
\Omega^{0,q}\left(X\right) & =\oplus_{m\in\mathbb{Z}}\Omega_{m}^{0,q}\left(X\right)\quad\textrm{ where }\label{eq: Fourier modes}\\
\Omega_{m}^{0,q}\left(X\right) & \coloneqq\left\{ \omega\in\Omega_{m}^{0,q}\left(X\right)|\mathcal{L}_{T}\omega=im\omega\right\} .\nonumber 
\end{align}
Indeed the orthogonal projection of any $\omega\in\Omega^{0,q}\left(X\right)$
onto its $m$th Fourier mode is given by 
\begin{align}
P_{m}: & \Omega^{0,q}\left(X\right)\rightarrow\Omega_{m}^{0,q}\left(X\right)\nonumber \\
\left(P_{m}\omega\right)\left(x\right) & \coloneqq\int_{S^{1}}d\theta\,\omega\left(x.e^{i\theta}\right)e^{-im\theta}\label{eq:projection onto m fourier comp.}
\end{align}
Since the $S^{1}$ action is assumed to be CR, we have $\left[\bar{\partial}_{b},T\right]=0$.
Hence the tangential CR operator preserves the Fourier modes \prettyref{eq: Fourier modes}
$\left[\bar{\partial}_{b},P_{m}\right]=0$, $\bar{\partial}_{b}:\Omega_{m}^{0,q}\left(X\right)\rightarrow\Omega_{m}^{0,q+1}\left(X\right)$
. One may then define the $m$-th equivariant Kohn-Rossi cohomology
\[
H_{b,m}^{*}\left(X\right)\coloneqq\frac{\textrm{ker}\left[\bar{\partial}_{b}:\Omega_{m}^{0,q}\left(X\right)\rightarrow\Omega_{m}^{0,q+1}\left(X\right)\right]}{\textrm{Im}\left[\bar{\partial}_{b}:\Omega_{m}^{0,q}\left(X\right)\rightarrow\Omega_{m}^{0,q+1}\left(X\right)\right]}.
\]
Its adjoint $\bar{\partial}_{b}^{*}$ with respect to the invariant
metric further commutes $\left[T,\bar{\partial}_{b}^{*}\right]=0$
with the generator. A similar commutation then applies to the Kohn
Laplacian $\left[T,\Box_{b}\right]=0$, thus giving a decomposition
\begin{align*}
\Box_{b} & =\oplus_{q=0,1}\oplus_{m\in\mathbb{Z}}\Box_{b,m}^{q}\\
\Box_{b,m}^{q} & \coloneqq\left.\Box_{b}\right|_{\Omega_{m}^{0,q}}:\Omega_{m}^{0,q}\left(X\right)\rightarrow\Omega_{m}^{0,q}\left(X\right).
\end{align*}
The equivariant version of the Hodge theorem holds
\begin{equation}
\textrm{ker}\left(\Box_{b,m}^{q}\right)=H_{b,m}^{*}\left(X\right)\label{eq:hodge theorem}
\end{equation}
\cite[Thm 3.7]{Cheng-Hsiao-Tsai-2015}.

The $S^{1}$-invariant operator $\Box^{X}\coloneqq-T^{2}+\Box_{b}$
is elliptic and self-adjoint with respect to $\left\langle ,\right\rangle $,
$dx$. There is thus a complete orthonormal basis $\varphi_{j,m}^{q}$,
$j=0,1,\ldots$, $m\in\mathbb{Z}$, of $L^{2}\left(X;\Lambda^{0,q}\right)$
consisting of joint eigenvectors $\Box^{X}\varphi_{j,m}^{q}=\lambda_{j,m}^{q}\varphi_{j,m}^{q};\,T\varphi_{j,m}^{q}=im\varphi_{j,m}^{q}$
with 
\begin{align*}
0 & \leq m^{2}\leq\lambda_{0,m}^{q}\leq\lambda_{1,m}^{q}\leq\ldots\nearrow\infty,
\end{align*}
$\forall m\in\mathbb{Z}$. Thus, for each fixed $m\in\mathbb{Z}$
the set $\left\{ \varphi_{j,m}^{q}\right\} _{j=0}^{\infty}$ is then
an orthonormal basis of $L_{m}^{2}\left(X;\Lambda^{0,q}\right)\coloneqq\left\{ \omega\in L_{m}^{2}\left(X;\Lambda^{0,q}\right)|T\omega=im\omega\right\} $
of eigenvectors for $\Box_{b,m}^{q}$. Note that $\Box_{b,m}^{q}$
is an unbounded operator on $L_{m}^{2}\left(X;\Lambda^{0,q}\right)$
with domain $\textrm{Dom}\left(\Box_{b,m}^{q}\right)=\left\{ \omega\in L_{m}^{2}\left(X;\Lambda^{0,q}\right)|\Box_{b,m}^{q}\omega\in L_{m}^{2}\left(X;\Lambda^{0,q}\right)\right\} $.
Similarly one also has 
\begin{align*}
\textrm{Dom}\left(\bar{\partial}_{b}\right) & =\left\{ \omega\in L^{2}\left(X;\Lambda^{0,*}\right)|\bar{\partial}_{b}\omega\in L^{2}\left(X;\Lambda^{0,*}\right)\right\} ,\\
\textrm{Dom}\left(\bar{\partial}_{b}^{*}\right) & =\left\{ \omega\in L^{2}\left(X;\Lambda^{0,*}\right)|\bar{\partial}_{b}^{*}\omega\in L^{2}\left(X;\Lambda^{0,*}\right)\right\} ,
\end{align*}
as unbounded operators on $L^{2}$.

We now have the following spectral gap property for $\Box_{b,m}^{q}$.
\begin{prop}
\label{prop:spectral gap}There exists a constants $c_{1},c_{2}>0$
such that 
\begin{align}
\textrm{Spec}\left(\Box_{b,m}^{0}\right) & \subset\left\{ 0\right\} \cup\left[c_{1}\left|m\right|^{2/r}-c_{2},\infty\right)\label{eq: spectral gap funs.}\\
\textrm{Spec}\left(\Box_{b,m}^{1}\right) & \subset\left[c_{1}\left|m\right|^{2/r}-c_{2},\infty\right)\label{eq:spectral gap one forms}
\end{align}
 for each $m\in\mathbb{Z}$.
\end{prop}
\begin{proof}
The Lichnerowicz formula \prettyref{eq:Lichnerowicz formula 3D} when
restricted for the restriction to the $m$-th Fourier mode $\Omega_{m}^{0,1}\left(X\right)$
gives
\[
2\Box_{b,m}^{1}=\underbrace{\Delta_{m}^{B,\Lambda^{0,1}}}_{\eqqcolon\left.\Delta^{B,\Lambda^{0,*}}\right|_{\Omega_{m}^{0,1}\left(X\right)}}+\frac{1}{2}r^{X}+\mathscr{L}\left(w,\bar{w}\right)m
\]
in terms of an invariant orthonormal frame $w\in T^{1,0}X$. Following
the subelliptic estimate \prettyref{eq:subelliptic estimate} and
pseudoconvexity gives 
\begin{align*}
2\lambda_{j.m}^{q}=\left\langle 2\Box_{b,m}^{1}\varphi_{j,m},\varphi_{j,m}\right\rangle  & =\left\langle \left[\Delta_{m}^{B,\Lambda^{0,1}}+\frac{1}{2}r^{X}+\mathscr{L}\left(w,\bar{w}\right)m\right]\varphi_{j,m},\varphi_{j,m}\right\rangle \\
 & \geq c_{1}\left\Vert \varphi_{j,m}\right\Vert _{H^{1/r}}^{2}-c_{2}\left\Vert \varphi_{j,m}\right\Vert ^{2}\\
 & =c_{1}\left\Vert T^{1/r}\varphi_{j,m}\right\Vert ^{2}-\left\Vert \varphi_{j,m}\right\Vert ^{2}=c_{1}m^{2/r}-c_{2}
\end{align*}
for $m\geq0$. Thus $\textrm{Spec}\left(\Box_{b,m}^{1}\right)\subset\left[c_{1}m^{2/r}-c_{2},\infty\right)$
for $m\geq0$ proving the second part \prettyref{eq:spectral gap one forms}.
Similarly one proves $\textrm{Spec}\left(\Box_{b,m}^{0}\right)\subset\left[c_{1}\left|m\right|^{2/r}-c_{2},\infty\right)$
for $m\leq0$. The first part \prettyref{eq: spectral gap funs.}
follows on noting that the Dirac operator $D_{b,m}\coloneqq\left.D_{b}\right|_{\Omega_{m}^{0,*}\left(X\right)}$
gives an isomorphism between the non-zero eigenspaces of $\Box_{b,m}^{0}$
and $\Box_{b,m}^{1}$ for each $m\in\mathbb{Z}$.
\end{proof}
It follows immediately from \prettyref{eq:hodge theorem}, \prettyref{eq:spectral gap one forms}
that $H_{b,m}^{1}\left(X\right)=0$ for $m\gg0$. This gives 
\begin{align*}
d_{m}\coloneqq\textrm{dim }H_{b,m}^{0}\left(X\right) & =\phi_{m}\left(s_{0}\right)\int_{X}\textrm{Td}_{b}\left(T^{1,0}X\right)e^{-md\theta}\wedge\theta\\
 & =\phi_{m}\left(s_{0}\right)m\left[\int_{X}d\theta\wedge\theta\right]+O\left(1\right)
\end{align*}
following the index theorem of \cite[Cor. 1.13]{Cheng-Hsiao-Tsai-2015},
where $\textrm{Td}_{b}\left(T^{1,0}X\right)$ denotes the tangential/invariant
Todd class of $T^{1,0}X$ \cite[Sec. 2.3]{Cheng-Hsiao-Tsai-2015}
and $\theta\left(T\right)=1$, $\theta\left(HX\right)=0$.

As another application of Proposition \prettyref{prop:spectral gap}
one has the estimate $\textrm{Spec}^{+}\left(D_{b,m}^{2}\right)\subset\left[c_{1}\left|m\right|^{2/r}-c_{2},\infty\right)$
for each $m\in\mathbb{Z}$ on the positive spectrum of the Dirac operator.
One then sees
\[
\left\Vert D_{b}^{2}\omega\right\Vert \geq c_{1}\left\Vert D_{b}\omega\right\Vert ,\quad\forall\omega\in\Omega^{0,1}\left(X\right),
\]
on its decomposition into Fourier modes. The last inequality is rewritten
$\left\Vert \bar{\partial}_{b}\bar{\partial}_{b}^{*}\omega\right\Vert \geq c_{1}\left\Vert \bar{\partial}_{b}^{*}\omega\right\Vert $,
$\forall\omega\in\Omega^{0,1}\left(X\right)$ and hence 
\begin{equation}
\left\Vert \bar{\partial}_{b}\omega\right\Vert \geq c_{1}\left\Vert \omega\right\Vert ,\,\forall\omega\in\textrm{Dom}\left(\bar{\partial}_{b}\right)\cap\overline{\textrm{Range}\left(\bar{\partial}_{b}^{*}\right)},\label{eq: closed range property}
\end{equation}
which is equivalent to the closed range property for $\bar{\partial}_{b}$
(see \cite{Hormander65-L2est} Sec. 1). 

\subsection{\label{sec:Szego-kernel-expansion}Szeg\H{o} kernel expansion}

We now investigate the asymptotic expansion of the Szeg\H{o} kernel
on diagonal. In the presence of a locally free circle action, its
$m$-th Fourier component of the Szeg\H{o} kernel $\Pi_{b,m}\left(x,x'\right)\coloneqq\frac{1}{2\pi}\int\Pi_{b}\left(x,x'e^{i\theta}\right)e^{im\theta}d\theta$
is given as the Schwartz kernel of the orthogonal projector
\begin{equation}
\Pi_{b,m}\coloneqq\Pi_{b}\circ P_{m}:L^{2}\left(X\right)\rightarrow\textrm{ker}\left(\Box_{b,m}^{0}\right)\subset L^{2}\left(X\right).\label{eq:Fourier component Szego projection}
\end{equation}
It is also written in terms of orthonormal zero eigenfunctions $\left\{ \psi_{1}^{m},\ldots,\psi_{d_{m}}^{m}\right\} =\left\{ \varphi_{j,m}^{0}\in L^{2}\left(X\right)|\lambda_{j,m}^{0}=0\right\} $
of $\Box_{b,m}^{0}$ via
\begin{equation}
\Pi_{b,m}\left(x,x'\right)\coloneqq\sum_{j=1}^{d_{m}}\psi_{j}^{m}\left(x\right)\overline{\psi_{j}^{m}\left(x'\right)}\label{eq:Bergman kernel in eigenfunctions}
\end{equation}
of $\textrm{ker}\left(\Box_{b,m}^{0}\right)$. The singularity of
the Szeg\H{o} kernel \prettyref{thm:main thm parametrix} corresponds
to the on-diagonal asymptotics of the Fourier component $\Pi_{b,m}\left(x,x\right)$
as $m\rightarrow\infty$ in the circle invariant case and we wish
to describe it in this subsection.

We shall first localize the problem. We use the local description
\prettyref{eq:locally circle bundle} of $X$ as the unit circle bundle
of $X=S^{1}L\xrightarrow{\pi}Y$ of a Hermitian holomorphic line bundle
$\left(L,\nabla^{L},h^{L}\right)$ over a complex Hermitian manifold
$Y$. Finally and as partly noted before under the identification
\prettyref{eq:local correspondence with sections} one has the $m$th
Fourier component of the Kohn Laplacian
\begin{equation}
\Box_{b,m}^{0}=\Box_{m}\coloneqq\frac{1}{2}D_{m}^{2}\label{eq:local Bochner Laplacian}
\end{equation}
is locally given in terms of the Kodaira Laplacian on tensor powers
$C^{\infty}\left(Y;L^{m}\right)$. We now define a modify the frame
$\left\{ e_{1},e_{2}\right\} $, used in the expressions \prettyref{eq:connection formulas},
and define the frame $\left\{ \tilde{e}_{1},\tilde{e}_{2}\right\} $
on $\mathbb{R}^{2}$ which agrees with $\left\{ e_{1},e_{2}\right\} $
on $B_{\varrho}\left(y\right)$ and with $\left\{ \partial_{x_{1}},\partial_{x_{2}}\right\} $
outside $B_{2\varrho}\left(x\right)$. Also define the modified metric
$\tilde{g}^{TY}$ and almost complex structure $\tilde{J}$ on $\mathbb{R}^{2}$
to be standard in this frame and hence agreeing with $g^{TY}$, $J$
on $B_{\varrho}\left(x\right)$. The Christoffel symbol of the corresponding
modified induced connection on $\Lambda^{0,*}$ now satisfies 
\[
\tilde{a}^{\Lambda^{0,*}}=0\quad\textrm{ outside }B_{2\varrho}\left(x\right).
\]
Being identified with the Levi form, the curvature $R^{L}$ is semi-positive
by assumption with its order of vanishing $\textrm{ord}_{x}\left(R^{L}\right)=r_{x}-2\in2\mathbb{N}_{0}$
being given in terms of the type of the point $x$. 

We may then Taylor expand the curvature 
\begin{align}
R^{L} & =\underbrace{\sum_{\left|\alpha\right|=r_{x}-2}R_{\alpha}^{L}y^{\alpha}dy_{1}dy_{2}}_{=R_{0}^{L}}+O\left(y^{r_{x}-1}\right)\quad\textrm{ with }\label{eq:curvature Taylor expansion}\\
iR_{0}^{L}\left(e_{1},e_{2}\right) & \geq0.\label{eq: first term semi-positive}
\end{align}
Now define the modified connection on $L$ via 
\begin{align}
\tilde{\nabla}^{L} & =d+\left[\underbrace{\int_{0}^{1}d\rho\,\rho y^{k}\left(\tilde{R}^{L}\right)_{jk}\left(\rho y\right)}_{=\tilde{a}_{j}^{L}}\right]dy_{j},\quad\textrm{ where}\nonumber \\
\tilde{R}^{L} & =\chi\left(\frac{\left|y\right|}{2\varrho}\right)R^{L}+\left[1-\chi\left(\frac{\left|y\right|}{2\varrho}\right)\right]R_{0}^{L}.\label{eq:modified connection-1}
\end{align}
which agrees with $\nabla^{L}$ on $B_{\varrho}\left(y\right)$. Note
that the curvature $\tilde{R}^{L}$ of $\tilde{\nabla}^{L}$ above
is also semi-positive by definition. Furthermore one also has $\tilde{R}^{L}=R_{0}^{L}+O\left(\varrho^{r_{y}-1}\right)$
and that the $\left(r_{x}-2\right)$-th derivative/jet of $\tilde{R}^{L}$
is non-vanishing at all points on $\mathbb{R}^{2}$ for 
\begin{equation}
0<\varrho<c\left|j^{r_{y}-2}R^{L}\left(y\right)\right|.\label{eq: curv =00003D000026 jet comp.}
\end{equation}
Here $c$ is a uniform constant on the manifold. We then define the
modified Kodaira Dirac operator on $\mathbb{R}^{2}$ by the similar
formula

\begin{equation}
\tilde{D}_{m}=c\circ\tilde{\nabla}^{\Lambda^{0,*}\otimes L^{m}}\label{eq:local Dirac}
\end{equation}
which agrees with $D_{m}$ on $B_{\varrho}\left(y\right)$. The above
satisfies a similar Lichnerowicz formula for the corresponding Kodaira
Laplacian
\begin{align}
2\tilde{\Box}_{m}\coloneqq\tilde{D}_{m}^{2} & =\left(\tilde{\nabla}^{\Lambda^{0,*}\otimes L^{m}}\right)^{*}\tilde{\nabla}^{\Lambda^{0,*}\otimes L^{m}}+m\tilde{R}^{L}\left(w,\bar{w}\right)\left[2\bar{w}i_{\bar{w}}-1\right]+\frac{1}{2}\tilde{r}^{X}\bar{w}i_{\bar{w}}\label{eq:modified local Bochner}
\end{align}
where $w=\frac{1}{\sqrt{2}}\left(\tilde{e}_{1}-i\tilde{e}_{2}\right)$,
$\tilde{r}^{X}\coloneqq\tilde{R}^{T^{1,0}X}\left(w,\bar{w}\right)$,
the adjoint being taken with respect to the metric $\tilde{g}^{TY}$
and corresponding volume form. The above \prettyref{eq:modified local Bochner}
again agrees with 
\begin{equation}
\tilde{\Box}_{m}=\Box_{m}\quad\textrm{on }B_{\varrho}\left(y\right)\label{eq:agreement with localization}
\end{equation}
while the endomorphism $\tilde{r}^{X}$ vanishes outside $B_{\varrho}\left(y\right)$.
Being semi-bounded below \prettyref{eq:modified local Bochner} is
essentially self-adjoint. A similar argument as Corollary \prettyref{prop:spectral gap}
gives a spectral gap 
\begin{equation}
\textrm{Spec}\left(\tilde{\Box}_{m}\right)\subset\left\{ 0\right\} \cup\left[c_{1}m^{2/r_{x}}-c_{2},\infty\right).\label{eq: spectral gap}
\end{equation}
Thus for $m\gg0$, the resolvent $\left(\tilde{\Box}_{m}-z\right)^{-1}$
is well-defined in a neighborhood of the origin in the complex plane.
From local elliptic regularity, the Bergman projector
\begin{equation}
\tilde{B}_{m}:L^{2}\left(\mathbb{R}^{2};L^{\otimes m}\right)\rightarrow\ker\left(\tilde{\Box}_{m}\right)\label{eq:local Bergman kernel}
\end{equation}
then has a smooth Schwartz kernel with respect to the Riemannian volume
of $\tilde{g}^{TY}$.

Now we choose a set of such BRT trivializations $\left\{ U_{j}=\left(-\frac{\pi}{s_{x_{j}}},\frac{\pi}{s_{x_{j}}}\right)\times B_{2\varrho}\left(x_{j}\right)\right\} _{j=1}^{N}$
\prettyref{eq: BRT condition} centered at $\left\{ x_{j}\in X\right\} _{j=1}^{N}$
with corresponding modified Laplacians $\left\{ \tilde{\Box}_{m,j}\right\} _{j=1}^{N}$
and Bergman projectors $\left\{ \tilde{B}_{m,j}\right\} _{j=1}^{N}$
such that $\left\{ U_{j}^{0}\coloneqq\left(-\frac{\pi}{2s_{x_{j}}},\frac{\pi}{2s_{x_{j}}}\right)\times B_{\varrho}\left(x_{j}\right)\right\} _{j=1}^{N}$cover
$X$. Choose a partition of unity $\left\{ \chi_{j}\in C_{c}^{\infty}\left(U_{j}^{0};\left[0,1\right]\right)\right\} _{j=1}^{N}$
subordinate to the the BRT cover, $j=1,\ldots,N$. Further choose
$\psi_{j}\in C_{c}^{\infty}\left(U_{j};\left[0,1\right]\right)$ such
that $\psi_{j}=1$ on $\textrm{spt}\left(\chi_{j}\right)$ and 
\begin{equation}
\sigma_{j}\in C_{c}^{\infty}\left(-\frac{\pi}{s_{x_{j}}},\frac{\pi}{s_{x_{j}}}\right)_{\theta}\textrm{ with }\int\sigma_{j}d\theta=1\label{eq:angular cutoff}
\end{equation}
 in each such trivialization. We note that a finite propagation argument
as in \cite[Sec. 1.6]{Ma-Marinescu} gives 
\begin{equation}
\chi_{j}\tilde{B}_{m,j}\psi_{j}=\chi_{j}\tilde{B}_{m,j}\quad\textrm{mod }O\left(m^{-\infty}\right)\label{eq:finite propag.}
\end{equation}
thus the right hand side above maybe assumed to be properly supported
mod $O\left(m^{-\infty}\right)$.

Now define the approximate Szeg\H{o} kernel via
\begin{align}
\tilde{\Pi}_{m} & \coloneqq\left(\sum_{j=1}^{\infty}\int dy'd\theta'\,\chi_{j}\left(y,\theta\right)e^{im\theta}\tilde{B}_{m,j}\left(y,y'\right)e^{-im\theta'}\psi_{j}\left(y',\theta'\right)\sigma_{j}\left(\theta'\right)\right)\label{eq:localized Szego kernel prelim}\\
\tilde{\Pi}_{b,m} & \coloneqq\tilde{\Pi}_{m}\circ P_{m}.\label{eq:localized Szego kernel}
\end{align}
We now have the following localization lemma.
\begin{lem}
\label{lem: localization Szego kernel} The approximate Szeg\H{o}
kernel \prettyref{eq:localized Szego kernel prelim} satisfies 
\begin{equation}
\tilde{\Pi}_{b,m}-\Pi_{b,m}=O\left(m^{-\infty}\right)\label{eq:localization Szego kernel}
\end{equation}
in the $C^{\infty}$ norm on the product $X\times X$.
\end{lem}
\begin{proof}
We first show that by direct computation that 
\begin{equation}
\tilde{\Pi}_{b,m}\Pi_{b,m}=\Pi_{b,m}\quad\textrm{ mod }O\left(m^{-\infty}\right).\label{eq:approx szego proj prop.}
\end{equation}
To this end, let $f\in C^{\infty}\left(X\right)$ and $g\coloneqq\Pi_{b,m}f$.
Then $g\in\textrm{ker }\left(\Box_{b,m}^{0}\right)$ and thus $g\left(ye^{i\theta}\right)=g_{0}\left(y\right)e^{im\theta}$
with $\tilde{\Box}_{m,j}g_{0}=0$ on each $B_{\varrho}\left(x_{j}\right)$
by \prettyref{eq:local Bochner Laplacian}, \prettyref{eq:agreement with localization}.
With 
\begin{equation}
\tilde{B}_{m,j}g_{0}=g_{0}\textrm{ mod }O\left(m^{-\infty}\right)\textrm{ on }B_{\varrho}\left(x_{j}\right)\label{eq:local bergman property}
\end{equation}
following from a finite propagation argument, we may calculate 
\begin{align*}
\tilde{\Pi}_{b,m}g & =\tilde{\Pi}_{m}g\\
 & =\sum_{j=1}^{\infty}\int dy'd\theta'\,\chi_{j}\left(x\right)e^{im\theta}\tilde{B}_{m,j}\left(y,y'\right)e^{-im\theta'}\sigma_{j}\left(\theta'\right)g_{0}\left(y'\right)e^{im\theta'}\\
 & =\sum_{j=1}^{\infty}\chi_{j}\left(x\right)g\quad\textrm{ mod }O\left(m^{-\infty}\right)\\
 & =g\quad\textrm{ mod }O\left(m^{-\infty}\right)
\end{align*}
using \prettyref{eq:angular cutoff}, \prettyref{eq:localized Szego kernel prelim}
and \prettyref{eq:local bergman property} showing \prettyref{eq:approx szego proj prop.}.

In similar vein, with $P_{m}f=g\in C_{m}^{\infty}\left(X\right)$
satisfying $g\left(ye^{i\theta}\right)=g_{0}\left(y\right)e^{im\theta}$
on each $B_{\varrho}\left(x_{j}\right)$ as before we calculate
\begin{align}
\tilde{\Pi}_{b,m}\Box_{b}f & =\tilde{\Pi}_{m}\circ P_{m}\circ\Box_{b}f\nonumber \\
 & =\tilde{\Pi}_{m}\circ\Box_{b,m}g\nonumber \\
 & =\sum_{j=1}^{\infty}\int dy'd\theta'\,\chi_{j}\left(x\right)e^{im\theta}\tilde{B}_{m,j}\left(y,y'\right)e^{-im\theta'}\sigma_{j}\left(\theta'\right)\tilde{\Box}_{m,j}g_{0}\left(y'\right)e^{im\theta'}\quad\textrm{ mod }O\left(m^{-\infty}\right)\nonumber \\
 & =0\quad\textrm{ mod }O\left(m^{-\infty}\right)\label{eq:projection annihilates}
\end{align}
using \prettyref{eq:local Bochner Laplacian}, \prettyref{eq:agreement with localization}
and another finite propagation argument. 

Finally letting 
\begin{align}
N_{m}: & L_{m}^{2}\left(X\right)\rightarrow\textrm{Dom}\left(\Box_{b,m}^{0}\right),\nonumber \\
N_{m}f & =\begin{cases}
0; & f\in\textrm{ker}\left(\Box_{b,m}^{0}\right),\\
\Box_{b,m}^{-1}f; & f\in\textrm{ker}\left(\Box_{b,m}^{0}\right)^{\perp}
\end{cases}\label{eq:partial inverse}
\end{align}
 denote the partial inverse of $\Box_{b,m}^{0}$ we calculate
\begin{align*}
\tilde{\Pi}_{b,m}^{*} & =P_{m}\tilde{\Pi}_{b,m}^{*}\\
 & =\left(N_{m}\Box_{b}P_{m}+\Pi_{b,m}\right)\tilde{\Pi}_{b,m}^{*}\\
 & =\Pi_{b,m}\tilde{\Pi}_{b,m}^{*}\quad\textrm{ mod }O\left(m^{-\infty}\right)\\
 & =\Pi_{b,m}\quad\textrm{ mod }O\left(m^{-\infty}\right)
\end{align*}
following \prettyref{eq:approx szego proj prop.}, \prettyref{eq:projection annihilates}
and proving the proposition on account of the self-adjointness of
$\Pi_{b,m}$.
\end{proof}
We note that the on-diagonal asymptotic expansion for the local Bergman
kernel $\tilde{B}_{m}\left(y,y\right)$ follows in a similar fashion
as \prettyref{thm:local Bergman is a symbol}. A slight difference
here is that $\tilde{B}_{m}\left(y,y\right)$ is defined with respect
to a more general metric while the metric in \prettyref{thm:local Bergman is a symbol}
is flat. This however makes little difference to the argument and
gives the following (cf. \cite[Thm. 3]{Marinescu-Savale18}).
\begin{thm}
\label{thm:local Bergman expansion} For any differential operator
$P$ of order $l$, the derivative of the local Bergman kernel has
the pointwise expansion on diagonal
\begin{equation}
P\Pi^{\tilde{\Box}_{m}}\left(y,y\right)=m^{\left(2+l\right)/r_{y}}\left[\sum_{j=0}^{N}c_{j}\left(P,y\right)m^{-2j/r_{y}}+O\left(m^{l-2N+1}\right)\right],\label{eq:local Bergman expansion}
\end{equation}
$\forall N\in\mathbb{N}_{0}$. The leading term, for $P=1$, is given
$c_{0,0}\left(1,y\right)=\Pi^{g_{x}^{HX},j^{r_{x}-2}\mathscr{L},J^{HX}}\left(0,0\right)>0$
in terms of the Bergman kernel of the model Kodaira Laplacian on $HX$
(see \cite[Sec. A]{Marinescu-Savale18}).
\end{thm}
Following this and the localization property of the Szeg\H{o} kernel
just proved now implies \prettyref{thm:Szego kernel expansion theorem}
as below.
\begin{proof}[Proof of \prettyref{thm:Szego kernel expansion theorem}]
 By the localization property \prettyref{eq:localization Szego kernel}
it suffices to show the pointwise expansion of the approximate Szeg\H{o}
kernel \prettyref{eq:localized Szego kernel}. In showing the expansion
at $x\in X$, we may further assume the BRT cover and partition of
unity defining \prettyref{eq:localized Szego kernel} is chosen so
that $\chi_{j}=\begin{cases}
1, & j=1\\
0, & j>1
\end{cases}$, near $x$. We then compute 
\begin{align}
\tilde{\Pi}_{b,m} & =\int_{0}^{2\pi}d\theta\,\tilde{\Pi}_{m}\left(x,xe^{i\theta}\right)e^{im\theta}\nonumber \\
 & =\sum_{l=0}^{s_{x}-1}\int_{\frac{2\pi}{s_{x}}l}^{\frac{2\pi}{s_{x}}\left(l+1\right)}d\theta\,\tilde{\Pi}_{m}\left(x,xe^{i\theta}\right)e^{im\theta}\nonumber \\
 & =\sum_{l=0}^{s_{x}-1}e^{i\frac{2\pi lm}{s_{x}}}\int_{0}^{\frac{2\pi}{s_{x}}l}d\theta\,\tilde{\Pi}_{m}\left(x,xe^{i\theta}\right)e^{im\theta}\nonumber \\
 & =\sum_{l=0}^{s_{x}-1}e^{i\frac{2\pi lm}{s_{x}}}\int_{0}^{\frac{2\pi}{s_{x}}l}d\theta\,\tilde{\Pi}_{m}\left(x,xe^{i\theta}\right)e^{im\theta}\nonumber \\
 & =\left(\sum_{l=0}^{s_{x}-1}e^{i\frac{2\pi lm}{s_{x}}}\right)\left(\int d\theta\,\sigma_{1}\left(\theta\right)\right)\tilde{B}_{m}\left(y,y\right)\nonumber \\
 & =\left(\sum_{l=0}^{s_{x}-1}e^{i\frac{2\pi lm}{s_{x}}}\right)m^{2/r_{y}}\left[\sum_{j=0}^{N}c_{j}\left(x\right)m^{-2j/r_{y}}+O\left(m^{-\left(2N+1\right)/r_{y}}\right)\right],\quad\forall N\in\mathbb{N},\label{eq:proof of expansion}
\end{align}
from \prettyref{eq:localized Szego kernel}, \prettyref{thm:local Bergman expansion}
to prove \prettyref{eq:Szego expansion}.
\end{proof}
As noted in \cite[Rem. 24]{Marinescu-Savale18} the expansion for
the local Bergman kernels $\tilde{B}_{m}\left(y,y\right)$ is the
same as the positive case on $X_{2,s}$ (the strongly pseudoconvex
points) and furthermore uniform in any $C^{l}$-topology on compact
subsets of $X_{2,s}$ cf.\ \cite[Theorem 4.1.1]{Ma-Marinescu}. In
particular the first two coefficients for $y\in Y_{2}$ are given
by 
\begin{align}
c_{0}\left(y\right) & =\Pi^{g_{x}^{HX},j^{r_{x}-2}\mathscr{L},J^{HX}}\left(0,0\right)=\frac{1}{2\pi}\tau^{L}\nonumber \\
c_{1}\left(y\right) & =\frac{1}{16\pi}\tau^{L}\left[\kappa-\Delta\ln\tau^{L}\right].\label{eq: recovering positive case}
\end{align}
The derivative expansion on $X_{2,s}$ is also known to satisfy $c_{0}=c_{1}=\ldots=c_{\left[\frac{l-1}{2}\right]}=0$
(i.e. begins at the same leading order $m$). 

In the next section we shall also need uniform estimates on the local
Bergman kernels as below.
\begin{thm}
\label{thm:uniform est. on local Bergman} The local Bergman kernel
satisfies the estimate 
\begin{equation}
\left[\inf_{x\in X_{r,\leq s}}\Pi^{g_{x}^{HX},j^{r_{x}-2}\mathscr{L},J^{HX}}\left(0,0\right)\right]\left[1+o\left(1\right)\right]m^{2/r}\leq\tilde{B}_{m}\left(y,y\right)\leq\left[\sup_{x\in X}\Pi^{g_{x}^{HX},j^{r_{x}-2}\mathscr{L},J^{HX}}\left(0,0\right)\right]\left[1+o\left(1\right)\right]m\label{eq:uniform Bergman estimate}
\end{equation}
with the $o\left(1\right)$ terms being uniform in $x\in X$. 

Furthermore, there exists constants $C_{l}$, $l=0,1,\ldots$, uniform
in $y\in Y$, such that for any differential operator $P_{l}$ of
order $l$, the derivative of the local Bergman kernel satisfies the
estimate
\begin{equation}
\left|P_{l}\tilde{B}_{m}\left(y,y\right)\right|\leq m^{//3}C_{l}\tilde{B}_{m}\left(y,y\right).\label{eq:uniform derivative estimate}
\end{equation}
\end{thm}
\begin{proof}
Note that theorem \prettyref{thm:local Bergman expansion} already
shows 
\begin{equation}
\Pi_{m}\left(y,y\right)\geq C_{r_{y}}\left(\left|j^{r_{y}-2}R^{L}\right|m\right)^{2/r_{y}}-c_{y}\label{eq: first non-uniform estimate}
\end{equation}
$\forall y\in Y$, with $c_{y}=c\left(\left|j^{r_{y}-2}R^{L}\left(y\right)\right|^{-1}\right)=O_{\left|j^{r_{y}-2}R^{L}\left(y\right)\right|^{-1}}\left(1\right)$
being a ($y$-dependent) constant given in terms of the norm of the
first non-vanishing jet. The norm of this jet affects the choice of
$\varrho$ needed for \prettyref{eq: curv =00003D000026 jet comp.};
which in turn affects the $C^{\infty}$-norms of the coefficients
of \prettyref{eq:local Bergman expansion} via \prettyref{eq:modified connection-1}.
We first show that this estimate extends to a small ($\left|j^{r_{y}-2}R^{L}\left(y\right)\right|$-
dependent) size neighborhood of $y$. To this end, for any $\varepsilon>0$
there exists a uniform constant $c_{\varepsilon}$ depending only
on $\varepsilon$ and$\left\Vert R^{L}\right\Vert _{C^{r}}$ such
that 
\begin{equation}
\left|j^{r_{y}-2}R^{L}\left(\mathsf{y}\right)\right|\geq\left(1-\varepsilon\right)\left|j^{r_{y}-2}R^{L}\left(y\right)\right|,\label{eq: first jet comparable}
\end{equation}
$\forall\mathsf{y}\in B_{c_{\varepsilon}\left|j^{r_{y}-2}R^{L}\right|}\left(y\right).$

We begin by rewriting the model Kodaira Laplacian $\tilde{\Box}_{m}$
\prettyref{eq:modified local Bochner} near $y$ in terms of geodesic
coordinates centered at $\mathsf{y}$. In the region 
\[
\mathsf{y}\in B_{c_{\varepsilon}\left|j^{r_{y}-2}R^{L}\right|}\left(y\right)\cap\left\{ C_{0}\left(\left|j^{0}R^{L}\left(\mathsf{y}\right)\right|m\right)\geq m^{2/r_{y}}\Pi^{g_{y}^{TY},j_{y}^{r_{y}-2}R^{L},J_{y}^{TY}}\left(0,0\right)\right\} 
\]
a rescaling of $\tilde{\Box}_{m}$ by $\delta_{m^{-1/2}}$, now centered
at $\mathsf{y}$, shows 
\begin{align}
\Pi_{m}\left(\mathsf{y},\mathsf{y}\right) & =m\Pi^{g_{\mathsf{y}}^{TY},j_{\mathsf{y}}^{0}R^{L},J_{\mathsf{y}}^{TY}}\left(0,0\right)+O_{\left|j^{r_{y}-2}R^{L}\left(y\right)\right|^{-1}}\left(1\right)\nonumber \\
 & =m\left|j^{0}R^{L}\left(\mathsf{y}\right)\right|\Pi^{g_{\mathsf{y}}^{TY},\frac{j_{\mathsf{y}}^{0}R^{L}}{\left|j^{0}R^{L}\left(\mathsf{y}\right)\right|},J_{\mathsf{y}}^{TY}}\left(0,0\right)+O_{\left|j^{r_{y}-2}R^{L}\left(y\right)\right|^{-1}}\left(1\right)\nonumber \\
 & \geq m^{2/r_{y}}\Pi^{g_{y}^{TY},j_{y}^{r_{y}-2}R^{L},J_{y}^{TY}}\left(0,0\right)+O_{\left|j^{r_{y}-2}R^{L}\left(y\right)\right|^{-1}}\left(1\right)\label{eq:est. region 1}
\end{align}
as in \prettyref{eq: first non-uniform estimate}. Now, in the region
\begin{align*}
\mathsf{y} & \in B_{c_{\varepsilon}\left|j^{r_{y}-2}R^{L}\right|}\left(y\right)\cap\left\{ C_{1}\left(\left|j^{1}R^{L}\left(\mathsf{y}\right)/j^{0}R^{L}\left(\mathsf{y}\right)\right|m\right)^{2/3}\right.\\
 & \qquad\left.\geq m^{2/r_{y}}\Pi^{g_{y}^{TY},j_{y}^{r_{y}-2}R^{L},J_{y}^{TY}}\left(0,0\right)\geq C_{0}\left(\left|j^{0}R^{L}\left(\mathsf{y}\right)\right|m\right)\right\} 
\end{align*}
a rescaling of $\tilde{\Box}_{m}$ by $\delta_{m^{-1/3}}$ centered
at $\mathsf{y}$ similarly shows 
\begin{align}
\Pi_{m}\left(\mathsf{y},\mathsf{y}\right) & =m^{2/3}\left[1+O\left(m^{2/r-2/3}\right)\right]\Pi^{g_{\mathsf{y}}^{TY},j_{\mathsf{y}}^{1}R^{L}/j_{\mathsf{y}}^{0}R^{L},J_{\mathsf{y}}^{TY}}\left(0,0\right)\nonumber \\
 & \qquad+O_{\left|j^{r_{y}-2}R^{L}\left(y\right)\right|^{-1}}\left(1\right)\nonumber \\
 & =m^{2/3}\left[1+O\left(m^{2/r-2/3}\right)\right]\left|j_{\mathsf{y}}^{1}R^{L}/j_{\mathsf{y}}^{0}R^{L}\right|^{2/3}\Pi^{g_{\mathsf{y}}^{TY},\frac{j_{\mathsf{y}}^{1}R^{L}/j_{\mathsf{y}}^{0}R^{L}}{\left|j_{\mathsf{y}}^{1}R^{L}/j_{\mathsf{y}}^{0}R^{L}\right|},J_{\mathsf{y}}^{TY}}\left(0,0\right)\nonumber \\
 & \qquad+O_{\left|j^{r_{y}-2}R^{L}\left(y\right)\right|^{-1}}\left(1\right)\\
 & \geq\left(1-\varepsilon\right)m^{2/r_{y}}\Pi^{g_{y}^{TY},j_{y}^{r_{y}-2}R^{L},J_{y}^{TY}}\left(0,0\right)+O_{\left|j^{r_{y}-2}R^{L}\left(y\right)\right|^{-1}}\left(1\right)\label{eq: est. region 2}
\end{align}
Next, in the region 
\begin{align*}
\mathsf{y} & \in B_{c_{\varepsilon}\left|j^{r_{y}-2}R^{L}\right|}\left(y\right)\cap\left\{ C_{2}\left(\left|j^{2}R^{L}\left(\mathsf{y}\right)/j^{1}R^{L}\left(\mathsf{y}\right)\right|m\right)^{1/2}\right.\\
 & \left.\geq m^{2/r_{y}}\Pi^{g_{y}^{TY},j_{y}^{r_{y}-2}R^{L},J_{y}^{TY}}\left(0,0\right)\geq\max\left[C_{0}\left(\left|j^{0}R^{L}\left(\mathsf{y}\right)\right|m\right),C_{1}\left(\left|j^{1}R^{L}\left(\mathsf{y}\right)/j^{0}R^{L}\left(\mathsf{y}\right)\right|m\right)^{2/3}\right]\right\} 
\end{align*}
a rescaling of $\tilde{\Box}_{m}$ by $\delta_{m^{-1/4}}$ centered
at $\mathsf{y}$ shows 
\begin{align}
\Pi_{m}\left(\mathsf{y},\mathsf{y}\right) & =m^{1/2}\left[1+O\left(m^{2/r-1/2}\right)\right]\Pi^{g_{\mathsf{y}}^{TY},j_{\mathsf{y}}^{2}R^{L}/j_{\mathsf{y}}^{1}R^{L},J_{\mathsf{y}}^{TY}}\left(0,0\right)+O_{\left|j^{r_{y}-2}R^{L}\left(y\right)\right|^{-1}}\left(1\right)\nonumber \\
 & =m^{1/2}\left[1+O\left(m^{2/r-1/2}\right)\right]\left|j_{\mathsf{y}}^{2}R^{L}/j_{\mathsf{y}}^{1}R^{L}\right|^{1/2}\Pi^{g_{\mathsf{y}}^{TY},\frac{j_{\mathsf{y}}^{2}R^{L}/j_{\mathsf{y}}^{1}R^{L}}{\left|j_{\mathsf{y}}^{2}R^{L}/j_{\mathsf{y}}^{1}R^{L}\right|},J_{\mathsf{y}}^{TY}}\left(0,0\right)+O_{\left|j^{r_{y}-2}R^{L}\left(y\right)\right|^{-1}}\left(1\right)\nonumber \\
 & \geq\left(1-\varepsilon\right)m^{2/r_{y}}\Pi^{g_{y}^{TY},j_{y}^{r_{y}-2}R^{L},J_{y}^{TY}}\left(0,0\right)+O_{\left|j^{r_{y}-2}R^{L}\left(y\right)\right|^{-1}}\left(1\right)\label{eq: est. region 3}
\end{align}
Continuing in this fashion, we are finally left with the region 
\begin{align*}
\mathsf{y} & \in B_{c_{\varepsilon}\left|j^{r_{y}-2}R^{L}\right|}\left(y\right)\cap\left\{ m^{2/r_{y}}\Pi^{g_{y}^{TY},j_{y}^{r_{y}-2}R^{L},J_{y}^{TY}}\left(0,0\right)\right.\\
 & \left.\geq\max\left[C_{0}\left(\left|j^{0}R^{L}\left(\mathsf{y}\right)\right|m\right),\ldots,C_{r_{y}-3}\left(\left|j^{r_{y}-3}R^{L}\left(\mathsf{y}\right)/j^{r_{y}-4}R^{L}\left(\mathsf{y}\right)\right|m\right)^{2/\left(r_{y}-1\right)}\right]\right\} .
\end{align*}
In this region we have 
\[
\left|j^{r_{y}-2}R^{L}\left(\mathsf{y}\right)/j^{r_{y}-3}R^{L}\left(\mathsf{y}\right)\right|\geq\left(1-\varepsilon\right)\left|j^{r_{y}-2}R^{L}\left(y\right)\right|+O\left(m^{2/r_{y}-2/\left(r_{y}-1\right)}\right)
\]
following \prettyref{eq: first jet comparable} with the remainder
being uniform. A rescaling by $\delta_{m^{-1/r_{y}}}$ then giving
a similar estimate in this region, we have finally arrived at 
\[
\Pi_{m}\left(\mathsf{y},\mathsf{y}\right)\geq\left(1-\varepsilon\right)m^{2/r_{y}}\Pi^{g_{y}^{TY},j_{y}^{r_{y}-2}R^{L},J_{y}^{TY}}\left(0,0\right)+O_{\left|j^{r_{y}-2}R^{L}\left(y\right)\right|^{-1}}\left(1\right)
\]
$\forall\mathsf{y}\in B_{c_{\varepsilon}\left|j^{r_{y}-2}R^{L}\right|}\left(y\right)$.

Finally a compactness argument finds a finite set of points $\left\{ y_{j}\right\} _{j=1}^{N}$
such that the corresponding $B_{c_{\varepsilon}\left|j^{r_{y_{j}}-2}R^{L}\right|}\left(y_{j}\right)$'s
cover $Y$. This gives a uniform constant $c_{1,\varepsilon}>0$ such
that 
\[
\Pi_{m}\left(y,y\right)\geq\left(1-\varepsilon\right)\left[\inf_{y\in Y_{r}}\Pi^{g_{y}^{TY},j_{y}^{r-2}R^{L},J_{y}^{TY}}\left(0,0\right)\right]m^{2/r}-c_{1,\varepsilon}
\]
$\forall y\in Y$ , $\varepsilon>0$ proving the lower bound of \prettyref{eq:uniform Bergman estimate}.
The argument for the upper bound is similar. 

The proof of the uniform estimate on the derivative \prettyref{eq:uniform derivative estimate}
is similar. Given $\varepsilon>0$ we find a uniform $c_{\varepsilon}$
such that \prettyref{eq: first jet comparable} holds for each $y\in Y$
and $\mathsf{y}\in B_{c_{\varepsilon}\left|j^{r_{y}-2}R^{L}\right|}\left(y\right)$.
Then rewrite the model Kodaira Laplacian $\tilde{\Box}_{m}$ \prettyref{eq:modified local Bochner}
near $y$ in terms of geodesic coordinates centered at $\mathsf{y}$.
In the region 
\[
\mathsf{y}\in B_{c_{\varepsilon}\left|j^{r_{y}-2}R^{L}\right|}\left(y\right)\cap\left\{ C_{0}\left(\left|j^{0}R^{L}\left(\mathsf{y}\right)\right|m\right)\geq m^{2/r_{y}}\Pi^{g_{y}^{TY},j_{y}^{r_{y}-2}R^{L},J_{y}^{TY}}\left(0,0\right)\right\} 
\]
a rescaling of $\tilde{\Box}_{m}$ by $\delta_{m^{-1/2}}$, now centered
at $\mathsf{y}$, shows 
\begin{align*}
\partial^{\alpha}\Pi_{m}\left(\mathsf{y},\mathsf{y}\right) & =\frac{m}{2\pi}\left(\partial^{\alpha}\tau^{L}\left(\mathsf{y}\right)\right)+O_{\left|j^{r_{y}-2}R^{L}\left(y\right)\right|^{-1}}\left(1\right)
\end{align*}
following \prettyref{eq: recovering positive case} as $r_{\mathsf{y}}=2$.
Diving the above by \prettyref{eq:est. region 1} gives 
\begin{align*}
\frac{\left|\partial^{\alpha}\Pi_{m}\left(\mathsf{y},\mathsf{y}\right)\right|}{\Pi_{m}\left(\mathsf{y},\mathsf{y}\right)} & \leq\frac{\left|\partial^{\alpha}\tau^{L}\left(\mathsf{y}\right)\right|}{\tau^{L}\left(\mathsf{y}\right)}+O_{\left|j^{r_{y}-2}R^{L}\left(y\right)\right|^{-1}}\left(m^{-1}\right)\\
 & \leq m^{\left|\alpha\right|/3}\left[\sup_{y\in Y}\frac{\left|\left[j^{\left|\alpha\right|}\Pi^{g_{y}^{TY},j_{y}^{1}R^{L}/j_{y}^{0}R^{L},J_{y}^{TY}}\right]\left(0,0\right)\right|}{\Pi^{g_{y}^{TY},j_{y}^{1}R^{L}/j_{y}^{0}R^{L},J_{y}^{TY}}\left(0,0\right)}\right]\Pi_{m}\left(\mathsf{y},\mathsf{y}\right)\\
 & +O_{\left|j^{r_{y}-2}R^{L}\left(y\right)\right|^{-1}}\left(m^{-1}\right)
\end{align*}
Next, in the region 
\begin{align*}
\mathsf{y} & \in B_{c_{\varepsilon}\left|j^{r_{y}-2}R^{L}\right|}\left(y\right)\cap\left\{ C_{1}\left(\left|j^{1}R^{L}\left(\mathsf{y}\right)/j^{0}R^{L}\left(\mathsf{y}\right)\right|m\right)^{2/3}\right.\\
 & \left.\geq m^{2/r_{y}}\Pi^{g_{y}^{TY},j_{y}^{r_{y}-2}R^{L},J_{y}^{TY}}\left(0,0\right)\geq C_{0}\left(\left|j^{0}R^{L}\left(\mathsf{y}\right)\right|m\right)\right\} 
\end{align*}
a rescaling of $\tilde{\Box}_{m}$ by $\delta_{m^{-1/3}}$ centered
at $\mathsf{y}$ similarly shows 
\begin{align*}
\partial^{\alpha}\Pi_{m}\left(\mathsf{y},\mathsf{y}\right) & =m^{\left(2+\left|\alpha\right|\right)/3}\left[1+O\left(m^{2/r-2/3}\right)\right]\left[\partial^{\alpha}\Pi^{g_{\mathsf{y}}^{TY},j_{\mathsf{y}}^{1}R^{L}/j_{\mathsf{y}}^{0}R^{L},J_{\mathsf{y}}^{TY}}\right]\left(0,0\right)\\
 & \qquad+O_{\left|j^{r_{y}-2}R^{L}\left(y\right)\right|^{-1}}\left(m^{\left(1+\left|\alpha\right|\right)/3}\right).
\end{align*}
Dividing this by \prettyref{eq: est. region 2} gives 
\begin{align*}
\frac{\left|\partial^{\alpha}\Pi_{m}\left(\mathsf{y},\mathsf{y}\right)\right|}{\Pi_{m}\left(\mathsf{y},\mathsf{y}\right)} & \leq m^{\left|\alpha\right|/3}\left(1+\varepsilon\right)\frac{\left|\left[\partial^{\alpha}\Pi^{g_{\mathsf{y}}^{TY},j_{\mathsf{y}}^{1}R^{L}/j_{\mathsf{y}}^{0}R^{L},J_{\mathsf{y}}^{TY}}\right]\left(0,0\right)\right|}{\left[\Pi^{g_{\mathsf{y}}^{TY},j_{\mathsf{y}}^{1}R^{L}/j_{\mathsf{y}}^{0}R^{L},J_{\mathsf{y}}^{TY}}\right]\left(0,0\right)}\\
 & \qquad+O_{\left|j^{r_{y}-2}R^{L}\left(y\right)\right|^{-1}}\left(m^{\left(\left|\alpha\right|-1\right)/3}\right)\\
 & \leq m^{\left|\alpha\right|/3}\left(1+\varepsilon\right)\left[\sup_{y\in Y}\frac{\left|\left[j^{\left|\alpha\right|}\Pi^{g_{y}^{TY},j_{y}^{1}R^{L}/j_{y}^{0}R^{L},J_{y}^{TY}}\right]\left(0,0\right)\right|}{\Pi^{g_{y}^{TY},j_{y}^{1}R^{L}/j_{y}^{0}R^{L},J_{y}^{TY}}\left(0,0\right)}\right]\\
 & \qquad+O_{\left|j^{r_{y}-2}R^{L}\left(y\right)\right|^{-1}}\left(m^{\left(\left|\alpha\right|-1\right)/3}\right).
\end{align*}
Continuing in this fashion as before eventually gives

\begin{align*}
\frac{\left|\partial^{\alpha}\Pi_{m}\left(\mathsf{y},\mathsf{y}\right)\right|}{\Pi_{m}\left(\mathsf{y},\mathsf{y}\right)} & \leq m^{\left|\alpha\right|/3}\left(1+\varepsilon\right)\left[\sup_{y\in Y}\frac{\left|\left[j^{\left|\alpha\right|}\Pi^{g_{y}^{TY},j_{y}^{1}R^{L}/j_{y}^{0}R^{L},J_{y}^{TY}}\right]\left(0,0\right)\right|}{\Pi^{g_{y}^{TY},j_{y}^{1}R^{L}/j_{y}^{0}R^{L},J_{y}^{TY}}\left(0,0\right)}\right]\\
 & \qquad+O_{\left|j^{r_{y}-2}R^{L}\left(y\right)\right|^{-1}}\left(m^{\left(\left|\alpha\right|-1\right)/3}\right)
\end{align*}
$\forall y\in Y$, $\mathsf{y}\in B_{c_{\varepsilon}\left|j^{r_{y}-2}R^{L}\right|}\left(y\right)$,
$\forall\alpha\in\mathbb{N}_{0}^{2}$. By compactness one again finds
a uniform $c_{1,\varepsilon}$ such that 
\[
\frac{\left|\partial^{\alpha}\Pi_{m}\left(y,y\right)\right|}{\Pi_{m}\left(y,y\right)}\leq m^{\left|\alpha\right|/3}\left(1+\varepsilon\right)\left[\sup_{y\in Y}\frac{\left|\left[j^{\left|\alpha\right|}\Pi^{g_{y}^{TY},j_{y}^{1}R^{L}/j_{y}^{0}R^{L},J_{y}^{TY}}\right]\left(0,0\right)\right|}{\Pi^{g_{y}^{TY},j_{y}^{1}R^{L}/j_{y}^{0}R^{L},J_{y}^{TY}}\left(0,0\right)}\right]+c_{1,\varepsilon}
\]
$\forall y\in Y$, proving the lemma. 
\end{proof}

\section{\label{sec:Equivariant-CR-Embedding}Equivariant CR Embedding }

In this section we construct the CR embedding for $X$ required to
prove \prettyref{thm: main embedding thm}. Firstly, setting $m=p.\left(s!\right)\in\left(s!\right).\mathbb{N}_{0}$
in \prettyref{eq:hodge theorem}, \prettyref{eq:Bergman kernel in eigenfunctions},
\prettyref{eq:uniform Bergman estimate} and \prettyref{eq:proof of expansion}
the base locus 
\begin{equation}
\textrm{Bl}_{p}\left(X\right)\coloneqq\left\{ x\in X|s\left(x\right)=0,\,\forall s\in H_{b,p.\left(s!\right)}^{0}\left(X\right)\right\} =\emptyset\label{eq: base locus  is empty}
\end{equation}
is empty for $p\gg0$. Thus the subspace 
\begin{equation}
\Phi_{p,x}\coloneqq\left\{ s\in H_{b,p.\left(s!\right)}^{0}\left(X\right)|s\left(x\right)=0\right\} \subset H_{b,p.\left(s!\right)}^{0}\left(X\right)\label{eq:Kodaira hyperplane}
\end{equation}
is a hyperplane for each $x\in X$. Identifying the Grassmanian $\mathbb{G}\left(d_{p.\left(s!\right)}-1;H_{b,p.\left(s!\right)}^{0}\left(X\right)\right)$,
$d_{p.\left(s!\right)}=\textrm{dim }H_{b,p.\left(s!\right)}^{0}\left(X\right)$,
with the projective space $\mathbb{P}\left[H_{b,p.\left(s!\right)}^{0}\left(X\right)^{*}\right]$,
by sending a non-zero element of $H_{b,p.\left(s!\right)}^{0}\left(X\right)^{*}$
to its kernel, gives a well-defined Kodaira map 
\begin{equation}
\Phi_{p}:X\rightarrow\mathbb{P}\left[H_{b,p.\left(s!\right)}^{0}\left(X\right)^{*}\right]\label{eq:Kodaira map}
\end{equation}
for $p\gg0$.

In terms of the basis $\left\{ \psi_{1}^{p.\left(s!\right)},\ldots,\psi_{d_{p.\left(s!\right)}}^{p.\left(s!\right)}\right\} $
of $\textrm{ker}\left(\Box_{b,p.\left(s!\right)}^{0}\right)=H_{b,p.\left(s!\right)}^{0}\left(X\right)$,
with $\bar{\partial}_{b}\psi_{j}=0$, $j=1,\ldots,d_{p.\left(s!\right)}$,
and corresponding dual basis of $H_{b,p.\left(s!\right)}^{0}\left(X\right)^{*}$
the map is written 
\begin{equation}
\Phi_{p}\left(x\right)\coloneqq\left(\psi_{1}^{p.\left(s!\right)}\left(x\right),\ldots,\psi_{d_{p.\left(s!\right)}}^{p.\left(s!\right)}\left(x\right)\right)\in\mathbb{C}^{d_{p.\left(s!\right)}}\label{eq:Kodaira map in triv.}
\end{equation}
and is seen to be CR. 

We now define the augmented Kodaira map 
\begin{align}
\Psi_{p}: & X\rightarrow\mathbb{C}^{N},\nonumber \\
\Psi_{p}\coloneqq & \left(\Phi_{p},\Psi_{p}^{1},\ldots,\Psi_{p}^{s}\right),\nonumber \\
\Psi_{p}^{k}\coloneqq & \left(\underbrace{\psi_{1}^{p.k},\ldots,\psi_{d_{p.k}}^{p.k}}_{\eqqcolon\Psi_{p}^{k,0}};\underbrace{\psi_{1}^{\left(p+1\right).k},\ldots,\psi_{d_{\left(p+1\right).k}}^{\left(p+1\right).k}}_{\eqqcolon\Psi_{p}^{k,1}}\right),\;1\le k\leq s,\nonumber \\
N\coloneqq & d_{p.\left(s!\right)}+\sum_{k=1}^{s}\left(d_{p.k}+d_{\left(p+1\right).k}\right),\label{eq:augmented Kodaira map}
\end{align}
which is again CR. We shall now show that the above augmented map
is an embedding for $p\gg0$. We first show that it is an immersion,
whereby it suffices to show that its first component $\Phi_{p}$ \prettyref{eq:Kodaira map in triv.}
defines an immersion; the augmented components of \prettyref{eq:augmented Kodaira map}
are required to separate further points.
\begin{thm}
The Kodaira map $\Phi_{p}$ \prettyref{eq:Kodaira map} is an immersion
for $p\gg0$.
\end{thm}
\begin{proof}
We work in a BRT trivialization \prettyref{eq: BRT condition}, \prettyref{eq: BRT condition}
near $x\in X$. We choose $\chi\in C_{c}^{\infty}\left(\left(-\varepsilon,\varepsilon\right);\left[0,1\right]\right)$
, $\chi=1$ on $\left(-\frac{\varepsilon}{2},\frac{\varepsilon}{2}\right)$,
$\sigma_{0}\left(\theta\right)\in C_{c}^{\infty}\left(-\frac{2\pi}{s},\frac{2\pi}{s}\right)$,
$\int\theta\sigma_{0}\left(\theta\right)d\theta=1$ and set 
\begin{align}
u_{1} & =y_{2}\chi\left(m^{1/r_{x}}y\right)\sigma\left(\theta\right)e^{im\theta}\nonumber \\
u_{2} & =y_{1}\chi\left(m^{1/r_{x}}y\right)\sigma\left(\theta\right)e^{im\theta}\nonumber \\
u_{3} & =\chi\left(m^{1/r_{x}}y\right)\left(m\theta\right)\sigma_{0}\left(m\theta\right)e^{im\theta}\quad\textrm{ and }\nonumber \\
v_{j} & =\Pi_{m,b}u_{j},\quad\quad j=1,2,3,\label{eq:embedding CR functions}
\end{align}
with $m=p.\left(s!\right)$ and $x=\left(y,\theta\right)$ being BRT
coordinates.

The equations $\bar{\partial}_{b}\Pi_{m,b}\left(.,x\right)=0$, $\bar{\partial}_{b}^{*}\Pi_{m,b}\left(.,x\right)=0$
written in the BRT chart give
\begin{align}
\partial_{y_{1}}\Pi_{m,b}\left(x',x\right) & =m\left[\frac{1}{r_{x}}y_{2}'\left(\sum_{\left|\alpha\right|=r_{x}-2}R_{\alpha}^{L}y'^{\alpha}\right)+O\left(y'^{r_{x}}\right)\right]\Pi_{m,b}\left(x',x\right)\nonumber \\
\partial_{y_{2}}\Pi_{m,b}\left(x',x\right) & =m\left[-\frac{1}{r_{x}}y_{1}'\left(\sum_{\left|\alpha\right|=r_{x}-2}R_{\alpha}^{L}y'^{\alpha}\right)+O\left(y'^{r_{x}}\right)\right]\Pi_{m,b}\left(x',x\right)\label{eq:Derivative of Szego kernel}
\end{align}
from \prettyref{eq: BRT condition}, \prettyref{eq:connection expansion}.
Further note that 
\begin{align}
\overline{\Pi_{m,b}u_{j}\left(x\right)} & =\int dx'\,\overline{\Pi_{m,b}\left(x,x'\right)}\overline{u_{j}\left(x'\right)}\nonumber \\
 & =\int dx'\,\Pi_{m,b}\left(x',x\right)\overline{u_{j}\left(x'\right)}\nonumber \\
 & =s_{x}\int dx'\,\tilde{\Pi}_{m}\left(x',x\right)\overline{u_{j}\left(x'\right)}\quad\textrm{mod }O\left(m^{-\infty}\right)\label{eq: manipulation}
\end{align}
from \prettyref{eq:localization Szego kernel}. We now estimate the
derivative of the CR functions \prettyref{eq:embedding CR functions}.
Below $c_{ij}\left(\left|j^{r_{x}-2}\mathcal{L}\right|\right)$, $1\leq i,j\leq3$,
continuous positive functions of the norms of the jet of the Levi
form. We then have 
\begin{align}
\partial_{y_{1}}\overline{v_{1}} & =s_{x}m\int dy'd\theta\left[\frac{1}{r_{x}}\left(y_{2}'\right)^{2}\left(\sum_{\left|\alpha\right|=r_{x}-2}R_{\alpha}^{L}y'^{\alpha}\right)+O\left(y'^{r_{x}+1}\right)\right]\times\nonumber \\
 & \qquad\qquad\tilde{\Pi}_{m}\left(x',x\right)\chi\left(m^{1/r_{x}}y'\right)\sigma\left(\theta\right)e^{im\theta}+O\left(m^{-\infty}\right)\nonumber \\
 & =s_{x}\int dy'd\theta\left[\frac{1}{r_{x}}\left(y_{2}'\right)^{2}\left(\sum_{\left|\alpha\right|=r_{x}-2}R_{\alpha}^{L}y'^{\alpha}\right)+O\left(m^{-1/r_{x}}y'^{r_{x}+1}\right)\right]\times\nonumber \\
 & \qquad\qquad\tilde{\Pi}_{m}\left(m^{-1/r_{x}}x',0\right)\chi\left(y'\right)\sigma\left(\theta\right)e^{im\theta}\nonumber \\
 & =s_{x}\int dy'd\theta\left[\frac{1}{r_{x}}\left(y_{2}'\right)^{2}\left(\sum_{\left|\alpha\right|=r_{x}-2}R_{\alpha}^{L}y'^{\alpha}\right)+O\left(m^{-1/r_{x}}y'^{r_{x}+1}\right)\right]\times\nonumber \\
 & \qquad\qquad\left[\Pi^{g_{x}^{HX},j^{r_{x}-2}\mathscr{L},J^{HX}}\left(x',0\right)+O\left(m^{-1/r_{x}}\right)\right]\chi\left(y'\right)\sigma\left(\theta\right)e^{im\theta}\nonumber \\
 & \geq\varepsilon^{r_{x}+2}c_{11}\left(\left|j^{r_{x}-2}\mathcal{L}\right|\right)+O_{x}\left(m^{-1/r_{x}}\right)\label{eq:derivative 11}
\end{align}
using \prettyref{eq:Derivative of Szego kernel} and \prettyref{eq: manipulation}.

And similarly,
\begin{align}
\partial_{y_{2}}\overline{v_{1}} & =s_{x}m\int dy'd\theta\left[\frac{1}{r_{x}}y_{1}'y_{2}'\left(\sum_{\left|\alpha\right|=r_{x}-2}R_{\alpha}^{L}y'^{\alpha}\right)+O\left(y'^{r_{x}+1}\right)\right]\times\nonumber \\
 & \qquad\qquad\tilde{\Pi}_{m}\left(x',0\right)\chi\left(m^{1/r_{x}}y'\right)\sigma\left(\theta\right)e^{im\theta}+O\left(m^{-\infty}\right)\nonumber \\
 & =s_{x}\int dy'd\theta\left[\frac{1}{r_{x}}y_{1}'y_{2}'\left(\sum_{\left|\alpha\right|=r_{x}-2}R_{\alpha}^{L}y'^{\alpha}\right)+O\left(m^{-1/r_{x}}y'^{r_{x}+1}\right)\right]\times\nonumber \\
 & \qquad\qquad\tilde{\Pi}_{m}\left(m^{-1/r_{y}}x',0\right)\chi\left(y'\right)\sigma\left(\theta\right)e^{im\theta}\nonumber \\
 & =s_{x}\int dy'd\theta\left[\frac{1}{r_{x}}y_{1}'y_{2}'\left(\sum_{\left|\alpha\right|=r_{x}-2}R_{\alpha}^{L}y'^{\alpha}\right)+O\left(m^{-1/r_{x}}y'^{r_{x}+1}\right)\right]\times\nonumber \\
 & \qquad\qquad\left[\Pi^{g_{x}^{HX},j^{r_{x}-2}\mathscr{L},J^{HX}}\left(x',0\right)+O\left(m^{-1/r_{x}}\right)\right]\chi\left(y'\right)\sigma\left(\theta\right)e^{im\theta}\nonumber \\
 & =s_{x}\int dy'd\theta\left[O\left(y'^{r_{x}+1}\right)+O\left(m^{-1/r_{x}}\right)\right]\chi\left(y'\right)\sigma\left(\theta\right)e^{im\theta}\nonumber \\
 & \leq\varepsilon^{r_{y}+3}c_{12}\left(\left|j^{r_{x}-2}\mathcal{L}\right|\right)+O_{x}\left(m^{-1/r_{x}}\right)\label{eq:derivative 12}
\end{align}
using a Taylor expansion $\Pi^{g_{x}^{HX},j^{r_{x}-2}\mathscr{L},J^{HX}}\left(y,0\right)=\Pi_{0}+y_{1}\Pi_{1}+y_{2}\Pi_{2}$
; $\Pi_{0}=\Pi^{g_{x}^{HX},j^{r_{x}-2}\mathscr{L},J^{HX}}\left(0,0\right)$
on $\textrm{spt}\left(\chi\right)$. Finally we compute
\begin{align*}
\partial_{\theta}\overline{v_{1}} & =s_{x}m\int dy'd\theta\tilde{\Pi}_{m}\left(x',x\right)y_{1}'\chi\left(m^{1/r_{x}}y'\right)\sigma\left(\theta\right)e^{im\theta}\\
 & =s_{x}m^{1-1/r_{x}}\int dy'd\theta m^{-2/r_{x}}\tilde{\Pi}_{m}\left(m^{-1/r_{y}}x',0\right)y_{1}'\chi\left(y'\right)\sigma\left(\theta\right)e^{im\theta}\\
 & =s_{x}m^{1-1/r_{x}}\int dy'd\theta\left[\Pi^{g_{x}^{HX},j^{r_{x}-2}\mathscr{L},J^{HX}}\left(x',0\right)+O\left(m^{-1/r_{x}}\right)\right]y_{1}'\chi\left(y'\right)\sigma\left(\theta\right)e^{im\theta}\\
 & \leq m^{1-1/r_{x}}\varepsilon^{4}c_{13}\left(\left|j^{r_{x}-2}\mathcal{L}\right|\right)+O_{x}\left(m^{1-2/r_{x}}\right).
\end{align*}
We have similar estimates on derivatives of $v_{2}$ 
\begin{align}
\partial_{y_{1}}v_{2} & \leq\varepsilon^{r_{y}+2}c_{21}\left(\left|j^{r_{y}-2}R^{L}\right|\right)+O_{x}\left(m^{-1/r_{x}}\right)\nonumber \\
\partial_{y_{2}}v_{2} & \geq\varepsilon^{r_{y}+2}c_{22}\left(\left|j^{r_{y}-2}R^{L}\right|\right)+O_{x}\left(m^{-1/r_{x}}\right)\nonumber \\
\partial_{\theta}v_{2} & \leq m^{1-1/r_{x}}\varepsilon^{4}c_{23}\left(\left|j^{r_{x}-2}\mathcal{L}\right|\right)+O_{x}\left(m^{1-2/r_{x}}\right).\label{eq: derivatives 21 =000026 22}
\end{align}
for two further constants $c_{21}\left(\left|j^{r_{y}-2}R^{L}\right|\right)$,
$c_{22}C\left(\left|j^{r_{y}-2}R^{L}\right|\right)\left(\left|j^{r_{y}-2}R^{L}\right|\right)$
depending only on the norm of jet of the Levi tensor at $x$.

Finally, and in similar vein, we estimate the derivative of $v_{3}$
\begin{align}
\partial_{\theta}v_{3} & =m\int dzd\theta\tilde{\Pi}_{b,m}\left(z,0\right)\chi\left(m^{1/r_{x}}y\right)\left(m\theta\right)\sigma_{0}\left(m\theta\right)\nonumber \\
 & =\int dzd\theta\left[\Pi^{g_{x}^{HX},j^{r_{x}-2}\mathscr{L},J^{HX}}\left(y,0\right)+O\left(m^{-1/r_{x}}\right)\right]\chi\left(y\right)\theta\sigma_{0}\left(\theta\right)\nonumber \\
 & \geq\varepsilon^{2}c_{33}\left(\left|j^{r_{y}-2}R^{L}\right|\right)+O_{x}\left(m^{-1/r_{x}}\right)\label{eq:derivative 33}
\end{align}
and 
\begin{align}
\partial_{y_{1}}v_{3} & =m\int dzd\theta\tilde{\Pi}_{b,m}\left(z,0\right)\chi\left(m^{1/r_{x}}y\right)\left(m\theta\right)\sigma_{0}\left(m\theta\right)\nonumber \\
 & =m\int dzd\theta\left[\frac{1}{r_{x}}y_{2}\left(\sum_{\left|\alpha\right|=r_{x}-2}R_{\alpha}^{L}y^{\alpha}\right)+O\left(y^{r_{x}}\right)\right]\times\nonumber \\
 & \qquad\qquad\tilde{\Pi}_{b,m}\left(z,0\right)\chi\left(m^{1/r_{x}}z\right)\left(m\theta\right)\sigma_{0}\left(m\theta\right)\nonumber \\
 & =\int dzd\theta\left[O\left(m^{-1}m^{1/r_{r}}y^{r_{x}-1}\right)\right]\tilde{\Pi}_{b,m}\left(m^{-1/r_{y}}z,0\right)\chi\left(z\right)\theta\sigma\left(\theta\right)\nonumber \\
 & \leq m^{-1+1/r_{x}}\varepsilon^{r_{x}+1}c_{13}\left(\left|j^{r_{y}-2}R^{L}\right|\right)+O_{x}\left(m^{-1+2/r_{x}}\right).\label{eq:derivative 31}
\end{align}
and similarly for$\partial_{y_{2}}v_{3}.$ Following these estimates
there exists $C\left(\left|j^{r_{y}-2}R^{L}\right|\right)$ such that
the differential of $x\mapsto\left(v_{1},v_{2},v_{3}\right)$, and
thus of $\Phi_{p}$, $m=p.\left(s!\right)$, is invertible at $x$
for $\varepsilon<C\left(\left|j^{r_{y}-2}R^{L}\right|\right)$ and
$p>C\left(\left|j^{r_{y}-2}R^{L}\right|\right)$. Thus for some $C_{1}\left(\left|j^{r_{y}-2}R^{L}\right|\right)$
the differential of $\Phi_{p}$ is invertible on a $C_{1}\left(\left|j^{r_{y}-2}R^{L}\right|\right)$
ball centered at $x$ for $p>C_{1}\left(\left|j^{r_{y}-2}R^{L}\right|\right)$;
which completes the proof following a compactness argument.
\end{proof}
Next to show the Kodaira map is injective, one needs the following
definition.
\begin{defn}
The peak function $S_{x_{0}}^{p}\in H_{b,p.\left(s!\right)}^{0}\left(X\right)$
at $x_{0}\in X$ is the unit norm element of the orthogonal complement
to $\Phi_{p,x_{0}}^{\perp}\subset H_{b,p.\left(s!\right)}^{0}\left(X\right)$
\prettyref{eq:Kodaira hyperplane}.
\end{defn}
Clearly, if the orthonormal basis $\left\{ \psi_{1},\ldots,\psi_{d_{p.\left(s!\right)}}\right\} $
of $\textrm{ker}\left(\Box_{b,p.\left(s!\right)}^{0}\right)=H_{b,p.\left(s!\right)}^{0}\left(X\right)$
is chosen so that $\psi_{j}\left(x_{0}\right)=0$, $1\leq j\leq d_{p.\left(s!\right)}-1$,
one has $S_{x_{0}}^{p}=\psi_{d_{p.\left(s!\right)}}$. From \prettyref{eq:Bergman kernel in eigenfunctions}
one then has 
\begin{align}
\left|S_{x_{0}}^{p}\left(x_{0}\right)\right|^{2} & =\Pi_{b,p.\left(s!\right)}\left(x_{0},x_{0}\right)\nonumber \\
S_{x_{0}}^{p}\left(x\right) & =\frac{1}{\Pi_{b,p.\left(s!\right)}\left(x_{0},x_{0}\right)}\Pi_{b,p.\left(s!\right)}\left(x,x_{0}\right).S_{x_{0}}^{p}\left(x_{0}\right).\label{eq:peak section in terms of Bergman kernel}
\end{align}

\begin{thm}
The augmented Kodaira map $\Psi_{p}$ \prettyref{eq:Kodaira map}
is injective for $p\gg0$.
\end{thm}
\begin{proof}
We assume to the contrary that there are two sequences of points $x_{p_{j}}^{1}$,
$x_{p_{j}}^{2}$ , $j=1,2,\ldots$, such that $p_{j}\rightarrow\infty$
as $j\rightarrow\infty$ and 
\begin{align}
x_{p_{j}}^{1} & \neq x_{p_{j}}^{2}\nonumber \\
\Psi_{p_{j}}\left(x_{p_{j}}^{1}\right) & =\Psi_{p_{j}}\left(x_{p_{j}}^{2}\right),\quad\forall j.\label{eq:injectivity failure}
\end{align}
By compactness we may further suppose $x_{p_{j}}^{1}\rightarrow x^{1}$,
$x_{p_{j}}^{2}\rightarrow x^{2}$ as $j\rightarrow\infty$. 

Case i: Suppose $e^{i\theta}x^{1}\neq x^{2}$, $\forall e^{i\theta}\in S^{1}$,
i.e. the limit points do not lie on the same $S^{1}$ orbit. The equation
\prettyref{eq:injectivity failure} in particular implies $\Phi_{p_{j}}\left(x_{p_{j}}^{1}\right)=\Phi_{p_{j}}\left(x_{p_{j}}^{2}\right)$
by definition \prettyref{eq:augmented Kodaira map}. Thus \prettyref{eq:peak section in terms of Bergman kernel}
implies 
\[
\Pi_{b,p_{j}.\left(s!\right)}\left(x_{p_{j}}^{1},x_{p_{j}}^{1}\right)\Pi_{b,p_{j}.\left(s!\right)}\left(x_{p_{j}}^{2},x_{p_{j}}^{2}\right)=\left|\Pi_{b,p_{j}.\left(s!\right)}\left(x_{p_{j}}^{1},x_{p_{j}}^{2}\right)\right|^{2}.
\]
The left hand side above is uniformly bounded below by $cp_{j}^{4/r}$
on account of \prettyref{eq:localized Szego kernel prelim}, \prettyref{eq:localization Szego kernel},
\prettyref{eq:uniform Bergman estimate}. While the right hand side
can be seen to be $O\left(p_{j}^{-\infty}\right)$ on choosing the
$S^{1}$ orbits of the BRT charts containing $x^{1}$, $x^{2}$ in
defining \prettyref{eq:localized Szego kernel prelim} to be disjoint.

Case ii: Suppose $e^{i\theta}x^{1}=x^{2},$ for some $e^{i\theta}\in S^{1}$.
We now again consider a BRT chart \prettyref{eq: BRT condition} of
the form $U=\left(-\frac{\pi}{s_{x^{1}}},\frac{\pi}{s_{x^{1}}}\right)\times B_{2\varrho}\left(x^{1}\right)$
containing the point $x^{1}$. As before this is obtained as the unit
circle bundle $S^{1}L\rightarrow Y$ over a hypersurface $Y\subset X$
containing the point $x^{1}$. For each $j$ we denote by $\left[x_{p_{j}}^{1}\right]$,
$\left[x_{p_{j}}^{2}\right]\in Y$ the unique points satisfying $x_{p_{j}}^{1}\in S^{1}.\left[x_{p_{j}}^{1}\right],\,x_{p_{j}}^{2}\in S^{1}\left[x_{p_{j}}^{2}\right].$
The inclusion further defines a local holomorphic map $\Phi_{p}:Y\rightarrow\mathbb{P}\left[H_{b,p.\left(s!\right)}^{0}\left(X\right)^{*}\right]$
satisfying $\Phi_{p_{j}}\left(\left[x_{p_{j}}^{1}\right]\right)=\Phi_{p_{j}}\left(\left[x_{p_{j}}^{2}\right]\right)$,
$j=1,2,\ldots$. Via the Noetherian property for analytic sets as
in \cite[Sec. 5.1]{Ma-Marinescu} this gives $\left[x_{p_{j}}^{1}\right]=\left[x_{p_{j}}^{2}\right]$
or $x_{p_{j}}^{2}=e^{i\theta}x_{p_{j}}^{1}\in S^{1}.x_{p_{j}}^{1}$
for $j\gg0$. Thus $x_{p_{j}}^{1},\,x_{p_{j}}^{2}$ lie on the same
orbit and with $k=s_{x_{p_{j}}^{1}}=s_{x_{p_{j}}^{2}}$ we have 
\begin{align*}
\Psi_{p}^{k}\left(x_{p_{j}}^{1}\right) & =\Psi_{p}^{k}\left(x_{p_{j}}^{2}\right)\\
\parallel\quad & \quad\parallel\\
\left(\Psi_{p}^{k,0}\left(x_{p_{j}}^{1}\right),\Psi_{p}^{k,1}\left(x_{p_{j}}^{1}\right)\right) & \left(e^{ip.k\theta}\Psi_{p}^{k,0}\left(x_{p_{j}}^{1}\right),e^{i\left(p+1\right).k\theta}\Psi_{p}^{k,1}\left(x_{p_{j}}^{1}\right)\right).
\end{align*}
It now follows that $\theta\in\frac{2\pi}{k}\mathbb{Z}$ implying
$x_{p_{j}}^{1}=x_{p_{j}}^{2}$ and contradicting \prettyref{eq:injectivity failure}.
\end{proof}
We note that following the closed range property \prettyref{eq: closed range property}
for $\bar{\partial}_{b}$ of \prettyref{subsec:Spectral-gap-and}
our embedding theorem \prettyref{thm: main embedding thm}, aside
from the equivariance, can be obtained from the main theorem of \cite{Christ89-embedding}.

\textbf{Acknowledgments.} The authors would like to thank J. Sjöstrand
and G. Marinescu for their guidance and several discussions regarding
Boutet de Monvel-Sjöstrand theory. The authors also thank professors
L. Lempert, D. Phong and T. Ohsawa for their insightful comments regarding
this work.

\bibliographystyle{siam}
\bibliography{biblio}
 
\end{document}